\documentclass[11pt]{amsart}

\usepackage{epigamath-voisin}

\usepackage[english]{babel}

\numberwithin{equation}{section}

\usepackage{amsmath}
\usepackage{amssymb}
\usepackage{amscd}
\usepackage{eucal}
\usepackage{upgreek}
\usepackage[makeroom]{cancel}
\usepackage[normalem]{ulem}
\usepackage{array}
\usepackage{verbatim}
\usepackage{mathtools}
\usepackage{stmaryrd}
\usepackage[shortlabels]{enumitem}

\usepackage{xy}
\xyoption{all}
\usepackage{tikz-cd}

\theoremstyle{plain}

\newtheorem{theorem}{Theorem}[section]

\newtheorem{lemma}[theorem]{Lemma}
\newtheorem{proposition}[theorem]{Proposition}
\newtheorem{corollary}[theorem]{Corollary}

\theoremstyle{definition}
\newtheorem{definition}[theorem]{Definition}

\theoremstyle{remark}
\newtheorem{remark}[theorem]{Remark}
\newtheorem{example}[theorem]{Example}

\newcommand{\nc}{\newcommand}

\nc{\md}{\operatorname{-}}

\renewcommand{\AA}{{\mathbb{A}}}
\nc{\CC}{{\mathbb{C}}}
\nc{\DD}{{\mathbb{D}}}
\nc{\LL}{{\mathbb{L}}}
\nc{\RR}{{\mathbb{R}}}
\renewcommand{\P}{{\mathbb{P}}}
\nc{\OO}{{\mathbb{O}}}

\nc{\QQ}{{\mathbb{Q}}}
\nc{\ZZ}{{\mathbb{Z}}}
\nc{\Z}{{\mathbb{Z}}}

\nc{\cA}{{\mathcal{A}}}
\nc{\cB}{{\mathcal{B}}}
\nc{\cC}{{\mathcal{C}}}
\nc{\cD}{{\mathcal{D}}}
\nc{\cE}{{\mathcal{E}}}
\nc{\cF}{{\mathcal{F}}}
\nc{\cG}{{\mathcal{G}}}
\nc{\cH}{{\mathcal{H}}}
\nc{\cI}{{\mathcal{I}}}
\nc{\cJ}{{\mathcal{J}}}
\nc{\cK}{{\mathcal{K}}}
\nc{\cL}{{\mathcal{L}}}
\nc{\cM}{{\mathcal{M}}}
\nc{\cN}{{\mathcal{N}}}
\nc{\cO}{{\mathcal{O}}}
\nc{\cP}{{\mathcal{P}}}
\nc{\cQ}{{\mathcal{Q}}}
\nc{\cR}{{\mathcal{R}}}
\nc{\cS}{{\mathcal{S}}}
\nc{\cT}{{\mathcal{T}}}
\nc{\cU}{{\mathcal{U}}}
\nc{\cV}{{\mathcal{V}}}
\nc{\cW}{{\mathcal{W}}}
\nc{\cX}{{\mathcal{X}}}
\nc{\cY}{{\mathcal{Y}}}
\nc{\cZ}{{\mathcal{Z}}}

\nc{\rc}{{\mathrm{c}}}
\nc{\rd}{{\mathrm{d}}}
\nc{\rf}{{\mathrm{f}}}
\nc{\rh}{{\mathrm{h}}}
\nc{\rs}{{\mathrm{s}}}
\nc{\rch}{{\mathrm{ch}}}
\nc{\rtd}{{\mathrm{td}}}

\nc{\rA}{{\mathrm{A}}}
\nc{\rB}{{\mathrm{B}}}
\nc{\rC}{{\mathrm{C}}}
\nc{\rD}{{\mathrm{D}}}
\nc{\rE}{{\mathrm{E}}}
\nc{\rF}{{\mathrm{F}}}
\nc{\rG}{{\mathrm{G}}}
\nc{\rH}{{\mathrm{H}}}
\nc{\rI}{{\mathrm{I}}}
\nc{\rJ}{{\mathrm{J}}}
\nc{\rK}{{\mathrm{K}}}
\nc{\rL}{{\mathrm{L}}}
\nc{\rM}{{\mathrm{M}}}
\nc{\rN}{{\mathrm{N}}}
\nc{\rO}{{\mathrm{O}}}
\nc{\rP}{{\mathrm{P}}}
\nc{\rQ}{{\mathrm{Q}}}
\nc{\rR}{{\mathrm{R}}}
\nc{\rS}{{\mathrm{S}}}
\nc{\rT}{{\mathrm{T}}}
\nc{\rU}{{\mathrm{U}}}
\nc{\rV}{{\mathrm{V}}}
\nc{\rW}{{\mathrm{W}}}
\nc{\rX}{{\mathrm{X}}}
\nc{\rY}{{\mathrm{Y}}}
\nc{\rZ}{{\mathrm{Z}}}

\nc{\bA}{{\mathbf{A}}}
\nc{\bB}{{\mathbf{B}}}
\nc{\bC}{{\mathbf{C}}}
\nc{\bD}{{\mathbf{D}}}
\nc{\bE}{{\mathbf{E}}}
\nc{\bF}{{\mathbf{F}}}
\nc{\bG}{{\mathbf{G}}}
\nc{\bH}{{\mathbf{H}}}
\nc{\bI}{{\mathbf{I}}}
\nc{\bJ}{{\mathbf{J}}}
\nc{\bK}{{\mathbf{K}}}
\nc{\bL}{{\mathbf{L}}}
\nc{\bM}{{\mathbf{M}}}
\nc{\bN}{{\mathbf{N}}}
\nc{\bO}{{\mathbf{O}}}
\nc{\bP}{{\mathbf{P}}}
\nc{\bQ}{{\mathbf{Q}}}
\nc{\bR}{{\mathbf{R}}}
\nc{\bS}{{\mathbf{S}}}
\nc{\bT}{{\mathbf{T}}}
\nc{\bU}{{\mathbf{U}}}
\nc{\bV}{{\mathbf{V}}}
\nc{\bW}{{\mathbf{W}}}
\nc{\bX}{{\mathbf{X}}}
\nc{\bY}{{\mathbf{Y}}}
\nc{\bZ}{{\mathbf{Z}}}

\nc{\ba}{{\mathbf{a}}}
\nc{\bb}{{\mathbf{b}}}
\nc{\bc}{{\mathbf{c}}}
\nc{\bd}{{\mathbf{d}}}
\nc{\be}{{\mathbf{e}}}
\nc{\bg}{{\mathbf{g}}}
\nc{\bh}{{\mathbf{h}}}
\nc{\bi}{{\mathbf{i}}}
\nc{\bj}{{\mathbf{j}}}
\nc{\bk}{{\mathbf{k}}}
\nc{\bl}{{\mathbf{l}}}
\nc{\bm}{{\mathbf{m}}}
\nc{\bn}{{\mathbf{n}}}
\nc{\bo}{{\mathbf{o}}}
\nc{\bp}{{\mathbf{p}}}
\nc{\bq}{{\mathbf{q}}}
\nc{\br}{{\mathbf{r}}}
\nc{\bs}{{\mathbf{s}}}
\nc{\bt}{{\mathbf{t}}}
\nc{\bu}{{\mathbf{u}}}
\nc{\bv}{{\mathbf{v}}}
\nc{\bw}{{\mathbf{w}}}
\nc{\bx}{{\mathbf{x}}}
\nc{\by}{{\mathbf{y}}}
\nc{\bz}{{\mathbf{z}}}

\nc{\fA}{{\mathfrak{A}}}
\nc{\fB}{{\mathfrak{B}}}
\nc{\fC}{{\mathfrak{C}}}
\nc{\fD}{{\mathfrak{D}}}
\nc{\fE}{{\mathfrak{E}}}
\nc{\fF}{{\mathfrak{F}}}
\nc{\fG}{{\mathfrak{G}}}
\nc{\fH}{{\mathfrak{H}}}
\nc{\fI}{{\mathfrak{I}}}
\nc{\fJ}{{\mathfrak{J}}}
\nc{\fK}{{\mathfrak{K}}}
\nc{\fL}{{\mathfrak{L}}}
\nc{\fM}{{\mathfrak{M}}}
\nc{\fN}{{\mathfrak{N}}}
\nc{\fO}{{\mathfrak{O}}}
\nc{\fP}{{\mathfrak{P}}}
\nc{\fQ}{{\mathfrak{Q}}}
\nc{\fR}{{\mathfrak{R}}}
\nc{\fS}{{\mathfrak{S}}}
\nc{\fT}{{\mathfrak{T}}}
\nc{\fU}{{\mathfrak{U}}}
\nc{\fV}{{\mathfrak{V}}}
\nc{\fW}{{\mathfrak{W}}}
\nc{\fX}{{\mathfrak{X}}}
\nc{\fY}{{\mathfrak{Y}}}
\nc{\fZ}{{\mathfrak{Z}}}

\nc{\fa}{{\mathfrak{a}}}
\nc{\fb}{{\mathfrak{b}}}
\nc{\fc}{{\mathfrak{c}}}
\nc{\fd}{{\mathfrak{d}}}
\nc{\fe}{{\mathfrak{e}}}
\nc{\ff}{{\mathfrak{f}}}
\nc{\fg}{{\mathfrak{g}}}
\nc{\fh}{{\mathfrak{h}}}
\nc{\fj}{{\mathfrak{j}}}
\nc{\fk}{{\mathfrak{k}}}
\nc{\fl}{{\mathfrak{l}}}
\nc{\fm}{{\mathfrak{m}}}
\nc{\fn}{{\mathfrak{n}}}
\nc{\fo}{{\mathfrak{o}}}
\nc{\fp}{{\mathfrak{p}}}
\nc{\fq}{{\mathfrak{q}}}
\nc{\fr}{{\mathfrak{r}}}
\nc{\fs}{{\mathfrak{s}}}
\nc{\ft}{{\mathfrak{t}}}
\nc{\fu}{{\mathfrak{u}}}
\nc{\fv}{{\mathfrak{v}}}
\nc{\fw}{{\mathfrak{w}}}
\nc{\fx}{{\mathfrak{x}}}
\nc{\fy}{{\mathfrak{y}}}
\nc{\fz}{{\mathfrak{z}}}

\nc{\sA}{{\mathsf{A}}}
\nc{\sB}{{\mathsf{B}}}
\nc{\sC}{{\mathsf{C}}}
\nc{\sD}{{\mathsf{D}}}
\nc{\sE}{{\mathsf{E}}}
\nc{\sF}{{\mathsf{F}}}
\nc{\sG}{{\mathsf{G}}}
\nc{\sH}{{\mathsf{H}}}
\nc{\sI}{{\mathsf{I}}}
\nc{\sJ}{{\mathsf{J}}}
\nc{\sK}{{\mathsf{K}}}
\nc{\sL}{{\mathsf{L}}}
\nc{\sM}{{\mathsf{M}}}
\nc{\sN}{{\mathsf{N}}}
\nc{\sO}{{\mathsf{O}}}
\nc{\sP}{{\mathsf{P}}}
\nc{\sQ}{{\mathsf{Q}}}
\nc{\sR}{{\mathsf{R}}}
\nc{\sS}{{\mathsf{S}}}
\nc{\sT}{{\mathsf{T}}}
\nc{\sU}{{\mathsf{U}}}
\nc{\sV}{{\mathsf{V}}}
\nc{\sW}{{\mathsf{W}}}
\nc{\sX}{{\mathsf{X}}}
\nc{\sY}{{\mathsf{Y}}}
\nc{\sZ}{{\mathsf{Z}}}

\nc{\sa}{{\mathsf{a}}}
\nc{\sd}{{\mathsf{d}}}
\nc{\se}{{\mathsf{e}}}
\nc{\sg}{{\mathsf{g}}}
\nc{\sh}{{\mathsf{h}}}
\nc{\si}{{\mathsf{i}}}
\nc{\sj}{{\mathsf{j}}}
\nc{\sk}{{\mathsf{k}}}
\nc{\sm}{{\mathsf{m}}}
\nc{\sn}{{\mathsf{n}}}
\nc{\so}{{\mathsf{o}}}
\nc{\sq}{{\mathsf{q}}}
\nc{\sr}{{\mathsf{r}}}
\nc{\st}{{\mathsf{t}}}
\nc{\su}{{\mathsf{u}}}
\nc{\sv}{{\mathsf{v}}}
\nc{\sw}{{\mathsf{w}}}
\nc{\sx}{{\mathsf{x}}}
\nc{\sy}{{\mathsf{y}}}
\nc{\sz}{{\mathsf{z}}}

\nc{\oA}{{\overline{A}}}
\nc{\oB}{{\overline{B}}}
\nc{\oC}{{\overline{C}}}
\nc{\oD}{{\overline{D}}}
\nc{\oE}{{\overline{E}}}
\nc{\oF}{{\overline{F}}}
\nc{\oG}{{\overline{G}}}
\nc{\oH}{{\overline{H}}}
\nc{\oI}{{\overline{I}}}
\nc{\oJ}{{\overline{J}}}
\nc{\oK}{{\overline{K}}}
\nc{\oL}{{\overline{L}}}
\nc{\oM}{{\overline{M}}}
\nc{\oN}{{\overline{N}}}
\nc{\oO}{{\overline{O}}}
\nc{\oP}{{\overline{P}}}
\nc{\oQ}{{\overline{Q}}}
\nc{\oR}{{\overline{R}}}
\nc{\oS}{{\overline{S}}}
\nc{\oT}{{\overline{T}}}
\nc{\oU}{{\overline{U}}}
\nc{\oV}{{\overline{V}}}
\nc{\oW}{{\overline{W}}}
\nc{\oX}{{\overline{X}}}
\nc{\oY}{{\overline{Y}}}
\nc{\oZ}{{\overline{Z}}}

\nc{\oa}{{\overline{a}}}
\nc{\ob}{{\overline{b}}}
\nc{\oc}{{\overline{c}}}
\nc{\od}{{\overline{d}}}
\nc{\of}{{\overline{f}}}
\nc{\og}{{\overline{g}}}
\nc{\oh}{{\overline{h}}}
\nc{\oi}{{\overline{i}}}
\nc{\oj}{{\overline{j}}}
\nc{\ok}{{\overline{k}}}
\nc{\ol}{{\overline{l}}}
\nc{\om}{{\overline{m}}}
\nc{\on}{{\overline{n}}}
\nc{\oo}{{\overline{o}}}
\nc{\op}{{\overline{p}}}
\nc{\oq}{{\overline{q}}}
\nc{\os}{{\overline{s}}}
\nc{\ot}{{\overline{t}}}
\nc{\ou}{{\overline{u}}}
\nc{\ov}{{\overline{v}}}
\nc{\ow}{{\overline{w}}}
\nc{\ox}{{\overline{x}}}
\nc{\oy}{{\overline{y}}}
\nc{\oz}{{\overline{z}}}

\nc{\tA}{{\tilde{A}}}
\nc{\tB}{{\tilde{B}}}
\nc{\tC}{{\tilde{C}}}
\nc{\tD}{{\tilde{D}}}
\nc{\tE}{{\tilde{E}}}
\nc{\tF}{{\tilde{F}}}
\nc{\tG}{{\tilde{G}}}
\nc{\tH}{{\tilde{H}}}
\nc{\tI}{{\tilde{I}}}
\nc{\tJ}{{\tilde{J}}}
\nc{\tK}{{\tilde{K}}}
\nc{\tL}{{\tilde{L}}}
\nc{\tM}{{\tilde{M}}}
\nc{\tN}{{\tilde{N}}}
\nc{\tO}{{\tilde{O}}}
\nc{\tP}{{\tilde{P}}}
\nc{\tQ}{{\tilde{Q}}}
\nc{\tR}{{\tilde{R}}}
\nc{\tS}{{\tilde{S}}}
\nc{\tT}{{\tilde{T}}}
\nc{\tU}{{\tilde{U}}}
\nc{\tV}{{\tilde{V}}}
\nc{\tW}{{\tilde{W}}}
\nc{\tX}{{\tilde{X}}}
\nc{\tY}{{\tilde{Y}}}
\nc{\tZ}{{\tilde{Z}}}

\nc{\tfD}{{\tilde{\fD}}}
\nc{\tcA}{{\tilde{\cA}}}
\nc{\tcB}{{\tilde{\cB}}}
\nc{\tcC}{{\tilde{\cC}}}
\nc{\tcD}{{\tilde{\cD}}}
\nc{\tcE}{{\tilde{\cE}}}
\nc{\tcF}{{\tilde{\cF}}}
\nc{\tcM}{{\tilde{\cM}}}
\nc{\tcP}{{\tilde{\cP}}}
\nc{\tcT}{{\tilde{\cT}}}

\nc{\tphi}{{\tilde{\varphi}}}

\nc{\ta}{{\tilde{a}}}
\nc{\tb}{{\tilde{b}}}
\nc{\tc}{{\tilde{c}}}
\nc{\td}{{\tilde{d}}}
\nc{\te}{{\tilde{e}}}
\nc{\tf}{{\tilde{f}}}
\nc{\tg}{{\tilde{g}}}
\nc{\ti}{{\tilde{\imath}}}
\nc{\tj}{{\tilde{j}}}
\nc{\tk}{{\tilde{k}}}
\nc{\tl}{{\tilde{l}}}
\nc{\tm}{{\tilde{m}}}
\nc{\tn}{{\tilde{n}}}
\nc{\tp}{{\tilde{p}}}
\nc{\tq}{{\tilde{q}}}
\nc{\tr}{{\tilde{r}}}
\nc{\ts}{{\tilde{s}}}
\nc{\tu}{{\tilde{u}}}
\nc{\tv}{{\tilde{v}}}
\nc{\tw}{{\tilde{w}}}
\nc{\tx}{{\tilde{x}}}
\nc{\ty}{{\tilde{y}}}
\nc{\tz}{{\tilde{z}}}

\nc{\hA}{{\hat{A}}}
\nc{\hB}{{\hat{B}}}
\nc{\hC}{{\hat{C}}}
\nc{\hD}{{\hat{D}}}
\nc{\hE}{{\hat{E}}}
\nc{\hF}{{\hat{F}}}
\nc{\hG}{{\hat{G}}}
\nc{\hH}{{\hat{H}}}
\nc{\hI}{{\hat{I}}}
\nc{\hJ}{{\hat{J}}}
\nc{\hK}{{\hat{K}}}
\nc{\hL}{{\hat{L}}}
\nc{\hM}{{\hat{M}}}
\nc{\hN}{{\hat{N}}}
\nc{\hO}{{\hat{O}}}
\nc{\hP}{{\hat{P}}}
\nc{\hQ}{{\hat{Q}}}
\nc{\hR}{{\hat{R}}}
\nc{\hS}{{\hat{S}}}
\nc{\hT}{{\hat{T}}}
\nc{\hU}{{\hat{U}}}
\nc{\hV}{{\hat{V}}}
\nc{\hW}{{\hat{W}}}
\nc{\hX}{{\widehat{X}}}
\nc{\hY}{{\hat{Y}}}
\nc{\hZ}{{\hat{Z}}}

\nc{\ha}{{\hat{a}}}
\nc{\hb}{{\hat{b}}}
\nc{\hc}{{\hat{c}}}
\nc{\hd}{{\hat{d}}}
\nc{\he}{{\hat{e}}}
\nc{\hg}{{\hat{g}}}
\nc{\hh}{{\hat{h}}}
\nc{\hi}{{\hat{i}}}
\nc{\hj}{{\hat{j}}}
\nc{\hk}{{\hat{k}}}
\nc{\hl}{{\hat{l}}}
\nc{\hm}{{\hat{m}}}
\nc{\hn}{{\hat{n}}}
\nc{\ho}{{\hat{o}}}
\nc{\hp}{{\hat{p}}}
\nc{\hq}{{\hat{q}}}
\nc{\hr}{{\hat{r}}}
\nc{\hs}{{\hat{s}}}
\nc{\hu}{{\hat{u}}}
\nc{\hv}{{\hat{v}}}
\nc{\hw}{{\hat{w}}}
\nc{\hx}{{\hat{x}}}
\nc{\hy}{{\hat{y}}}
\nc{\hz}{{\hat{z}}}

\nc{\hcC}{{\widehat{\cC}}}
\nc{\hcT}{{\widehat{\cT}}}

\nc{\eps}{\upepsilon}
\nc{\lan}{\big\langle}
\nc{\ran}{\big\rangle}
\nc{\kk}{{\Bbbk}}
\nc{\io}{\upiota}
\nc{\Kr}{\mathsf{Kr}}
\nc{\cKr}{\mathcal{K}\!\mathit{r}}

\nc{\Dm}{\bD^{-}}
\nc{\Db}{\bD^{\mathrm{b}}}
\nc{\Dbc}{\bD^{\mathrm{b}}_{\mathrm{c}}}
\nc{\Dp}{\bD^{\mathrm{perf}}}
\nc{\Dperf}{\bD^{\mathrm{perf}}}
\nc{\Dqc}{\bD_{\mathrm{qc}}}
\nc{\Du}{\bD}
\nc{\Dsing}{\bD^{\mathrm{sg}}}
\nc{\Dg}{\bD^{\mathrm{sg}}}

\def\ol{\overline}

\newcommand{\hf}{{\mathrm{hf}}}

\newcommand{\sing}{{\mathrm{sg}}}
\newcommand{\opp}{\mathrm{op}}

\nc{\Rn}{\rR_{\mathrm{node}}}
\nc{\Cn}{\cC_{\mathrm{node}}}
\nc{\Dfd}[1]{\bD_{\mathrm{fd}}(#1)}

\def\bw#1#2{\textstyle{\bigwedge\hskip-0.9mm^{#1}}\hskip0.2mm{#2}}

\nc{\xrightiso}[1]{ \xrightarrow[{\ \raisebox{0.5ex}[0ex][0ex]{$\sim$}\ }]{#1} }

\nc{\thick}{\mathbf{thick}}

\DeclareMathOperator{\Hom}{\mathrm{Hom}}

\DeclareMathOperator{\Ext}{\mathrm{Ext}}
\DeclareMathOperator{\cExt}{\mathcal{E}\!\mathit{xt}}

\DeclareMathOperator{\RHom}{\mathrm{RHom}}
\DeclareMathOperator{\cRHom}{\mathrm{R}\mathcal{H}\mathit{om}}

\DeclareMathOperator{\Spec}{\mathrm{Spec}}

\DeclareMathOperator{\Bl}{\mathrm{Bl}}

\DeclareMathOperator{\Pic}{\mathrm{Pic}}
\DeclareMathOperator{\Cl}{\mathrm{Cl}}
\DeclareMathOperator{\Br}{\mathrm{Br}}

\DeclareMathOperator{\Ker}{\mathrm{Ker}}

\DeclareMathOperator{\Ima}{\mathrm{Im}}
\DeclareMathOperator{\Cone}{\mathrm{Cone}}

\DeclareMathOperator{\pr}{\mathrm{pr}}

\DeclareMathOperator{\Gr}{\mathrm{Gr}}

\DeclareMathOperator{\Fl}{\mathrm{Fl}}

\DeclareMathOperator{\id}{\mathrm{id}}
\DeclareMathOperator{\rank}{\mathrm{rk}}

\DeclareMathOperator{\colim}{\mathrm{colim}}
\DeclareMathOperator{\hocolim}{\mathrm{hocolim}}
\DeclareMathOperator{\holim}{\mathrm{holim}}

\def\Pinfty#1{\P^{\infty,{#1}}}

\newenvironment{renumerate}{\begin{enumerate}[label={\textup{(\roman*)}}]}{\end{enumerate}}
\newenvironment{aenumerate}{\begin{enumerate}[label={\textup{(\alph*)}}]}{\end{enumerate}}

\DeclareRobustCommand\longhookrightarrow
    {\lhook\joinrel\longrightarrow}
\DeclareRobustCommand\longtwoheadrightarrow
     {\relbar\joinrel\twoheadrightarrow}
     \newcommand{\lra}{\longrightarrow}

\newcommand{\supth}[1]{\ensuremath{#1^{\mathrm{th}}}}

\YearArticle{2023} \EpigaArticleNr{12} \ReceivedOn{January 20, 2023}
\InFinalFormOn{August 3, 2023}
\AcceptedOn{September 25, 2023}

\title{Categorical absorptions of singularities and degenerations}
\titlemark{Categorical absorptions of singularities and degenerations}

\author{Alexander Kuznetsov}
\address{Algebraic Geometry Section, Steklov Mathematical Institute of Russian Academy of Sciences, 8 Gubkin str., Moscow 119991, Russia \\
  Laboratory of Algebraic Geometry, HSE, 6 Usacheva Str., Moscow 119048, Russia}
\email{akuznet@mi-ras.ru}
\author{Evgeny Shinder}
\address{School of Mathematics and Statistics, University of Sheffield, Hounsfield Road, S3 7RH, UK \\
Hausdorff Center for Mathematics at the University of Bonn, Endenicher Allee 60, 53115 Bonn, Germany}
\email{eugene.shinder@gmail.com}

\authormark{A.~Kuznetsov and E.~Shinder}

\AbstractInEnglish{We introduce the notion of categorical absorption
  of singularities: an operation that removes from the derived
  category of a singular variety a small admissible subcategory
  responsible for singularity and leaves a smooth and proper category.
  We construct (under appropriate assumptions) a categorical
  absorption for a projective variety~$X$ with isolated ordinary
  double points.  We further show that for any smoothing~$\cX/B$
  of~$X$ over a smooth curve~$B$, the smooth part of the derived
  category of~$X$ extends to a smooth and proper over~$B$ family of
  triangulated subcategories in the fibers of~$\cX$.}

\MSCclass{14F08, 14D06, 14E15, 14J45}

\KeyWords{Absorption of singularities, adherence,
crepant categorical resolution of singularities, derived
category, ordinary double point, p-infinity object,
semiorthogonal decomposition}

\acknowledgement{A.\,K. was partially supported by the HSE University
  Basic Research Program.  E.\,S. was partially supported by the EPSRC
  grant EP/T019379/1 ``Derived categories and algebraic K-theory of
  singularities'', and by the ERC Synergy grant ``Modern Aspects of
  Geometry: Categories, Cycles and Cohomology of Hyperk\"ahler
  Varieties''.}

\dedication{To Claire Voisin, with admiration}

\begin{document}


\maketitle

\begin{prelims}

\DisplayAbstractInEnglish

\bigskip

\DisplayKeyWords

\medskip

\DisplayMSCclass

\end{prelims}


\newpage

\setcounter{tocdepth}{1}

\tableofcontents


\section{Introduction}

Resolution of singularities is a very important instrument in algebraic geometry, that allows one to replace complicated geometry of singular schemes by much more tractable geometry of smooth varieties.  Its categorical version -- categorical resolution of singularities -- has been defined in~\cite{K08,KL}.

Let~$X$ be a singular proper scheme over a field~$\kk$.  One associates with it the pair~$(\Dp(X),\Db(X))$ of triangulated categories, where~$\Dp(X)$ is the category of perfect complexes on~$X$ (this category is proper but not smooth) and~$\Db(X)$ is the bounded derived category of coherent sheaves on~$X$ (this category is smooth but not proper).  Categorical resolution allows one to replace this pair with a single \emph{smooth and proper} triangulated category~$\cD$.

Such a category~$\cD$ is ``larger'' than~$\Db(X)$ and~$\Dp(X)$ in the sense that~$\Db(X)$ is (expected to be) a localization of~$\cD$ and~$\Dp(X)$ is a subcategory of~$\cD$.  In this paper we look at the problem from a different angle suggesting to replace (when possible) the categories~$\Db(X)$ and~$\Dp(X)$ by a ``smaller'' smooth and proper triangulated category in the following sense.

\begin{definition}
\label{def:intro-absorption}
We say that a triangulated subcategory~$\cP \subset \Db(X)$ \emph{absorbs singularities of~$X$} if~$\cP$ is admissible in~$\Db(X)$ and both orthogonals~$\cP^\perp$ and~${}^\perp\cP$ are smooth and proper.
\end{definition}

Note that since we assume $\cP$ to be admissible, we have an equivalence~$\cP^\perp \simeq {}^\perp \cP$ induced by the mutation functors, so 
in the definition of absorption, it suffices to assume that either of the orthogonals is smooth and proper.  Furthermore, if~$\cP$ absorbs singularities of a Gorenstein scheme~$X$, it is easy to check (see Lemma~\ref{lemma:absorption-hocolim}) that the smooth and proper category~$\cD \coloneqq {}^\perp\cP$ (which is equivalent to~$\cP^\perp$) is contained in~$\Dp(X)$ and admissible in both~$\Dp(X)$ and~$\Db(X)$; in this sense the operation replacing~$\Db(X)$ and~$\Dp(X)$ by~$\cD$ is ``opposite'' to the categorical resolution operation.

There is a trivial example~$\cP = \Db(X)$, so that~$\cD = 0$, which of course is not interesting.  Therefore, the idea is to make~$\cP$ as small as possible, so that~$\cD$ maximally reflects the geometry of~$X$.

The following geometric example of absorption illustrates the idea.  Let $\upsigma \colon X \to Y$ be an ``antiresolution'': a proper morphism, where~$X$ is singular, $Y$ is smooth, and both~$X$ and~$Y$ are proper, and assume~\mbox{$\upsigma_*(\cO_X) {}\cong{} \cO_Y$}, where~$\upsigma_* \colon \Db(X) \to \Db(Y)$ is the derived pushforward functor.  Let
\begin{equation*}
\cP \coloneqq \Ker(\upsigma_*) \subset \Db(X).
\end{equation*}
We claim that~$\cP$ absorbs the singularities of~$X$.  Indeed, the two adjoint functors~$\upsigma^* \colon \Db(Y) \to \Db(X)$ and~$\upsigma^! \colon \Db(Y) \to \Db(X)$ of~$\upsigma_*$ are fully faithful and provide semiorthogonal decompositions
\begin{equation*}
\Db(X) = \langle \cP, \upsigma^*(\Db(Y)) \rangle = \langle \upsigma^!(\Db(Y)), \cP \rangle;
\end{equation*}
in particular, the category~$\cP^\perp \simeq \Db(Y) \simeq {}^\perp\cP$ is smooth and proper.

\begin{example}
\label{ex:geometric}
Let us further consider two special cases of this situation:
\begin{aenumerate}
\item 
\label{ex:blowup}
$Y$ is a smooth and proper surface, $Z = \Spec(\kk[\eps_0] / \eps_0^2) \subset Y$, $X = \Bl_Z(Y)$, 
and~$\upsigma \colon X \to Y$ is the blowup morphism.
Let~$E \cong \P^1 \subset X$ be the exceptional locus taken with the reduced scheme structure
(notice that the exceptional divisor~$2E$ is Cartier, while~$E$ is only a Weil divisor).
\item 
\label{ex:gluing}
$Y = \P^1$, $X = Y \cup_{y_0} E$, where~$E \cong \P^1$ intersects~$Y$ transversely at a point~$y_0$, 
and~$\upsigma \colon X \to Y$ is the contraction of~$E$ to~$\{y_0\}$,
so that~$E$ is again, in a sense, the exceptional locus of~$\upsigma$.
\end{aenumerate}
\end{example} 

A simple computation shows that in both cases~$\cP = \langle \cO_E(-1) \rangle$, 
but while in case~\ref{ex:blowup} we have
\begin{equation}
\label{eq:theta-1-eps-0}
\cP \simeq \Db\left(\kk[\eps_0]/\eps_0^2\right) \simeq \Dp(\kk[\uptheta_1])\hphantom{,}
\end{equation}
by the blowup formula (note that~$Z$ is a local complete intersection),
in case~\ref{ex:gluing} we have
\begin{equation}
\label{eq:theta-2-eps-1}
\cP \simeq \Db\left(\kk[\eps_1]/\eps_1^2\right) \simeq \Dp(\kk[\uptheta_2]),
\end{equation}
where~$\eps_p$ is a variable of degree~$-p$, $\uptheta_q$ is a variable of degree~$q$, and the second equivalences in~\eqref{eq:theta-1-eps-0} and~\eqref{eq:theta-2-eps-1} are given by the Koszul duality (see Proposition~\ref{prop:sap-sbq} for details).  The difference in the structure of~$\cP$ in these two cases leads to a completely different behaviour with respect to smoothings of~$X$, where we use the following.

\begin{definition}
\label{def:intro-smoothing}
We say that a flat projective morphism~$f \colon \cX \to B$
to a smooth pointed curve $(B,o)$ is a \emph{smoothing} of~$X$ if:
\begin{itemize}
\item
the central {scheme fiber of~$f$} is isomorphic to~$X$; \textit{i.e.}, $\cX_o \cong X$;
\item
the morphism $f$ is smooth away from the central fiber; and
\item
the total space~$\cX$ is smooth.
\end{itemize}
\end{definition}

In the case of Example~\ref{ex:geometric}\ref{ex:blowup}, a smoothing of~$X$ can be obtained as follows.  Let~$\cY \to B$ be any smooth deformation of~$Y$ (\textit{e.g.}, $\cY = Y \times B$), let~$\cZ \to B$ be a double covering ramified over the central point~$o \in B$ (so that the central fiber~$\cZ_o$ is isomorphic to~$Z$) with smooth~$\cZ$, and let~$\cZ \hookrightarrow \cY$ be a closed embedding over~$B$ (coinciding with the embedding~$Z \hookrightarrow Y$ over the point~$o$).  Then~$\cX \coloneqq \Bl_\cZ(\cY)$ is a smoothing of~$X$.  The blowup formula gives in this case a $B$-linear semiorthogonal decomposition
\begin{align*}
\Db(\cX) &= \langle \Db(\cZ), \Db(\cY) \rangle,
\shortintertext{and for each point~$b \ne o$ in~$B$ we obtain by base change (see~\cite{K11}) a semiorthogonal decomposition}
\Db(\cX_b) &= \langle \Db(\cZ_b), \Db(\cY_b) \rangle.
\end{align*}
So, in this case we see that both components~$\cP = \Db(Z)$ and~${}^\perp\cP = \Db(Y)$ of~$\Db(X)$ deform simultaneously to a pair of smooth and proper categories; thus, in a sense, the smoothing of~$X$ is achieved through a smoothing of the singular component~$\cP$.

In the case of Example~\ref{ex:geometric}\ref{ex:gluing}. the situation is quite different.  In this case~$X$ is a singular conic, so any smoothing of~$X$ is a conic bundle~$\cX \to B$ with~$X = \cX_o$ being the only singular fiber.  Note that~$E \subset \cX$ is a smooth rational curve (one of the components of the central fiber), and a standard computation shows that its normal bundle is~$\cO_E(-1)$.  Therefore, the curve~$E {} \subset \cX$ can be contracted; \textit{i.e.}, there is a smooth surface~$\cY$ with a point~$y_0 \in \cY$ such that~$\cX = \Bl_{y_0}(\cY)$ with~$E$ being the exceptional divisor.  Then~$\cY \to B$ is a smooth $\P^1$-fibration (with~$\cY_o \cong Y$), and the blowup formula gives a $B$-linear semiorthogonal decomposition
\begin{equation}
\label{eq:ex-gluing}
\Db(\cX) = \langle \cO_E(-1), \Db(\cY) \rangle.
\end{equation}
Its first component is generated by the exceptional sheaf~$\cO_E(-1)$ supported in the central fiber of the conic bundle; therefore, it disappears after base change to any point~$b \ne o$ in~$B$, and we obtain an equivalence
\begin{equation*}
\Db(\cX_b) \simeq \Db(\cY_b).
\end{equation*}
So, in this case the component~$\cP$ does not deform away from the central fiber; thus, a smoothing of~$X$ is achieved by dropping ``the singular part''~$\cP$ of~$\Db(X)$ and deforming its ``smooth part''~${}^\perp\cP {} \simeq \Db(\cY_o)$.  In Definition~\ref{def:intro-deformation-absorption} we axiomatize this remarkable situation.

First, we recall some notation.  For a set of objects~$S$ in a triangulated category~$\cT$, we denote by~$\langle S \rangle \subset \cT$ the smallest triangulated and by~$\thick(S) \subset \cT$ the smallest thick (\textit{i.e.}, triangulated and closed under direct summands) subcategory of~$\cT$ containing~$S$.

\begin{definition}
\label{def:intro-deformation-absorption}
Let~$f \colon \cX \to B$ be a smoothing of~$X$, and let~$\io \colon X \hookrightarrow \cX$ be the embedding of the central fiber.  We say that a triangulated subcategory~$\cP \subset \Db(X)$ \emph{provides a deformation absorption of singularities} (with respect to the smoothing~$\cX$\,) if~$\cP$ absorbs singularities of~$X$ and the triangulated subcategory~$\langle \io_*\cP \rangle \subset \Db(\cX)$ generated in~$\Db(\cX)$ by the pushforwards of objects of~$\cP$ is admissible.

Similarly, we say that~$\cP \subset \Db(X)$ \emph{provides a thick deformation absorption} (with respect to the smoothing~$\cX$) if~$\cP$ absorbs singularities of~$X$ and the triangulated subcategory~$\thick(\io_*\cP) \subset \Db(\cX)$ is admissible.

We say that~$\cP$ \emph{provides a universal (thick) deformation absorption of singularities of~$X$} if the corresponding property holds for any smoothing of~$X$.
\end{definition}

If~$\cP$ provides a deformation absorption, so that~$\langle \io_*\cP \rangle \subset \Db(\cX)$ is admissible, then~$\thick(\io_*\cP) = \langle \io_*\cP \rangle$, so we will uniformly use the notation~$\thick(\io_*\cP)$ for the resulting admissible subcategory of~$\Db(\cX)$.

To clarify the definition, note that the category~$\thick(\io_*\cP)$ is contained in~$\Db_X(\cX)$, the subcategory of~$\Db(\cX)$ formed by objects with set-theoretical support on~$X$.  Moreover, by~\cite[Theorem~1.1]{KS:hfd} the category~$\thick(\io_*\cP)$ is always admissible in~$\Db_X(\cX)$.  Thus, the deformation absorption property amounts to the assumption that~$\thick(\io_*\cP)$ keeps being admissible in the larger category~$\Db(\cX)$.

Using the base change technique, \textit{cf.}~\cite{K11}, it is easy to see that Definition~\ref{def:intro-deformation-absorption} leads to the following result.

\begin{theorem}
\label{thm:intro-deformation-absorption-to-smoothing}
Let~$X$ be a projective variety, and let~$\cP \subset \Db(X)$ provide a \textup(thick\textup) deformation absorption of singularities of\, $X$ with respect to a smoothing~$f \colon \cX \to B$.  Then the subcategory
\begin{equation*}
\cD \coloneqq {}^\perp (\thick(\io_*\cP)) \subset \Db(\cX)
\end{equation*}
is $B$-linear, and there is a $B$-linear semiorthogonal decomposition
\begin{equation*}
\Db(\cX) = \langle \thick(\io_*\cP), \cD \rangle.
\end{equation*}
Moreover, the central fiber~$\cD_o$ of\,~$\cD$ 
can be described as
\begin{equation}
\cD_o \simeq {}^\perp \cP \subset \Db(X),
\end{equation}
and if\, $b \ne o$, then~$\cD_b \simeq \Db(\cX_b)$.  In particular, $\cD$ is smooth and proper over~$B$.
\end{theorem}

\begin{remark}
\label{rem:contraction}
In the situation of Example~\ref{ex:geometric}\ref{ex:gluing}, if a smoothing of~$X$ is given by a conic bundle~$\cX/B$ obtained from a $\P^1$-bundle~$\cY/B$ as~$\cX = \Bl_{y_0}(\cY)$ with exceptional divisor~$E$, then the object~$\io_*\cO_E(-1)$ is exceptional on $\cX$, so the subcategory it generates in~$\Db(\cX)$ is admissible.  This means that the subcategory~\mbox{$\cP = \langle \cO_E(-1) \rangle \subset \Db(X)$} is an example of a subcategory providing a universal deformation absorption of singularities.  Furthermore, in this case the semiorthogonal decomposition~\eqref{eq:ex-gluing} implies that for the subcategory~$\cD \subset \Db(\cX)$ defined in Theorem~\ref{thm:intro-deformation-absorption-to-smoothing}, we have~$\cD \simeq \Db(\cY)$.
\end{remark}

Next, we observe that the deformation absorption phenomenon of Example~\ref{ex:geometric}\ref{ex:gluing} is not specific to the concrete geometric situation, but rather a consequence of the equivalence~\eqref{eq:theta-2-eps-1}.  To state this result we introduce the following definition.

\begin{definition}
\label{def:pinfty-geometric}
Let~$q \ge 1$.  An object~$\rP \in \Db(X)$ is called a~\emph{$\Pinfty{q}$-object} if
\begin{equation*}
\Ext^\bullet(\rP,\rP) \cong \kk[\uptheta], 
\end{equation*}
where~$\deg(\uptheta) = q$.
When we do not want to specify the parameter~$q$, we just say that~$\rP$ is a~$\P^\infty$-object.
\end{definition}

Since~$\Ext^\bullet(\rP, \rP)$ is infinite-dimensional, $\P^\infty$-objects can only exist on singular varieties.  In the body of the paper we define~$\Pinfty{q}$-objects in the general situation of arbitrary triangulated categories, and then the definition becomes a bit more technical (see Definition~\ref{def:pinfty}), but in the geometric situation it is equivalent to the one stated above (see Remark~\ref{rem:pinfty}).  Our main emphasis is on~$\Pinfty{1}$- and~$\Pinfty{2}$-objects because, as we will see, they are geometrically meaningful. Moreover, we prove that~$\Pinfty{q}$-objects with~$q \ne 1, 2$ do not exist on smoothable varieties; see~Corollary~\ref{cor:no-pinfty-big}.

As~\eqref{eq:theta-2-eps-1} shows, the situation of Example~\ref{ex:geometric}\ref{ex:gluing} corresponds to a~$\Pinfty{2}$-object.
In this case we have the following result, which we formulate in a slightly more general situation.

\begin{theorem}
\label{thm:intro-pinfty2}
Let~$X$ be a projective variety.
Let~$(\rP_1,\dots,\rP_r)$ be a semiorthogonal collection of\,~$\Pinfty{2}$-objects in~$\Db(X)$ such that the subcategory
\begin{equation*}
\cP \coloneqq \langle \rP_1, \dots, \rP_r \rangle
\end{equation*}
absorbs singularities of\,~$X$.  Then for any smoothing~$\cX \to B$ of\,~$X$ the collection~$\io_*\rP_1,\dots,\io_*\rP_r$ in~$\Db(\cX)$ is exceptional.  In particular, $\cP$ provides a universal deformation absorption of singularities of\, $X$.
\end{theorem}

If we combine Theorem~\ref{thm:intro-pinfty2} with Theorem~\ref{thm:intro-deformation-absorption-to-smoothing}, we obtain a smooth and proper over~$B$ category~$\cD$, and following the analogy of Remark~\ref{rem:contraction}, we consider passing from~$\Db(\cX)$ to~$\cD$ as a categorical incarnation of the contraction~$\cX \to \cY$.

We have a similar generalization of the situation of Example~\ref{ex:geometric}\ref{ex:blowup},
which by~\eqref{eq:theta-1-eps-0} corresponds to a~$\Pinfty{1}$-object.
Again, we state the result in a slightly more general situation.

\begin{theorem}
\label{thm:intro-pinfty1}
Let~$X$ be a projective variety.
Let~$(\rP_1,\dots,\rP_r)$ be a semiorthogonal collection 
of\,~$\Pinfty{1}$-objects in~$\Db(X)$ such that the subcategory
\begin{equation*}
\cP \coloneqq \langle \rP_1, \dots, \rP_r \rangle
\end{equation*}
absorbs singularities of\, $X$.  Then for any smoothing~$\cX \to B$ of\, $X$ after an \'etale base change, there are~$r$ double coverings~$\cZ_i \to B$, $1 \le i \le r$, \'etale over~$B \setminus \{o\}$, and a $B$-linear semiorthogonal decomposition
\begin{equation*}
\Db(\cX) = \langle \Db(\cZ_1), \dots, \Db(\cZ_r), \cD \rangle,
\end{equation*}
where~$\cD$ is smooth and proper over~$B$, $\cD_o = {}^\perp\cP \subset \Db(X)$, and~$\Db(\cZ_i)_o = \langle \rP_i \rangle$.  Moreover, a semiorthogonal decomposition of\,~$\Db(\cX)$ with these properties is unique Zariski locally around~$o \in B$.
\end{theorem}

In the paper, besides proving the above results (see~Section~\ref{ss:abs-dabs-proofs}), we discuss possible approaches to absorbing singularities by~$\Pinfty{q}$-objects for~$q \in \{1,2\}$ and obtaining smooth and proper families of categories out of them.  One possibility here is to use the intrinsic structure of the category~$\langle \rP \rangle$ generated by a~$\Pinfty{q}$-object; we call such a category the \emph{categorical ordinary double point} and discuss its properties in detail in~Section~\ref{sec:categorical-odp}.

Another approach is to start with a resolution of singularities~$\pi \colon \tX \to X$ and find appropriate admissible subcategory~$\tcP \subset \Db(\tX)$ such that~$\Ker(\pi_*) \subset \tcP$ and the Verdier localization
\begin{equation*}
\cP \coloneqq \tcP / \Ker(\pi_*) \subset \Db(X)
\end{equation*}
is equivalent to a categorical ordinary double point (or has a semiorthogonal decomposition into several categorical ordinary double points).  Then the category~$\cP^\perp \simeq \tcP^\perp$ (see Theorem~\ref{thm:contractions}) is equivalent to an admissible subcategory in the derived category of a smooth and proper variety~$\tX$, so it is smooth and proper, and hence~$\cP$ absorbs singularities of~$X$.  This approach works well when~$\Ker(\pi_*)$ is generated by a \emph{spherical object}~$\rK \in \Db(\tX)$ and~$\tcP$ is generated by an exceptional pair~$(\cE,\cE')$ such that
\begin{equation*}
\dim \Ext^\bullet(\cE,\rK) = 1
\end{equation*}
(we say in this case that~$\cE$ is \emph{adherent} to~$\rK$; see Definition~\ref{def:adherence}; this generalizes the notion introduced in~\cite{KKS20}) and~$\cE'$ is obtained from~$\cE$ by the spherical twist functor associated with~$\rK$.  We check in Lemma~\ref{lem:kronecker-recogn} that in this case the pair~$(\cE,\cE')$ is indeed exceptional with
\begin{equation*}
\Ext^\bullet(\cE,\cE') \cong \kk \oplus \kk[-q],
\qquad 
q \ge 1.
\end{equation*}
We call the category generated by such a pair the \emph{graded Kronecker quiver category} and discuss its properties in~Section~\ref{sec:kronecker}.
The main result here is Proposition~\ref{prop:contr-Cp}, where we show that
the Verdier localization of the graded Kronecker quiver category 
by the spherical object~$\rK$ is equivalent 
to the categorical ordinary double point.

However, if~$\pi \colon \tX \to X$ is a resolution, the condition that~$\Ker(\pi_*)$ is generated by spherical objects is very restrictive, and for most singularities such resolutions do not exist.  So, to circumvent this problem, we replace a geometric resolution with a \emph{categorical resolution of singularities} in the sense of~\cite{K08,KL}.  In other words, we consider a smooth and proper triangulated category~$\tcT$ and a functor
\begin{equation*}
\pi_* \colon \tcT \lra \Db(X)
\end{equation*}
that behaves as the derived pushforward for a geometric resolution.
There are several ways to make this assumption precise: 
one is to assume that~$\pi_*$ is a Verdier localization, 
but we prefer a slightly weaker assumption that was introduced by Efimov in~\cite{Ef20}.

\begin{definition}
\label{def:cc}
We say that a functor~$\pi_* \colon \tcT \to \cT$ is a \emph{categorical contraction} 
if it factors as
\begin{equation*}
\tcT \longtwoheadrightarrow \tcT / \Ker(\pi_*) \longhookrightarrow \cT,
\end{equation*}
where the first arrow is a Verdier localization and the second arrow is a dense embedding (\textit{i.e.}, a fully faithful functor such that every object of~$\cT$ is a direct summand of an object of~$\tcT / \Ker(\pi_*)$).
\end{definition}

\begin{remark}
\label{rem:disclaimer}
In~\cite[Definition~3.7]{Ef20} the same notion is called a localization.
We find this confusing and change the terminology to avoid possible misunderstanding.
\end{remark}

As we mentioned above, it is useful for our goals to have a categorical contraction~$\pi_* \colon \tcT \to \Db(X)$
such that the subcategory~$\Ker(\pi_*) \subset \tcT$ is generated by spherical objects.
Using Serre duality we observe that this condition 
implies the following property, discussed in more detail in~\cite[Section~5]{KS:hfd}
(in fact, the definition of crepancy~\cite[Definition~5.6]{KS:hfd} is different from the definition given below,
but~\cite[Lemma~5.7]{KS:hfd} says that these definitions are equivalent when~$\tcT$ is smooth and proper).

\begin{definition}
\label{def:ccc}
We say that a categorical contraction~$\pi_* \colon \tcT \to \cT$ from a smooth and proper category~$\tcT$ is \emph{crepant} 
if the orthogonals~$\Ker(\pi_*)^\perp \subset \tcT$ and~${}^\perp\Ker(\pi_*) \subset \tcT$ 
of the kernel~$\Ker(\pi_*) \subset \tcT$ coincide.
\end{definition}

In the last part of the paper we show that the approach to the construction of absorption of singularities via crepant categorical contractions can be made completely effective.  Actually, in~Section~\ref{sec:nodal-varieties} for a variety~$X$ of dimension at least two whose singularities are \emph{ordinary double points}, we construct an admissible subcategory~$\cD \subset \Db(\tX)$ in the derived category of the blowup~$\pi \colon \tX \to X$ of the singular locus of~$X$ such that the restriction~$\pi_*\vert_\cD \colon \cD \to \Db(X)$ of the pushforward functor for the blowup morphism is a crepant categorical contraction and~$\Ker(\pi_*\vert_\cD)$ is generated by spherical objects; see Theorem~\ref{thm:main-nodal}.  When this paper was finished we learned about the work~\cite{CGLMMPS}, where the same crepant categorical resolutions of nodal varieties have been independently constructed.

In the last section, Section~\ref{sec:nodal-absorption}, we combine the above results by making explicit the adherence condition in the constructed crepant categorical resolution of a nodal variety~$X$ (see Theorem~\ref{thm:main-nodal-adherence}).  In particular, we show that if the adherence condition is satisfied, then~$\Db(X)$ has a semiorthogonal collection of~$\Pinfty{p+1}$-objects absorbing singularities of~$X$, where~$p \in \{0,1\}$ is the \emph{parity} of~$\dim(X)$.  This allows us to apply Theorem~\ref{thm:intro-pinfty1} when~$\dim(X)$ is even, or Theorem~\ref{thm:intro-pinfty2} when~$\dim(X)$ is odd.  As a baby example we discuss the case of nodal quadrics in Proposition~\ref{prop:quadric}.

Furthermore, in Section~\ref{ss:obstructions} we show that the singularity category of a variety that admits an absorption by categorical ordinary double points is idempotent complete, and using this observation we deduce explicit necessary conditions for the existence of such absorption for nodal varieties of dimension at most~$3$; see Proposition~\ref{prop:obstruction123}.  In the case of threefolds, this condition is the \emph{maximal nonfactoriality} condition from~\cite{Kalck-Pavic-Shinder}; see Definition~\ref{def:mnf}.  We also check in Corollary~\ref{cor:mnf-suff} that for a projective threefold~$X$ with~$\rH^{>0}(X, \cO_X) = 0$ and a single ordinary double point, the maximal nonfactoriality condition is also sufficient for the existence of absorption.

We conclude the paper with a few geometric applications of our results; see~Section~\ref{ss:examples}.
Namely, we consider
nodal curves (Section~\ref{sss:curves}),
nodal threefolds in general (Section~\ref{sss:threefolds}),
and nodal quintic del Pezzo threefolds in particular (Section~\ref{sss:v5}).
Even more interesting examples of prime Fano threefolds are discussed 
in a separate companion paper~\cite{KS-II}.

When we worked on this paper, one of our principles was to separate the geometric aspects from the categorical ones, and to a large extent Sections~\ref{sec:categorical-odp}--\ref{sec:kronecker} are categorical, Section~\ref{sec:absorption} combines categories with geometry, while Sections~\ref{sec:nodal-varieties}--\ref{sec:nodal-absorption} are mostly geometric.  On the other hand, we managed to develop the categorical aspects much further than our geometric applications required.  We feel these developments are interesting and might be useful for further research; curious readers can find these in~\cite{KS:hfd}.

In fact, \cite{KS:hfd} includes a systematic treatment of the notion of homologically finite-dimensional objects, briefly mentioned in Lemma~\ref{lem:Pinf}, with an emphasis on the relation between semiorthogonal decompositions of a triangulated category and the category of its homologically finite-dimensional objects, that leads to a relation between semiorthogonal decompositions of~$\Db(X)$ and~$\Dp(X)$.  It also contains a development of the concept of categorical contractions and crepancy, and a generalization of the relation between~$\Pinfty2$-objects on the special fiber~$X$ and exceptional objects on the total space~$\cX$ of a smoothing (observed in Theorem~\ref{thm:intro-pinfty2}) to any admissible subcategories of~$\Db(X)$.

We would also like to mention that the results of the present paper are related to the big project aimed at the study of semiorthogonal decompositions of algebraic varieties in degenerating families.  We refer to~\cite{K21:icm} for a survey of other results in this direction.

\subsubsection*{Notation}

Throughout the paper~$\Bbbk$ denotes a base field.  All schemes are assumed to be of finite type over~$\Bbbk$, and all algebras and categories we work with are $\Bbbk$-linear.

We write~$C^\opp$ and~$\cT^\opp$ for the opposite dg-algebra or category, respectively.

For a dg-algebra~$C$ we denote by
\begin{itemize}
\item 
$\bD(C)$ the unbounded derived category of right dg-modules over~$C$;
\item 
$\Dp(C) \coloneqq \thick(C) \subset \bD(C)$ the subcategory of perfect dg-modules, \textit{i.e.}, the smallest thick subcategory of~$\bD(C)$ containing the free dg-module~$C$;
\item 
$\Db(C) \subset \bD(C)$ the subcategory of dg-modules with finite-dimensional total cohomology.
\end{itemize}

Similarly, for a projective scheme~$X$ we write~$\Dqc(X)$ for the unbounded derived category of quasicoherent sheaves, while~$\Dp(X) \subset\Dqc(X)$ and~$\Db(X) \subset\Dqc(X)$ stand for the category of perfect complexes and the bounded derived category of coherent sheaves, respectively.

We write~$\cT = \langle \cA_1, \dots, \cA_m \rangle$ for a semiorthogonal decomposition with components~$\cA_1,\dots,\cA_m$, and we write~$\bL_{\cA_i}$, $\bR_{\cA_i}$ for the left and right mutation functors with respect to the component~$\cA_i$, provided it is both left and right admissible.

All pullback, pushforward, and tensor product functors are derived (unless specified otherwise).

\subsection*{Acknowledgements}

We would like to thank 
Sasha Efimov, 
Haibo Jin,
Martin Kalck,
Shinnosuke Okawa,
Dima Orlov, 
Alex Perry, 
Yura Prokhorov,
Michael Wemyss
for their help and interest in this work.
We are grateful to the anonymous referee for useful comments.

\section{Categorical ordinary double points and \texorpdfstring{$\boldsymbol{\P^\infty}$}{PP\textasciicircum infty}-objects}
\label{sec:categorical-odp}

In this section we define categorical ordinary double points and~$\Pinfty{q}$-objects in arbitrary triangulated categories and explain how to construct fully faithful functors with admissible image from a categorical ordinary double point into a given triangulated category.  Here we work over an arbitrary field~$\kk$.

\subsection{Categorical ordinary double points}
\label{ss:codp}

For~$p \ge 0$ and~$q \ge 1$ consider the following differential graded algebras;
\begin{equation*}
\begin{aligned}
{\sA}_{p} &\coloneqq \Bbbk[\eps]/(\eps^{2}), 
&\qquad
\deg(\eps) &= - p, 
&\qquad
\rd(\eps) &= 0,
\\
{\sB}_{q} &\coloneqq \Bbbk[\uptheta], 
&\qquad
\deg(\uptheta) &= q, 
&\qquad
\rd(\uptheta) &= 0.
\end{aligned}
\end{equation*}
As the differentials are zero, we will sometimes consider these dg-algebras simply as graded algebras.

In this subsection we study the derived categories of dg-modules over~$\sA_p$ and~$\sB_q$ and their relation.  Recall that a graded algebra~$C$ is called \emph{intrinsically formal} if for every dg-algebra~$C'$ such that~$\rH^\bullet(C') \cong C$, there is a quasi-isomorphism~$C \cong C'$.  The following lemma is well known (see, \textit{e.g.}, \cite[Section~2]{KYZ09}).

\begin{lemma}
\label{lem:sap-sbq-formal}
For~$p \ge 0$ and~$q \ge 1$ the graded algebras~$\sA_p$ and~$\sB_q$ are intrinsically formal.
\end{lemma}
\begin{proof}
For~$\sB_q$ there is a standard argument.  Let~$B$ be any dg-algebra with~$\rH^\bullet(B) \cong \kk[\uptheta]$.  Let~$\theta \in B^q$ be any lift of~$\uptheta \in \rH^q(B)$.  Since~$\sB_q$ is a free associative (graded) algebra, there is a unique homomorphism of graded algebras~$\sB_q \to B$, $\uptheta \mapsto \theta$.  It follows from the Leibniz rule that it is a homomorphism of dg-algebras, and it is obvious that it induces an isomorphism on the cohomology; therefore, it is a quasi-isomorphism.

For~$\sA_p$ we argue as follows.  First, any dg-algebra~$A$ with~$\rH^\bullet(A) \cong \kk[\eps]/(\eps^2)$ corresponds to an $A_\infty$-structure on~$\kk[\eps]/(\eps^2)$.  But for~$n \ge 3$ the degree of the operation~$\bm_n$ is negative, so~\mbox{$\bm_n(\eps,\dots,\eps) = 0$}.  Thus, any $A_\infty$-structure on~$A$ is trivial, so that~$A$ is quasi-isomorphic to its cohomology.
\end{proof}

Recall that 
\begin{itemize}
\item 
a dg-algebra~$C$ is \emph{proper} if~$\dim(\rH^\bullet(C)) < \infty$, where~$\rH^\bullet(C)$ is the total cohomology of~$C$;
\item 
a dg-algebra~$C$ is \emph{smooth} if the diagonal bimodule~$\Delta_C$ is \emph{perfect}, 
\textit{i.e.}, $\Delta_C \in \Dp(C \otimes C^\opp)$.
\end{itemize}
Here and below tensor products of dg-algebras are taken over~$\Bbbk$.

Similarly, a dg-category~$\cT$ is \emph{proper} if~$\dim(\rH^\bullet(\Hom_\cT(T_1,T_2))) < \infty$ for any objects~$T_1,T_2 \in \cT$,
and~$\cT$ is \emph{smooth} if~$\Delta_\cT \in \Dp(\cT \otimes \cT^\opp)$, where~$\Delta_\cT$ is the diagonal dg-bimodule over~$\cT$.
Note that a dg-algebra~$C$ is proper or smooth if and only if the dg-category~$\cT = \Dp(C)$ has the same property.

Recall that~$\Db(\sA_p) \subset \bD(\sA_p)$ and~$\Db(\sB_q) \subset \bD(\sB_q)$ 
are the categories of dg-modules over respective dg-algebras 
which have total finite-dimensional cohomology over~$\kk$.
We also denote by~$\kk_\sA$ and~$\kk_\sB$ the simple (1-dimensional) dg-modules over~$\sA_p$ and~$\sB_q$, respectively.

\begin{proposition}
\label{prop:sap-sbq}
Let~$p \ge 0$ and~$q \ge 1$.
\begin{renumerate}
\item 
\label{item:sbq}
The dg-algebra~$\sB_q$ is smooth \textup(but not proper\,\textup), 
and~$\langle \kk_\sB \rangle = \Db(\sB_q) \simeq \Dp(\sA_{q-1})$. 
\item 
\label{item:sap}
The dg-algebra~$\sA_p$ is proper \textup(but not smooth\,\textup), and~$\langle \kk_\sA \rangle = \Db(\sA_p) \simeq \Dp(\sB_{p+1})$.
\end{renumerate}
Moreover, the categories~$\Dp(\sA_p)$ and~$\Dp(\sB_q)$ are generated by the free modules~$\sA_p$ and~$\sB_q$ as triangulated categories, \textit{i.e.}, without additional idempotent completion.
\end{proposition}

\begin{proof}
\ref{item:sbq}~ The smoothness of the dg-algebra~$\sB_q$ is evident: the diagonal bimodule can be written as
\begin{equation*}
\Delta_{\sB_q} 
\cong \Cone\left(\sB_q \otimes \sB_q[-q] \xrightarrow{\ \uptheta' - \uptheta''\ } \sB_q \otimes \sB_q\right),
\end{equation*}
where~$\uptheta'$ and~$\uptheta''$ are the generators of the first and second factor, respectively.

The containment~$\langle \kk_\sB \rangle \subset \Db(\sB_q)$ is evident, and to prove the converse containment, let~$M \in \Db(\sB_q)$.  We prove~$M \in \langle \kk_\sB \rangle$ by induction on~$\dim(\rH^\bullet(M))$.  Let~$i$ be the maximal integer such that~$\rH^i(M) \ne 0$ (it exists because~$\rH^\bullet(M)$ is finite-dimensional).  Using the distinguished triangle
\begin{equation}
\label{eq:sbq-k}
\sB_q[-q] \xrightarrow{\ \uptheta\ } \sB_q \xrightarrow{\quad} \kk_\sB
\end{equation}
and the equality~$\Ext^\bullet_{\sB_q}(\sB_q,M) = \rH^\bullet(M)$, we obtain 
an exact sequence
\begin{equation*}
\cdots \lra \Ext^i(\kk_\sB, M) \lra \rH^i(M) \lra \rH^{i+q}(M) \lra \cdots.
\end{equation*}
Since~$q > 0$, the group~$\rH^{i+q}(M)$ is zero, so any class from $\rH^i(M)$ lifts to~$\Ext^i(\kk_\sB, M)$.  Clearly, the corresponding morphism~$\kk_\sB[-i] \to \rM$ induces an embedding on~$\rH^i$.  If~$M'$ is the cone of this morphism, we have~$\dim(\rH^\bullet(M')) = \dim(\rH^\bullet(M)) - 1$.  By induction~$M' \in \langle \kk_\sB \rangle$, so also~$M \in \langle \kk_\sB \rangle$.  Thus,~\mbox{$\Db(\sB_q) = \langle \kk_\sB \rangle$}.

Finally, applying the functor~$\Ext^\bullet(-, \kk_\sB)$ to~\eqref{eq:sbq-k}, we obtain
\begin{equation*}
\Ext^\bullet_{\sB_q}(\kk_\sB, \kk_\sB) \cong 
\kk \oplus \kk[q-1],
\end{equation*}
and denoting by~$\eps$ a generator of the second summand, it is easy to see that~$\eps^2 = 0$.  Now, since~$\Db(\sB_q)$ is generated by~$\kk_\sB$, we obtain an equivalence~$\Db(\sB_q) \simeq \Dp(\RHom(\kk_\sB, \kk_\sB))$ from Keller's Morita theory~\cite[Theorem~8.5(c)]{Keller:tilting}, and the proof is finished since Lemma~\ref{lem:sap-sbq-formal} implies that~$\RHom(\kk_\sB, \kk_\sB)$ is quasi-isomorphic to~$\sA_{q-1}$.

\ref{item:sap}~
The dg-algebra~$\sA_p$ is proper because it is finite-dimensional,
and the equality~$\langle \kk_\sA \rangle = \Db(\sA_p)$ 
can be obtained easily by using the standard t-structure 
(constructed, \textit{e.g.}, in~\cite[Theorem 1.3]{HoshinoKatoMiyachi}).
Thus, it remains to check that~$\Ext^\bullet_{\sA_p}(\kk_\sA, \kk_\sA) \cong \sB_{p+1}$ 
and apply Lemma~\ref{lem:sap-sbq-formal} as in~\ref{item:sbq}.
We postpone the computation of~$\Ext^\bullet_{\sA_p}(\kk_\sA, \kk_\sA)$ till Lemma~\ref{lemma:ap-k}\ref{item:ext-sap-k-k} below.

Assume that~$\sA_p$ is smooth.  Then the dg-module~$\kk_\sA$ is perfect by~\cite[Lemmas~3.5 and~3.6]{Lunts} and~\cite[Theorem 3.18]{Orlov16}, and since~$\sA_p$ is proper, the space~$\Ext^\bullet_{\sA_p}(\kk_\sA, \kk_\sA)$ must be finite-dimensional.  But this space is actually isomorphic to~$\sB_{p+1}$, hence infinite-dimensional.  Therefore, $\sA_p$ is not smooth.

The last statement of Proposition~\ref{prop:sap-sbq} follows from the fact that~$\Db(\sB_q)$ and~$\Db(\sA_p)$ are generated by the simple modules~$\kk_\sB$ and~$\kk_\sA$, without additional idempotent completions, and from the equivalences of~\ref{item:sap} and~\ref{item:sbq} since under these equivalences, simple modules go to free modules.
\end{proof}

The module~$\kk_\sA$ is not contained in~$\Dp(\sA_p)$ and thus cannot be directly expressed in terms of~$\sA_p$; to compute~$\Ext^\bullet_{\sA_p}(\kk_\sA, \kk_\sA)$ we use an interpretation of~$\kk_\sA$ as homotopy colimit.  Recall from~\cite{BN93} that the \emph{homotopy colimit} of an infinite chain of morphisms~$F_1 \xrightarrow{\ \phi_1\ } F_2 \xrightarrow{\ \phi_{2}\ } \cdots$ in a \emph{cocomplete} (\textit{i.e.}, admitting arbitrary direct sums) triangulated category is defined as
\begin{equation*}
\hocolim F_i \coloneqq
\Cone\left( \bigoplus_{i=1}^\infty F_i 
\xrightarrow{\ \id - \mathop{\oplus}\limits_{i = 1}^\infty \phi_i\ } 
\bigoplus_{i=1}^\infty F_i \right).
\end{equation*}

Let $\hcT$ be a cocomplete triangulated category.  Recall that a set of objects~$G_\alpha \in \hcT$ is a set of \emph{compact generators} if the functors~$\Hom(G_\alpha, - )$ commute with direct sums for each~$\alpha$ and the right orthogonal of all~$G_\alpha$ in~$\hcT$ is zero, \textit{i.e.}, $\Ext^\bullet(G_\alpha, F) = 0$ for~$F \in \hcT$ and all~$\alpha$ implies~$F = 0$.

\begin{lemma}
\label{lem:hocolim-shift}
Let~$\hcT$ be a cocomplete triangulated category with a set of compact generators~$G_\alpha \in \hcT$.  If for an object~$F \in \hcT$ the complex~$\Ext^\bullet(G_\alpha,F)$ is bounded above for each~$\alpha$, then~$\hocolim F[s_i] = 0$ for any strictly increasing sequence~$\{s_i\}$ of integers and any chain of morphisms~$F[s_1] \to F[s_2] \to \cdots$.
\end{lemma}

\begin{proof}
Since~$G_\alpha$ is compact in~$\hcT$, the functor~$\RHom(G_\alpha, -)$ commutes with homotopy colimits by~\cite[Lemma~2.11]{K11}, so~\cite[Remark~2.2]{BN93} gives for each~$n \in \ZZ$ an isomorphism
\begin{equation*}
\Ext^n(G_\alpha, \hocolim F[s_i]) \cong \colim \Ext^n(G_\alpha, F[s_i]).    
\end{equation*}
The right-hand side is zero because~$\Ext^\bullet(G_\alpha, F)$ is bounded above.  Thus,~$\hocolim F[s_i]$ is right orthogonal to all~$G_\alpha$ and therefore is zero.
\end{proof}

Note that by the definition of~$\sA_p$, we have a distinguished triangle
\begin{equation}
\label{eq:sap-triangle}
\sA_p \lra \kk_\sA \lra \kk_\sA[p + 1].
\end{equation}
The computation of~$\Ext^\bullet_{\sA_p}(\kk_\sA, \kk_\sA)$ is based on using~\eqref{eq:sap-triangle} iteratively and passing to the homotopy colimit.

\begin{lemma}
\label{lemma:ap-k}
Let~$p \ge 0$.
\begin{renumerate}
\item 
\label{item:ap-k}
There are a chain of morphisms~$\sA_p = \sA_p^{(1)} \to \sA_p^{(2)} \to \cdots$ in~$\Dp(\sA_p)$ such that
\begin{equation*}
\sA_p^{(i+1)} = \Cone\left(\sA_p[i(p+1) - 1] \to \sA_p^{(i)}\right)
\end{equation*}
and a sequence of morphisms~$\sA_p^{(i)} \to \kk_\sA$ compatible with the chain maps~$\sA_p^{(i-1)} \to \sA_p^{(i)}$ such that
\begin{equation}
\label{eq:cone-sapi-k}
\Cone\left(\sA_p^{(i)} \to \kk_\sA\right) \cong \kk_\sA[i(p+1)].
\end{equation}
\item 
\label{item:hocolim-ap-k}
In~$\bD(\sA_p)$ one has~$\hocolim \kk_\sA[i(p+1)] = 0$ and\,~$\hocolim \sA_p^{(i)} \cong \kk_\sA$.
\item 
\label{item:ext-sap-k-k}
One has~$\Ext^\bullet_{\sA_p}(\kk_\sA, \kk_\sA) \cong \sB_{p+1}$.
\end{renumerate}
\end{lemma}

\begin{proof}
\ref{item:ap-k}~
We construct the chain inductively.
Consider the composition
\begin{equation*}
\sA_p[i(p + 1) - 1] \lra \kk_\sA[i(p + 1) - 1] \lra \sA_p^{(i)},
\end{equation*}
where the first arrow is the (shifted) morphism from~\eqref{eq:sap-triangle} and the second arrow is the connecting morphism in~\eqref{eq:cone-sapi-k}, and define~$\sA_p^{(i+1)}$ as the cone of the composition~$\sA_p[i(p + 1) - 1] \to \sA_p^{(i)}$.  Then~\eqref{eq:cone-sapi-k} for~$i + 1$ follows easily from the octahedral axiom.

\ref{item:hocolim-ap-k}~ As $p + 1 > 0$, we have~$\hocolim \kk_\sA[i(p+1)] = 0$ by Lemma~\ref{lem:hocolim-shift} because~$\Ext^\bullet_{\sA_p}(\sA_p,\kk_\sA) \cong \kk$ is bounded above and~$\sA_p$ is a compact generator for~$\bD(\sA_p)$.  To compute~$\hocolim \sA_p^{(i)}$ we apply the functor~$\hocolim$ to the chain of morphisms~$\sA_p^{(i)} \to \kk_\sA$ constructed in~\ref{item:ap-k}.  From~\eqref{eq:cone-sapi-k} and the exactness of the homotopy colimit, we deduce a distinguished triangle
\begin{equation*}
\hocolim \sA_p^{(i)} \lra \hocolim{} \kk_\sA \lra \hocolim \kk_\sA[i(p+1)],
\end{equation*}
and since the last term was shown to be zero, we have~$\hocolim \sA_p^{(i)} \cong \hocolim \kk_\sA \cong {} \kk_\sA$.

\ref{item:ext-sap-k-k}~
Applying the functor~$\Ext_{\sA_p}^\bullet(-,\kk_\sA)$ to the triangle~$\sA_p^{(i)} \to \sA_p^{(i+1)} \to \sA_p[i(p+1)]$, 
we deduce that
\begin{equation*}
\Ext_{\sA_p}^\bullet\left(\sA_p^{(i)}, \kk_\sA\right) \cong \kk[\uptheta]/\uptheta^i,
\qquad\text{where }
\deg(\uptheta) = p + 1.
\end{equation*}
Now passing to the homotopy colimit and using the isomorphism of~\ref{item:hocolim-ap-k}, we conclude that
\begin{equation*}
\Ext^\bullet_{\sA_p}(\kk_\sA, \kk_\sA) \cong
\Ext^\bullet_{\sA_p}\left(\hocolim \sA_p^{(i)}, \kk_\sA\right) \cong
\holim \Ext^\bullet_{\sA_p}\left(\sA_p^{(i)}, \kk_\sA\right) \cong
\holim \kk[\uptheta]/\uptheta^i,
\end{equation*}
where~$\holim$ is the homotopy limit defined as~$\holim F_i \coloneqq \Cone(\prod F_i \to \prod F_i)[-1]$, analogously to the homotopy colimit.

It follows that~$\Ext_{\sA_p}^\bullet(\kk_\sA, \kk_\sA) \cong \kk[\uptheta]$ as graded vector spaces, and it remains to identify the algebra structure.  For this, using~\eqref{eq:sap-triangle}, we obtain a distinguished triangle
\begin{equation*}
\Ext_{\sA_p}^\bullet(\kk_\sA, \kk_\sA)[-p-1] \xrightarrow{\ \uptheta\ }
\Ext_{\sA_p}^\bullet(\kk_\sA, \kk_\sA) \xrightarrow{\quad}
\kk,
\end{equation*}
which implies that the multiplication by~$\uptheta$ is injective, so~$\Ext_{\sA_p}^\bullet(\kk_\sA, \kk_\sA) \cong \kk[\uptheta]$ as algebras.
\end{proof}

Later we will use the following simple observation.

\begin{corollary}
\label{cor:ap2}
The category~$\Dp(\sA_p)$ is thickly generated by~$\sA_p^{(2)}$; \textit{i.e.}, $\Dp(\sA_p) = \thick(\sA_p^{(2)})$.
\end{corollary}

\begin{proof}
Recall the defining triangle~$\sA_p[p] \xrightarrow{\eps} \sA_p \xrightarrow{\quad} \sA_p^{(2)}$, and consider the composition
\begin{equation*}
\sA_p^{(2)}[-1-p] \lra \sA_p \lra \sA_p^{(2)},
\end{equation*}
where the first arrow is a shift of the connecting morphism in the above triangle.  Using the octahedral axiom, it is easy to see that its cone is isomorphic to the cone of the morphism~$\sA_p[p] \xrightarrow{\eps^2} \sA_p[-p]$, and as~$\eps^2 = 0$, it is isomorphic to~$\sA_p[-p] \oplus \sA_p[p + 1]$.  Thus, $\sA_p$ belongs to the thick subcategory generated by~$\sA_p^{(2)}$, and hence the latter thickly generates~$\Dp(\sA_p)$.
\end{proof}

\subsection{$\P^\infty$-objects}

As by Proposition~\ref{prop:sap-sbq} the category~$\Db(\sA_p) \simeq \Dp(\sB_{p+1})$ is generated by the simple module~$\kk_\sA$, to embed it into another triangulated category~$\cT$, we need to find in~$\cT$ an object that could serve as the image of~$\kk_\sA$.  Looking at the properties of~$\kk_\sA$ proved in Lemma~\ref{lemma:ap-k}, we arrive at the following. 

\begin{definition}
\label{def:pinfty}
Let~$q \ge 1$.  An object~$\rP$ in a triangulated category~$\cT$ is called a~\emph{$\Pinfty{q}$-object} if
\begin{aenumerate}
\item
\label{def:pinfty:ext}
$\Ext^\bullet(\rP,\rP) \cong \kk[\uptheta]$, where~$\deg(\uptheta) = q$, and
\item
\label{def:pinfty:hocolim}
$\hocolim (\rP \xrightarrow{\uptheta} \rP[q] \xrightarrow{\uptheta} \rP[2q]  \xrightarrow{\uptheta} \cdots) = 0$,
\end{aenumerate}
where the homotopy colimit in~\ref{def:pinfty:hocolim} is taken in any cocomplete category~$\hcT$ containing~$\cT$.  For instance, if~$\cT$ is a small triangulated dg-category, we can take~$\hcT = \bD(\cT)$.

When we do not want to specify the parameter~$q$, we just say that~$\rP$ is a~$\P^\infty$-object.
\end{definition}

\begin{remark}
\label{rem:pinfty}
When~$\cT = \Db(X)$, property~\ref{def:pinfty:hocolim} is satisfied automatically by Lemma~\ref{lem:hocolim-shift} because if~$\rP \in \Db(X)$, then~$\Ext^\bullet(G,\rP)$ is bounded above for any compact (\textit{i.e.}, perfect) object~$G \in \Dqc(X)$.  On the other hand, there are examples of objects for which~\ref{def:pinfty:ext} is true but~\ref{def:pinfty:hocolim} fails.  The simplest such example (maybe somewhat unexpectedly) is the free module~$\sB_q$ in~$\cT = \Dp(\sB_q)$.  Indeed, in this case
\begin{equation*}
\widehat\sB_q \coloneqq 
\hocolim \left(\sB_q \xrightarrow{\ \uptheta\ } \sB_q[q] \xrightarrow{\ \uptheta\ } \sB_q[2q]  \xrightarrow{\ \uptheta\ } \cdots\right) \cong \kk[\uptheta, \uptheta^{-1}].
\end{equation*}
On the other hand, one can check that the object~$\Cone(\sB_q \to \widehat\sB_q) \cong \uptheta^{-1}\kk[\uptheta^{-1}]$ is a $\Pinfty{q}$-object.
\end{remark}

\begin{definition}
\label{def:cse}
For any $\Pinfty{q}$-object~$\rP \in \cT$ we define another object~$\rM$ from the distinguished triangle
\begin{equation}
\label{def:cm}
\rM \xrightarrow{\quad} \rP \xrightarrow{\ \uptheta\ } \rP[q].
\end{equation}
We call~$\rM \in \cT$ the \emph{canonical self-extension of~$\rP$}.
\end{definition}

It follows easily from~\eqref{def:cm} and Definition~\ref{def:pinfty}\ref{def:pinfty:ext} that
\begin{equation}
\label{eq:ext-cm-cr}
\Ext^\bullet(\rM, \rP) \cong \kk.
\end{equation}
Furthermore, the following computation is obvious.

\begin{lemma}
\label{lem:ext-cm}
If a pair of objects~$(\rM,\rP)$ in a triangulated category~$\cT$ satisfies~\eqref{def:cm} and~\eqref{eq:ext-cm-cr}, then
\begin{equation}
\label{eq:ext-cM}
\Ext^\bullet(\rM, \rM) \cong \sA_{q-1}.
\end{equation}
\end{lemma}

In these terms we can state a criterion for a full embedding of a categorical double point into a $\kk$-linear dg-enhanced triangulated category~$\cT$.  We will say that an object~$F \in \cT$ is \emph{left} or \emph{right homologically finite-dimensional} if~$\Ext^\bullet(F,T)$ or~$\Ext^\bullet(T,F)$, respectively, is finite-dimensional over~$\kk$ for any~$T \in \cT$ (\textit{cf.}~\cite[Section~4.1]{KS:hfd}).

{\samepage
\begin{lemma}
\label{lem:Pinf}
Let $\cT$ be a dg-enhanced triangulated category, and let~$(\rM,\rP)$ be a pair of objects in~$\cT$ satisfying~\eqref{eq:ext-cm-cr} and such that~$\Cone(\rM \to \rP) \cong \rP[q]$ and\,~$\hocolim \rP[iq] = 0$ for the chain from Definition~\textup{\ref{def:pinfty}\ref{def:pinfty:hocolim}}.

\begin{enumerate}[label={\textup{(\roman*)}}]
\item
\label{item:phi-rp}
We have~$\RHom_\cT(\rM, \rM) \cong \sA_{q - 1}$, and the functor
\begin{equation*}
\Phi_\rM \colon \bD(\sA_{q - 1}) \lra \bD(\cT),
\qquad
F \longmapsto F \otimes_{\sA_{q - 1}} \rM
\end{equation*}
is fully faithful, with~$\Phi_\rM(\sA_{q-1}) \cong \rM$ and\,~$\Phi_\rM(\kk_\sA) \cong \rP$.  Furthermore,
\begin{equation*}
\Phi_\rM\left(\Db(\sA_{q - 1})\right) = \langle \rP \rangle \subset \cT
\qquad\text{and}\qquad
\Phi_\rM\left(\Dp(\sA_{q - 1})\right) = \langle \rM \rangle \subset \cT.
\end{equation*}
In particular, $\rP$ is a $\Pinfty{q}$-object, and~$\rM$ is its canonical self-extension.
\item
\label{item:rp-admissible}
The functor\,~$\Phi_\rM\vert_{\Db(\sA_{q-1})} \colon \Db(\sA_{q-1}) \to \cT$ has a right adjoint if and only if the subcategory~$\langle \rP \rangle \subset \cT$ is right admissible, which holds if and only if\,~$\rM $ is left homologically finite-dimensional.

\item
\label{item:rp-left-admissible}
The functor\,~$\Phi_\rM\vert_{\Db(\sA_{q-1})} \colon \Db(\sA_{q-1}) \to \cT$ has a left adjoint if and only if the subcategory~$\langle \rP \rangle \subset \cT$ is left admissible, and if\,~$\cT \simeq \cT^\opp$, this is equivalent to~$\rM$ being right homologically finite-dimensional.
\end{enumerate}
In particular, if\,~$\cT = \Db(X)$, $X$ is a projective Gorenstein scheme, 
and~$\rM \in \Dp(X)$, then~$\langle \rP \rangle \subset \Db(X)$ is admissible.
\end{lemma}
}

\begin{proof}
\ref{item:phi-rp}~ By Lemma~\ref{lem:ext-cm} a combination of~\eqref{eq:ext-cm-cr} with the assumption~\mbox{$\Cone(\rM \to \rP) \cong \rP[q]$} implies~\eqref{eq:ext-cM}.  Using Lemma~\ref{lem:sap-sbq-formal} we deduce~$\RHom_\cT(\rM,\rM) \cong \sA_{q-1}$, so ~$\rM$ is a $\sA_{q-1}\md\cT$-bimodule, and it gives a continuous (\textit{i.e.}, commuting with arbitrary direct sums) fully faithful derived tensor product functor
\begin{equation*}
\Phi_\rM \colon \bD(\sA_{q-1}) \lra \bD(\cT),
\qquad
F \longmapsto F \otimes_{\sA_{q - 1}} \rM; 
\end{equation*}
see~\cite[Section~3.8]{Kel06}.  Obviously, $\Phi_\rM(\sA_{q-1}) \cong \rM$.  Furthermore, it follows from the continuity of~$\Phi_\rM$ and Lemma~\ref{lemma:ap-k} that
\begin{equation*}
\Phi_\rM(\kk_\sA) \cong
\Phi_\rM(\hocolim \sA_{q-1}^{(i)}) \cong
\hocolim \rM^{(i)},
\end{equation*}
where $\rM^{(i)} \coloneqq \Phi_\rM(\sA_{q-1}^{(i)})$.
Using the inductive construction of the objects~$\sA_{q-1}^{(i)}$ in Lemma~\ref{lemma:ap-k} and the octahedral axiom,
it is easy to find distinguished triangles
\begin{equation}
\label{eq:rmi}
\rM^{(i)} \lra \rP \xrightarrow{\ \uptheta^i\ } \rP[iq]
\end{equation}
for all~$i \ge 1$ such that the maps~$\rM^{(i)} \to \rP$ are compatible with the chain maps~$\rM^{(1)} \to \rM^{(2)} \to \cdots$.
Therefore, taking the homotopy colimit of the above triangles, we obtain the triangle
\begin{equation*}
\hocolim \rM^{(i)} \lra \hocolim{} \rP \lra \hocolim \rP[iq],
\end{equation*}
and since the last term is zero by assumption, we obtain an isomorphism~$\Phi_\rM(\kk_\sA) \cong \hocolim \rP \cong \rP$.  As the functor~$\Phi_\rM$ is fully faithful, we conclude from Lemma~\ref{lemma:ap-k}\ref{item:ext-sap-k-k} that~$\rP$ is a $\Pinfty{q}$-object.

Finally, since by Proposition~\ref{prop:sap-sbq} the categories~$\Db(\sA_{q-1})$ and~$\Dp(\sA_{q-1})$ are generated by the objects~$\kk_\sA$ and~$\sA_{q-1}$, respectively, we conclude that their images are~$\langle \rP \rangle$ and~$\langle \rM \rangle$.

\ref{item:rp-admissible}~ Since the functor~$\Phi_\rM$ is fully faithful, right admissibility of~$\langle \rP \rangle$ is equivalent to the existence of a right adjoint for~$\Phi_\rM$ on~$\Db(\sA_{q-1})$; this proves the first equivalence.

Now note that the functor~$\Phi_\rM$ on~$\bD(\sA_{q-1})$ always has the right adjoint given by
\begin{equation*}
\RHom_\cT(\rM,-) \colon \bD(\cT) \lra \bD\left(\sA_{q-1}\right);
\end{equation*}
see~\cite[Section~3.8]{Kel06}.  If~$\rM$ is left homologically finite-dimensional, this functor takes~$\cT$ to~$\Db(\sA_{q-1})$, so it gives a right adjoint to~$\Phi_\rM$ on~$\cT$.  Conversely, if a right adjoint functor~$\Phi_\rM^! \colon \cT \to \Db(\sA_{q-1})$ exists, then for any~$T \in \cT$
\begin{equation*}
\Ext^\bullet_\cT(\rM, T) \cong 
\Ext^\bullet_\cT\left(\Phi_\rM\left(\sA_{q-1}\right), T\right) \cong 
\Ext^\bullet_{\sA_{q-1}}\left(\sA_{q-1}, \Phi_\rM^!(T)\right) 
\cong \rH^\bullet\left(\Phi_\rM^!(T)\right)
\end{equation*}
is finite-dimensional, so~$\rM$ is left homologically finite-dimensional.  This proves the second equivalence.

\ref{item:rp-left-admissible}~ We apply~\ref{item:rp-admissible} to the opposite categories, using an equivalence~$\sA_{q-1}^\opp \simeq \sA_{q-1}$ (which holds because~$\sA_{q-1}$ is commutative) and the equivalence~$\cT \simeq \cT^\opp$.  Passing to opposite categories swaps left and right adjoints and left and right homological finite-dimensionality, and the result follows.

The final statement is a combination of parts~\ref{item:rp-admissible} and~\ref{item:rp-left-admissible} because left homological finite-dimensionality for an object~$\rM \in \Db(X)$ is equivalent to the condition~$\rM \in \Dp(X)$, see~\cite[Proposition~1.11]{Orl06}, and when~$X$ is Gorenstein, the Grothendieck duality implies that the same is true for right homological finite-dimensionality.
\end{proof}

We conclude this section with an obvious consequence of Corollary~\ref{cor:ap2}.

\begin{corollary}
\label{cor:sm2}
Under the assumptions of Lemma~\textup{\ref{lem:Pinf}}, we have~$\rM \in \thick(\rM^{(2)}) \subset \cT$.
\end{corollary}

We will also need the following simple consequence of~\eqref{eq:rmi}, generalizing~\eqref{eq:ext-cm-cr}.

\begin{corollary}
\label{cor:ext-rmi-rp}
We have~$\Ext^\bullet(\rM^{(i)}, \rP) \cong \kk[\uptheta]/\uptheta^i$.
\end{corollary}

Despite the apparent symmetry between the algebras~$\sA_p$ and~$\sB_q$ and their derived categories, the category we will use mostly is the category~$\Db(\sA_p)$.  In particular, in~Section~\ref{sec:nodal-absorption} we will show that varieties with ordinary double points often admit subcategories equivalent to~$\Db(\sA_p)$ in their bounded derived categories.  For this reason we call~$\Db(\sA_p)$ \emph{the categorical ordinary double point} of degree~$p$.


\section{Kronecker quiver and categorical ordinary double points}
\label{sec:kronecker}

In this section we introduce the graded Kronecker quiver category, study its properties, and show that the categorical ordinary double point can be obtained as its Verdier localization with respect to a spherical object.  Also, we explain how to construct fully faithful functors with admissible image from the graded Kronecker quiver category.  We keep working over an arbitrary field~$\kk$.

\subsection{The graded Kronecker quiver}

Let~$q \ge 1$ be a positive integer.  We define the \emph{graded Kronecker quiver of degree~$q$} as the dg-quiver
\begin{equation}
\label{eq:krq}
\Kr_q \coloneqq
\left\{
\
\xymatrix@1@C=4em{
\bullet \ar@/^/[r]^{\upalpha_q} \ar@/_/[r]_{\upalpha_0} & \bullet,
}
\qquad
\deg(\upalpha_0) = 0,\quad
\deg(\upalpha_q) = q,
\qquad
\rd(\upalpha_0) = \rd(\upalpha_q) = 0
\
\right\}.
\end{equation}
We will also use the same notation~$\Kr_q$ for the path dg-algebra of this quiver.

\begin{remark}
The same definition makes sense for any~$q$, {not necessarily positive}.  However, for $q < 0$ the resulting dg-algebra~$\Kr_q$ is Morita equivalent to~$\Kr_{-q}$, while for~$q = 0$ the derived category of~$\Kr_0$ is equivalent to the derived category of the projective line~$\P^1$ (\textit{cf.}~the proof of Lemma~\ref{lem:rkpm-spherical}).
\end{remark}

\begin{lemma}
\label{lemma:kr-formal}
The dg-algebra~$\Kr_{q}$ is intrinsically formal, smooth, and proper.
\end{lemma}

\begin{proof}
The intrinsic formality is proved by the argument of Lemma~\ref{lem:sap-sbq-formal}; see also~\cite[proof of Proposition~A.3]{K19}.  The smoothness is clear since the quiver~\eqref{eq:krq} is directed, and the properness is evident.
\end{proof}

A combination of the smoothness and properness of the dg-algebra~$\Kr_q$ implies that the category~$\Dp(\Kr_q)$ of perfect dg-modules over~$\Kr_q$ coincides with the category~$\Db(\Kr_q)$ of dg-modules with bounded total cohomology.  We denote this category by
\begin{equation*}
\cKr_q \coloneqq \Dp\left(\Kr_q\right) = \Db\left(\Kr_q\right) 
\end{equation*}
and call it \emph{the graded Kronecker quiver category of degree~$q$}.

The vertices of the quiver provide two idempotents in~$\Kr_q$, and since the quiver is directed, the corresponding direct summands of the free module form an exceptional pair~$(\cE,\cE')$ in~$\cKr_q$ generating the category.  In other words, we have a semiorthogonal decomposition
\begin{equation*}
\cKr_q \coloneqq \langle \cE, \cE' \rangle,
\end{equation*}
with exceptional objects~$\cE$ and~$\cE'$ and
\begin{equation}
\label{eq:ext-ce-cep}
\Ext^\bullet(\cE, \cE') = \kk \oplus \kk[-q].
\end{equation}
The intrinsic formality (Lemma~\ref{lemma:kr-formal}) of the dg-algebra~$\Kr_q$ has the following useful consequence.

\begin{corollary}
\label{cor:krq-category}
Every dg-enhanced triangulated category~$\cT$ generated by an exceptional pair~$(\cE,\cE')$ satisfying~\eqref{eq:ext-ce-cep} is equivalent to~$\cKr_q$.
\end{corollary}

We define the objects~$\rK_+, \rK_- \in \cKr_q$ from the distinguished triangles
\begin{align}
\label{def:ck}
\rK_+ & \xrightarrow{\quad} \cE \xrightarrow{\ \upalpha_0\ } \cE',\\
\label{def:tcm}
\rK_- & \xrightarrow{\quad} \cE \xrightarrow{\ \upalpha_q\ } \cE'[q].
\end{align}

\begin{lemma}
\label{lemma:exts-tcap}
The~$\Ext$-groups between {the objects~$\rK_-,\cE,\cE',\rK_+ \in \cKr_q$} are given in the table below
\begin{equation*}
  \renewcommand{\arraystretch}{1.1} 
\begin{array}{|c||c|c|c|c|}
\hline
& \hphantom{1234.}\rK_-\hphantom{1234.} &
\hphantom{12345}\cE\hphantom{12345} &
\hphantom{12345}\cE'\hphantom{12345} &
\rK_+ \\
\hline
\hline
\hphantom{12}\rK_-\hphantom{12} & \kk \oplus \kk[q-1] & \kk & \kk & 0 \\
\hline
\cE & \kk[q-1]  & \kk & \kk \oplus \kk[-q] & \kk[-q-1] \\
\hline
\cE' & \kk[q-1]  & 0 & \kk  & \kk[-1] \\
\hline
\rK_+ & 0  & \kk & \kk[-q] & \kk \oplus \kk[-q-1] \\
\hline
\end{array}
\end{equation*}
where each entry shows the~$\Ext$-groups from the object in the given row to the object in the given column.
\end{lemma}

\begin{proof}
These are easily computed from the defining triangles~\eqref{def:ck} and~\eqref{def:tcm}, using the exceptionality of the pair~$(\cE,\cE')$ and~\eqref{eq:ext-ce-cep}.
\end{proof}

Recall that an object~$\rK \in \cT$ in a triangulated category with a Serre functor~$\bS_\cT$ is \emph{$r$-spherical}, \textit{cf.}~\cite{Seidel-Thomas}, if
\begin{equation*}
\Ext^\bullet(\rK, \rK) = \kk \oplus \kk[-r]
\qquad\text{and}\qquad
\bS_\cT(\rK) \cong \rK[r].
\end{equation*}
Note that since the Kronecker quiver category~$\cKr_q$ is generated by an exceptional pair, it has a Serre functor (\textit{e.g.}, by~\cite[Proposition~3.8]{BK}).

\begin{lemma}
\label{lem:rkpm-spherical}
The objects~$\rK_+$ and~$\rK_-$ in $\cKr_q$ are~$(1+q)$-spherical and~$(1-q)$-spherical, respectively.
\end{lemma}

\begin{proof}
The spaces~$\Ext^\bullet(\rK_+, \rK_+)$ and~$\Ext^\bullet(\rK_-, \rK_-)$ have already been computed in Lemma~\ref{lemma:exts-tcap}, so it remains to compute~$\bS_{\cKr_q}(\rK_\pm)$.  We start with a few remarks about the structure of the category~$\cKr_q$, which is very similar to the structure of~$\Db(\P^1) \simeq{} \cKr_0$ (\textit{e.g.}, the objects~$\cE_i$ constructed below are analogous to the line bundles~$\cO(i)$, while~$\rK_+$ and~$\rK_-$ are analogous to the structure sheaves of the points~$0,\infty \in \P^1$).

Consider the infinite sequence of exceptional objects~$\cE_i \in \cKr_q$, $i \in \Z$, defined by
\begin{equation*}
\cE_0 = \cE,
\qquad
\cE_1 = \cE',
\qquad
\cE_{i+1} = \bR_{\cE_i}(\cE_{i-1})[1-q],
\qquad
\cE_{i-1} = \bL_{\cE_i}(\cE_{i+1})[q-1],
\end{equation*}
where~$\bR$ and~$\bL$ stand for the right and left mutation functors.  Spelling out the definition of~$\cE_2$, we have a distinguished triangle (which is analogous to the Euler sequence on~$\P^1$)
\begin{equation}
\label{eq:ce012}
\cE_0 \xrightarrow{\ (\upalpha_0,\upalpha_q)\ } \cE_1 \oplus \cE_1[q] \lra
\cE_2[q].
\end{equation}
It follows from this that~$\Ext^\bullet(\cE_1,\cE_2) \cong \kk \oplus \kk[-q] \cong \Ext^\bullet(\cE_0, \cE_1)$.  Since both~$(\cE_0,\cE_1)$ and~$(\cE_1,\cE_2)$ are exceptional pairs generating~$\cKr_q$, Corollary~\ref{cor:krq-category} implies that there is an autoequivalence
\begin{equation*}
\tau \colon \cKr_q \to \cKr_q
\qquad
\text{such that $\tau(\cE_0) = \cE_1$ and $\tau(\cE_1) = \cE_2$}
\end{equation*}
(analogous to the~$\cO(1)$-twist in~$\Db(\P^1)$), and upon replacing~\eqref{eq:ce012} with an isomorphic triangle, we may assume that the second arrow in~\eqref{eq:ce012} takes the form~$(-\tau(\upalpha_q),\tau(\upalpha_0))$.  Then it follows from the definition of the sequence~$\cE_i$ that
\begin{equation}
\label{eq:tau}
\tau(\cE_i) \cong \cE_{i+1}
\qquad\text{for all~$i \in \Z$.}
\end{equation}
Also note that the octahedral axiom applied to~\eqref{eq:ce012} implies that
\begin{alignat}{3}
\Cone\left(\cE_1[q] \xrightarrow{\ \tau(\upalpha_0)\ } \cE_2[q]\right) &{}\cong
\Cone\left(\cE_0 \xrightarrow{\ \upalpha_0\ } \cE_1\right) &{}\cong \rK_+[1],\\
\Cone\left(\cE_1 \xrightarrow{\ \tau(\upalpha_q)\ } \cE_2[q]\right) &{}\cong
\Cone\left(\cE_0 \xrightarrow{\ \upalpha_q\ } \cE_1[q]\right) &{}\cong \rK_-[1];
\end{alignat}
hence~$\tau(\rK_+) \cong \rK_+[-q]$, $\tau(\rK_-) \cong \rK_-$, and therefore, we have distinguished triangles
\begin{equation}
\label{eq:rkpm-new}
\rK_+[-iq] \xrightarrow{\quad} \cE_i \xrightarrow{\ \tau^i(\upalpha_0)\ } \cE_{i+1}
\qquad\text{and}\qquad
\rK_- \xrightarrow{\quad} \cE_i \xrightarrow{\ \tau^i(\upalpha_q)\ } \cE_{i+1}[q]
\end{equation}
for any~$i \in \Z$.  On the other hand, it follows from~\eqref{eq:ce012} that~$\Ext^\bullet(\cE_2,\cE_0) \cong \kk[q-1]$, and therefore 
\begin{equation*}
\Ext^\bullet(\cE_0,\bS_{\cKr_q}(\cE_2)) \cong
\Ext^\bullet(\cE_2, \cE_0)^\vee \cong \kk[1-q],
\qquad
\Ext^\bullet(\cE_1,\bS_{\cKr_q}(\cE_2)) \cong
\Ext^\bullet(\cE_2, \cE_1)^\vee = 0.
\end{equation*}
Since the category~$\cKr_q$ is generated by the exceptional pair~$(\cE_0,\cE_1)$, the above isomorphisms imply that~$\bS_{\cKr_q}(\cE_2) \cong \cE_0[1-q]$.  Since the Serre functor commutes with any autoequivalence, we conclude from this and~\eqref{eq:tau} that for all~$i \in \Z$ we have
\begin{equation*}
\bS_{\cKr_q}(\cE_i) \cong \cE_{i-2}[1-q].
\end{equation*}
Applying this to triangle~\eqref{def:ck} and using the first triangle in~\eqref{eq:rkpm-new} for~$i = -2$, we conclude that
\begin{equation*}
\bS_{\cKr_q}(\rK_+) \cong
\Cone\left(\bS_{\cKr_q}(\cE_0) \lra \bS_{\cKr_q}(\cE_1)\right)[-1] \cong
\Cone\left(\cE_{-2}[-q] \lra \cE_{-1}[-q]\right) \cong
\rK_+[1+q].
\end{equation*}
Thus, $\rK_+$ is $(1+q)$-spherical, and a similar computation shows~$\rK_-$ is~$(1-q)$-spherical.
\end{proof}

\begin{remark}\label{rem:sphere}
The category~$\cKr_q$ also has a topological interpretation (see~\cite{BBD81} and~\cite{deCat-Mig} for the required material).  Let~$\rS^{q+1}$ be a $(q+1)$-dimensional real sphere with a point~$\{P\} \hookrightarrow \rS^{q+1}$ and its open complement~$U \cong \RR^{q+1} \hookrightarrow \rS^{q+1}$.  Then~$\cKr_q$ is equivalent to the category of complexes of sheaves of $\Bbbk$-vector spaces on~$\rS^{q+1}$ constructible with respect to the stratification~$\{P, U\}$.  Under this equivalence the object~$\rK_{+}$ corresponds to the constant sheaf~$\Bbbk_{\rS^{q+1}}$.

Note that this interpretation gives an alternative explanation for the fact that~$\rK_+ = \Bbbk_{\rS^{q+1}}$ is spherical.  Indeed, the self Ext-algebra of the constant sheaf~$\Bbbk_{\rS^{q+1}}$ is isomorphic to the cohomology of the sphere, and using Verdier duality one can show that the Serre functor of~$\cKr_{q+1}$ acts on~$\Bbbk_{\rS^{q+1}}$ as the shift~$[q+1]$.
\end{remark}

In fact, in~\cite{Kalck-Yang-graded} a more general ``orbifold version'' of the graded Kronecker quiver has been studied.  Its special case, the graded quiver~$\Gamma(1,1,r)$ defined in~\cite[Theorem~1.1]{Kalck-Yang-graded}, is isomorphic to the quiver~$\Kr_r$, and Lemma~\ref{lem:rkpm-spherical} can be deduced from~\cite[Theorem~1.2(a)]{Kalck-Yang-graded}.

\subsection{Localization and adherence}

The next proposition relates the graded Kronecker quiver category~$\cKr_{p+1}$ to the categorical double point~$\Db(\sA_p)$.

\begin{proposition}
\label{prop:contr-Cp}
There is an equivalence
\begin{equation*}
\cKr_{p+1} / \langle \rK_+ \rangle \simeq \Db(\sA_p)
\end{equation*}
between the Verdier localization of the Kronecker quiver category and the categorical double point such that the corresponding localization functor~$\uprho_* \colon \cKr_{p+1} \to \Db(\sA_p)$ satisfies
\begin{renumerate}
\item
\label{item:ckr-a-simple}
$\uprho_*(\cE) \cong \uprho_*(\cE') = \kk_\sA$, 
$\uprho_*(\upalpha_0) = \id_{\kk_\sA} \in \Hom_{\sA_p}(\kk_\sA, \kk_\sA)$, 
\mbox{$\uprho_*(\upalpha_{p+1}) = \uptheta \in \Ext^{p+1}_{\sA_p}(\kk_\sA, \kk_\sA)$};
\item
\label{item:ckr-a-free}
$\uprho_*(\rK_-) = \sA_p$.
\end{renumerate}
\end{proposition}

\begin{proof}
Set~$\cT \coloneqq \cKr_{p+1} / \langle \rK_+ \rangle$, and let~$\uprho_* \colon \cKr_{p+1} \to \cT$ be the localization functor.  We will use Lemma~\ref{lem:Pinf} to construct the equivalence~$\cT \simeq \Db(\sA_p)$.  So, let
\begin{equation*}
\rP \coloneqq \uprho_*(\cE) \cong \uprho_*(\cE')  \in \cT,
\qquad
\rM \coloneqq \uprho_*(\rK_-) \in \cT
\end{equation*}
(the isomorphism~$\uprho_*(\cE) \cong \uprho_*(\cE')$ follows from~\eqref{def:ck}; it is induced by~$\uprho_*(\upalpha_0)$).  Since~$\cKr_{p+1}$ is generated by~$\cE$ and~$\cE'$, it follows that~$\cT$ is generated by~$\rP$.  By Lemma~\ref{lemma:exts-tcap} the object~$\rK_-$ is orthogonal to~$\rK_+$, so by the definition of Verdier localization, we have
\begin{equation*}
\Ext^\bullet_{\cT}(\rM, \rP) \cong \Ext^\bullet_{{\cKr_{p+1}}}(\rK_-, \cE) \cong \Bbbk,
\end{equation*}
where the second isomorphism follows again from Lemma~\ref{lemma:exts-tcap}.  Furthermore, applying~$\uprho_*$ to~\eqref{def:tcm} we conclude that~$\Cone(\rM \to \rP) \cong \rP[p + 1]$, with the morphism~$\rP \to \rP[p+1]$ given by~$\uprho_*(\upalpha_{p+1})$.  Finally, note that~$\hocolim \cE[i(p+1)] = 0$ in~$\bD(\Kr_{p+1})$ by Lemma~\ref{lem:hocolim-shift} because~$\cE$ and~$\cE'$ compactly generate~$\bD(\Kr_{p+1})$, and since~$\uprho_*$ induces a continuous (\textit{i.e.}, commuting with arbitrary direct sums) functor~$\bD(\Kr_{p+1}) \to {\bD(\cT)}$, it follows that~$\hocolim \rP[i(p+1)] = 0$.

Now applying Lemma~\ref{lem:Pinf}\ref{item:phi-rp} we obtain a fully faithful functor~$\Phi_\rM \colon \Db(\sA_p) \to \cT$ that takes~$\kk_\sA$ to~$\rP$.  Since~$\cT$ is generated by~$\rP$, this functor is essentially surjective, so it is an equivalence.  This completes the proof of the first statement of the proposition.

Since~$\uptheta \in \Ext^{p+1}(\kk_\sA,\kk_\sA)$ is taken by the functor~$\Phi_\rM$ to a nontrivial element in~$\Hom(\rP,\rP[p+1])$, we conclude that, rescaling~$\uptheta$ if necessary, we obtain~$\uprho_*(\upalpha_{p+1}) = \uptheta$.  This proves~\ref{item:ckr-a-simple}.  Similarly, since~$\Phi_\rM$ takes~$\sA_p$ to~$\rM$, comparing with the definition of~$\rM$, we deduce~\ref{item:ckr-a-free}.
\end{proof}

\begin{remark}
Since~$\Ker(\uprho_*) = \langle \rK_+ \rangle$ is generated by a spherical object, the functor~$\uprho_*$ is an example of a crepant categorical contraction (see Definitions~\ref{def:cc} and~\ref{def:ccc}).  One can also show that~$\uprho_*$ has a left adjoint functor~$\uprho^* \colon \Dp(\sA_p) \to \cKr_{p+1}$ such that~\mbox{$\uprho^*(\sA_p) \cong \rK_-$} and~$(\cKr_{p+1},\uprho^*,\uprho_*)$ is a weakly crepant categorical resolution in the sense of~\cite{K08}.

Conversely, it is easy to check that the dg-algebra~$\Kr_{p+1}$ is Morita equivalent to the \emph{Auslander resolution} of the dg-algebra~$\sA_p$; see the definition in~\cite[Section~5]{KL} for the case of algebras, \cite[Section~2.3]{O20} for the case of dg-algebras, and~\cite[Example~5.3]{KL} for {a computation in} the case~$p = 0$.  Therefore, the graded Kronecker quiver category~$\cKr_{p+1}$ provides a categorical resolution in the sense of~\cite{KL} for the categorical ordinary double point~$\Db(\sA_p)$.
\end{remark}

We conclude this section by explaining how an admissible subcategory equivalent to~$\cKr_q$ can be embedded into a given triangulated category.  For this the following definition, generalizing~\cite[Definition~3.6]{KKS20}, is convenient.

\begin{definition}
\label{def:adherence}
Let~$\tcT$ be a proper triangulated category with a Serre functor, and let~$\rK \in \tcT$ be a spherical object.  We say that an exceptional object~$\cE \in \tcT$ is \emph{adherent} to~$\rK$ if
\begin{equation}
\label{eq:dim-ck-ce-condition}
\dim \Ext^\bullet(\rK,\cE) = \dim \Ext^\bullet(\cE, \rK) = 1.
\end{equation}
Note that the first equality follows from Serre duality and the definition of a spherical object.
\end{definition}

Recall that a spherical object~$\rK$ in a dg-enhanced proper triangulated category~$\tcT$ gives rise to an autoequivalence
\begin{equation}
\label{eq:sph-twist}
\bT_\rK(-) \coloneqq{} \Cone\left(\RHom^\bullet(\rK, -) \otimes \rK \xrightarrow{\ \mathrm{ev}\ } -\right),
\end{equation}
which is called a \emph{spherical twist}; \textit{cf.}~\cite{Seidel-Thomas}.

\begin{lemma}
\label{lem:kronecker-recogn}
Let~$\tcT$ be a dg-enhanced proper triangulated category with a Serre functor~$\bS_\tcT$, and let~$\rK \in \tcT$ be a~$(p+2)$-spherical object with~$p \ge 0$.  If an exceptional object~$\cE \in \tcT$ is adherent to~$\rK$, then
\begin{equation}
\label{eq:ce-bt-ce}
(\cE, \bT_{\rK}(\cE))
\end{equation}
is an exceptional pair and the subcategory~$\langle \cE, \bT_{\rK}(\cE) \rangle \subset \tcT$ generated by~\eqref{eq:ce-bt-ce} is equivalent to\,~$\cKr_{p+1}$ and admissible in~$\tcT$.
\end{lemma}

\begin{proof}
Shifting~$\rK$ appropriately (note that this does not affect the corresponding spherical twist), we may assume~$\Ext^\bullet(\rK, \cE) = \Bbbk$.  Since~$\bT_{\rK}$ is an autoequivalence, $\bT_{\rK}(\cE)$ is also an exceptional object.  Furthermore, by definition~\eqref{eq:sph-twist} of the spherical twist, we have a distinguished triangle
\begin{equation*}
\rK \lra \cE \lra \bT_{\rK}(\cE).
\end{equation*}
Applying the functor~$\Ext^\bullet(-, \cE)$ to this triangle, we see that the pair~$(\cE, \bT_{\rK}(\cE))$ is exceptional.  {On the other hand, using Serre duality in~$\tcT$, we obtain
\begin{equation*}
\Ext^\bullet(\cE, \rK)
\cong \Ext^\bullet(\bS_\tcT^{-1}(\rK), \cE)^\vee
\cong \Ext^\bullet(\rK[-p-2], \cE)^\vee
= \Bbbk[-p-2], 
\end{equation*}
and so, applying~$\Ext^\bullet(\cE, -)$ to the above triangle}, we see that the $\Ext$-groups between~$\cE$ and~$\bT_{\rK}(\cE)$ are those of the graded Kronecker quiver of degree~$p + 1$.  Thus, by Corollary~\ref{cor:krq-category} the subcategory in~$\tcT$ generated by the pair~$(\cE, \bT_{\rK}(\cE))$ is equivalent to~$\cKr_{p+1}$, and since this subcategory is generated by an exceptional pair in a proper triangulated category, it is admissible.
\end{proof}

In Section~\ref{sec:nodal-absorption} we will present several geometric applications of Lemma~\ref{lem:kronecker-recogn} (or rather of Theorem~\ref{thm:contractions} that relies on it), where the category~$\tcT$ will be taken to be a crepant categorical resolution of~$\Db(X)$.  Here we mention the following simpler example; see~\cite[Example~2.14]{KKS20} for details.

\begin{example}
\label{ex:surf-node-resol}
Let~$\tcT = \Db(\tX)$, where~$\tX$ is a smooth projective surface such that~$\rH^{>0}(\tX, \cO_{\tX}) = 0$.  If~$E \subset \tX$ is a $(-2)$-curve, then~$\rK = \cO_E(-1)$ is a $2$-spherical object.  Furthermore, if~$D \in \Pic(\tX)$ is a divisor class satisfying~$D \cdot E = 1$, then the exceptional object~$\cE = \cO_\tX(D)$ is adherent to~$\rK$ and~$\bT_{\rK}(\cE) = \cO(E+D)$.  In this situation Lemma~\ref{lem:kronecker-recogn} says that the subcategory
\begin{equation*}
\langle \cO_\tX(D), \cO_\tX(E+D) \rangle
\end{equation*}
is admissible in~$\Db(\tX)$ and is equivalent to~$\cKr_1$.
\end{example}


\section{Adherence and categorical absorption of singularities}
\label{sec:absorption}

In this section we relate the notion of adherence introduced in Definition~\ref{def:adherence} to the notion of categorical absorption of singularities from Definition~\ref{def:intro-absorption}.  After that we prove Theorems~\ref{thm:intro-deformation-absorption-to-smoothing}, \ref{thm:intro-pinfty2}, and~\ref{thm:intro-pinfty1} from the introduction.  We keep working over an arbitrary field~$\kk$.

\subsection{From adherence to categorical absorption}

The following result allows us to descend semiorthogonal decompositions along Verdier localizations.

\begin{proposition}
\label{prop:sod-crepancy}
Let $\pi_* \colon \tcT \to \cT$ be a Verdier localization with the kernel category~\mbox{$\cK \coloneqq \Ker(\pi_*) \subset \tcT$}.  Let
\begin{equation*}
\tcT = \left\langle \tcT_1, \dots, \tcT_n \right\rangle
\end{equation*}
be a semiorthogonal decomposition compatible with~$\cK$, \textit{i.e.}, such that the subcategories~\mbox{$\cK_i \coloneqq \cK \cap \tcT_i$} form a semiorthogonal decomposition~$\cK = \langle \cK_1, \dots, \cK_n \rangle$.  Then there is a semiorthogonal decomposition
\begin{equation*}
\cT = \langle \cT_1, \dots, \cT_n \rangle
\end{equation*}
with~$\cT_i \coloneqq \pi_*(\tcT_i) \subset \cT$, and the restriction~$\pi_*\vert_{\tcT_i} \colon \tcT_i \to \cT_i$ is a Verdier localization with kernel~$\cK_i$.
\end{proposition}

\begin{proof}
We assume~$n = 2$, as the general case follows by induction.  The natural functors
\begin{equation*}
\cT_i \simeq \tcT_i / \cK_i \longhookrightarrow \tcT / \cK = \cT
\end{equation*}
are fully faithful by~\cite[Lemma~1.1]{Orl06}.  Moreover, it also follows that the subcategories~$\cT_i$ are semiorthogonal; see~\cite[Proposition~1.11]{Orl06}.  Indeed, the semiorthogonality of~$\tcT_1$ and~$\tcT_2$ can be rephrased by saying that the left adjoint functor for the embedding~$\tcT_1 \hookrightarrow \tcT$ vanishes on the subcategory~$\tcT_2$, and~\cite[Lemma~1.1]{Orl06} then implies that the left adjoint functor for the embedding~$\cT_1 \hookrightarrow \cT$ vanishes on~$\cT_2$, which implies the semiorthogonality of~$\cT_1$ and~$\cT_2$.  Finally, the categories~$\cT_i$ generate~$\cT$ because~$\pi_*$ is essentially surjective and the~$\tcT_i$ generate~$\tcT$.  Thus, $\cT = \langle \cT_1, \cT_2 \rangle$ is a semiorthogonal decomposition.
\end{proof}

We now explain how to construct subcategories absorbing singularities.  Let~$\delta_{ij}$ be the Kronecker delta.  Recall the definition of a crepant categorical contraction (Definitions~\ref{def:cc} and~\ref{def:ccc}).

\begin{theorem}
\label{thm:contractions}
Let~$\tcT$ be a dg-enhanced smooth and proper triangulated category, and let~$\rK_1, \dots, \rK_r \in \tcT$ be a collection of~$r$ completely orthogonal $(p+2)$-spherical objects {with}~$p \ge 0$.  Assume that there is a \textup(nonfull\,\textup) exceptional collection~$\cE_1, \dots, \cE_r$ in~$\tcT$ such that
\begin{equation}
\label{eq:dim-ext-ce-ck}
\dim \Ext^\bullet(\cE_i,\rK_j) = \delta_{ij};
\end{equation}
in particular, $\cE_i$ is adherent to~$\rK_i$ for each~$i$.  Set~$\cK \coloneqq \langle \rK_1, \dots, \rK_r \rangle \subset \tcT$, and consider the localization
\begin{equation*}
\pi_* \colon \tcT \lra \cT \coloneqq \tcT/\cK.
\end{equation*}
Then the following hold:

\begin{enumerate}[label={\textup{(\roman*)}}]
\item
\label{item:tcc}
For each~$1 \le i \le r$ the triangulated subcategory
\begin{equation*}
\tcP_i \coloneqq \langle \cE_i, \bT_{\rK_i}\cE_i \rangle
\end{equation*}
generated in~$\tcT$ by~$\cE_i$ and~$\bT_{\rK_i}(\cE_i)$ \textup(or, equivalently, by~$\cE_i$ and~$\rK_i$\textup) is equivalent to {the graded Kronecker quiver category}~$\cKr_{p+1}$ and is admissible in~$\tcT$.  Moreover, the collection {of subcategories~$\tcP_1, \dots, \tcP_r \subset \cT$} is semiorthogonal in~$\tcT$, and the subcategory
\begin{equation*}
\tcP \coloneqq \left\langle \tcP_1, \dots, \tcP_r \right\rangle
\end{equation*}
is admissible in~$\tcT$.

\item
\label{item:cc}
For each~$i$ the object~$\rP_i \coloneqq \pi_*(\cE_i)$ is a~$\Pinfty{p+1}$-object, and the category~$\cP_i \coloneqq \pi_*(\tcP_i) = \langle \rP_i \rangle$ is equivalent to the categorical ordinary double point~$\Db(\sA_p)$ and is admissible in~$\cT$.  Moreover, the collection~$\cP_1, \dots, \cP_r$ is semiorthogonal in~$\cT$, the subcategory
\begin{equation*}
\cP \coloneqq \langle \cP_1, \dots, \cP_r \rangle = \langle \rP_1, \dots, \rP_r \rangle
\end{equation*}
is equal to~$\pi_*(\tcP)$ and absorbs singularities of~$\cT$, and the functor~$\pi_*$ induces equivalences
\begin{equation}
\label{eq:nodal-cc-perp}
{}^\perp\tcP \simeq {}^\perp\cP,
\qquad
\tcP^\perp \simeq \cP^\perp.
\end{equation}

\item
\label{item:ccc-and-gorenstein}
The functor~$\pi_*$ is a crepant categorical contraction.
\end{enumerate}
\end{theorem}

\begin{proof}
\ref{item:tcc}~ By Lemma~\ref{lem:kronecker-recogn} the subcategory~$\tcP_i \subset \tcT$ is equivalent to~{$\cKr_{p+1}$} and admissible.  To check that the subcategories~$\tcP_1, \dots, \tcP_r \subset \tcT$ are semiorthogonal, it is enough to note that the generating sets of these subcategories are semiorthogonal; \textit{i.e.}, for all~$i > j$ we have
\begin{equation*}
\Ext^\bullet(\cE_i, \cE_j) = 0,
\qquad
\Ext^\bullet(\cE_i, \rK_j) = 0,
\qquad
\Ext^\bullet(\rK_i, \cE_j) = 0,
\qquad
\Ext^\bullet(\rK_i, \rK_j) = 0.
\end{equation*}
Indeed, the first holds for~$i > j$ since the collection~$\cE_1, \dots, \cE_r$ is exceptional, the second holds for~$i \ne j$ {by~\eqref{eq:dim-ext-ce-ck}}, the third {follows from the second and Serre duality since}~$\rK_i$ is spherical, and the last holds for all~$i \ne j$ since the spherical objects are assumed to be orthogonal.  Since each~$\tcP_i$ is admissible in~$\tcT$, the subcategory~$\tcP$ is admissible as well.

\ref{item:cc}~
Consider the semiorthogonal decompositions
\begin{equation*}
\tcT = \left\langle \tcP_1, \dots, \tcP_r, {}^\perp\tcP \right\rangle,
\qquad
\tcT = \left\langle \tcP^\perp, \tcP_1, \dots, \tcP_r \right\rangle.
\end{equation*}
Applying Proposition~\ref{prop:sod-crepancy} to the subcategories~$\langle \rK_1 \rangle \subset \tcP_1$, \dots, $\langle \rK_r \rangle \subset \tcP_r$ and~$0 \subset {}^\perp\tcP$ or $0 \subset \tcP^\perp$ (and defining~$\cK \subset{} \tcT$ as~$\cK \coloneqq \langle \rK_1, \dots, \rK_r \rangle$), we obtain semiorthogonal decompositions
\begin{equation*}
\tcT / \cK = \left\langle \cP_1, \dots, \cP_r, {}^\perp\cP \right\rangle,
\qquad
\tcT / \cK = \left\langle \cP^\perp, \cP_1, \dots, \cP_r \right\rangle,
\end{equation*}
where $\cP_i = \pi_*(\tcP_i) \simeq \tcP_i / \langle \rK_i \rangle$, $\cP = \pi_*(\tcP) \subset \cT / \cK$ is the subcategory generated by~$\cP_1, \dots, \cP_r$, and the functor~$\pi_*$ induces equivalences~\eqref{eq:nodal-cc-perp}.  Furthermore, we deduce the equivalence~$\cP_i \simeq \Db(\sA_{p})$ from Proposition~\ref{prop:contr-Cp}, and since~$\pi_*\vert_{\tcP_i}$ is isomorphic to the localization functor~$\uprho_*$ from Proposition~\ref{prop:contr-Cp}, we also conclude that~$\rP_i \coloneqq \pi_*(\cE_i)$ is the $\Pinfty{p+1}$-generator of~$\cP_i$.

To show that~$\cP_i$ is right admissible, we consider the semiorthogonal decomposition
\begin{equation*}
\tcT = \left\langle \bS_{\tcT}\left(\tcP_{i+1}\right), \dots, \bS_{\tcT}\left(\tcP_n\right), \tcP^\perp, \tcP_1, \dots, \tcP_i \right\rangle,
\end{equation*}
where~$\bS_\tcT$ is the Serre functor of~$\tcT$, and note that since the objects~$\rK_j$ are spherical, we have
\begin{equation*}
\langle\bS_{\tcT}(\rK_{i+1}), \dots, \bS_{\tcT}(\rK_{n}), 0, \rK_1, \dots, \rK_i \rangle = 
\langle \rK_{i+1}, \dots, \rK_n, 0, \rK_1, \dots, \rK_n \rangle = \cK. 
\end{equation*}
Hence the above arguments provide a semiorthogonal decomposition of~$\cT = \tcT/\cK$, where~$\cP_i = \tcP_i/\langle \rK_i\rangle$ stands on the right, so it is right admissible.  Similarly, using the semiorthogonal decomposition
\begin{equation*}
\tcT = \left\langle \tcP_i, \dots, \tcP_n, {}^\perp\tcP, \bS_{\tcT}^{-1}\left(\tcP_{1}\right), \dots, \bS_{\tcT}^{-1}\left(\tcP_{i-1}\right) \right\rangle
\end{equation*}
and the spherical property of the objects~$\rK_j$, we show that~$\cP_i$ is left admissible.

Finally, $\cP$ absorbs singularities of~$\cT$ because~${}^\perp \cP$ and~$\cP^\perp$ are equivalent to admissible subcategories~{${}^\perp\tcP$ and~$\tcP^\perp$} in the smooth and proper category~$\tcT$, so they are smooth and proper.

\ref{item:ccc-and-gorenstein}~ This is obvious because~$\Ker(\pi_*)$ is generated by spherical objects.
\end{proof}

See Section~\ref{sec:nodal-absorption}, in particular Theorem~\ref{thm:main-nodal-adherence}, for geometric applications of Theorem~\ref{thm:contractions} to nodal varieties.  Meanwhile, just note that in the situation of Example~\ref{ex:surf-node-resol}, if~$\pi \colon \tX \to X$ is the contraction of~$E$, we obtain a fully faithful embedding~$\Db(\sA_0) \hookrightarrow \Db(X)$; \textit{cf.}~\cite[Example~3.17]{KKS20}.

\subsection{Absorption and deformation absorption}
\label{ss:abs-dabs-proofs}

Recall the notion of absorption from Definition~\ref{def:intro-absorption}.  We will need the following observation.  Let~$\omega_X^\bullet$ be the dualizing complex of~$X$.

\begin{lemma}
\label{lemma:absorption-hocolim}
Assume that~$\cP \subset \Db(X)$ absorbs singularities of a projective scheme~$X$, and let~${}^\perp \cP, \cP^\perp \subset \Db(X)$ be its orthogonals in~$\Db(X)$.  Then
\begin{equation*}
{}^\perp \cP \subset \Dp(X)
\qquad\text{and}\qquad
\cP^\perp \subset \Dp(X) \otimes \omega_X^\bullet,
\end{equation*}
and these embeddings are admissible. 
\end{lemma}

\begin{proof}
Let~$F \in {}^\perp\cP$.  Then~$\Ext^\bullet(F,F')$ is finite-dimensional for any~$F' \in {}^\perp\cP$ because~${}^\perp\cP$ is proper, and~$\Ext^\bullet(F,P) = 0$ for any~$P \in \cP$ by the definition of the orthogonal.  Therefore, $\Ext^\bullet(F,G)$ is finite-dimensional for any~$G \in \Db(X)$, and hence~$F$ is perfect by~\cite[Proposition~1.11]{Orl06}, and the first inclusion is proved.  The second inclusion follows from the first and Grothendieck duality.

Finally, since the category~${}^\perp\cP \simeq \cP^\perp$ is smooth, it is equivalent to the derived category of a dg-algebra, see, \textit{e.g.}, \cite[Lemma~2.6]{Toen}, so it has a strong generator by~\cite[Lemmas~3.5 and~3.6(a)]{Lunts}.  Since~${}^\perp\cP \simeq \cP^\perp$ is also proper and idempotent complete, it is saturated by~\cite[Theorem~1.3]{Bondal-vdB}.  Therefore, its image in the proper categories~$\Dp(X)$ and~$\Dp(X) \otimes \omega_X^\bullet$ is admissible.
\end{proof}

\begin{example}
Assume that~$\cP$ absorbs singularities of a projective Gorenstein variety~$X$.  Then~$\omega_X^\bullet$ is a shift of a line bundle, and by Lemma~\ref{lemma:absorption-hocolim} we have~${}^\perp\cP, \cP^\perp \subset \Dperf(X)$.  If in addition~$\cP$ has an admissible semiorthogonal decomposition
\begin{equation}\label{eq:kawamata-dec}
\cP = \langle \Db(R_1), \dots, \Db(R_m) \rangle,
\end{equation}
where the~$R_i$ are finite-dimensional associative algebras, then~$\langle \cP^\perp, \cP \rangle$ is a ``Kawamata semiorthogonal decomposition'' in the sense of~\cite{Kalck-Pavic-Shinder}.  Conversely, a Kawamata semiorthogonal decomposition provides absorption of singularities.  For examples of Kawamata decompositions, see~\cite{Kalck-Pavic-Shinder, Pavic-Shinder-delPezzo, Fei-delPezzo}.
\end{example}

Recall the definitions of smoothing and (thick) deformation absorption from Definitions~\ref{def:intro-smoothing} and~\ref{def:intro-deformation-absorption}, respectively.  Given a smoothing~$f \colon \cX \to B$ of~$X$, we denote by~$\io \colon X \hookrightarrow \cX$ the embedding of the central fiber.  Note that since~$X$ is a hypersurface in a smooth variety~$\cX$, it follows that~$X$ has (at worst) Gorenstein singularities.  Furthermore, since~$X \subset \cX$ is a Cartier divisor with trivial normal bundle, for each~$F \in \Db(X)$ we have the standard distinguished triangle
\begin{equation}
\label{eq:rrr-triangle}
F[1] \lra \io^*\io_*(F) \lra F \lra F[2].
\end{equation}
Recall the definition of a~$\Pinfty{q}$-object (Definition~\ref{def:pinfty}) and its canonical self-extension (Definition~\ref{def:cse}).  Also recall the objects~$\rM^{(i)}$ defined in the proof of Lemma~\ref{lem:Pinf}.

\begin{lemma}
\label{lem:Pinfty-deform}
Let~$X$ be a Gorenstein projective variety with a smoothing~$f \colon \cX \to B$, and let~$\io \colon X \hookrightarrow \cX$ be the embedding of the central fiber.  If\,~$\rP \in \Db(X)$ is a $\Pinfty{q}$-object, then~$q \in \{1,2\}$.  Moreover,
\begin{itemize}
\item if~$q = 2$, the morphism~$\rP \to \rP[2]$ in~\eqref{eq:rrr-triangle} for~$F = \rP$ is given by~$\uptheta$ and~$\io^*\io_*\rP \cong \rM$, and
\item if~$q = 1$, the morphism~$\rP \to \rP[2]$ in~\eqref{eq:rrr-triangle} for~$F = \rP$ is given by~$\uptheta^2$ and~$\io^*\io_*\rP \cong \rM^{(2)}$,
\end{itemize}
where~$\rM$ is the canonical self-extension of\,~$\rP$ and the object~$\rM^{(2)}$ is defined in~\eqref{eq:rmi}.  In both cases~$\rM$ is a perfect complex on~$X$.
\end{lemma}

\begin{proof}
First, we note that the smoothness of~$\cX$ implies that~$\io_*\rP \in \Db(\cX)$ is a perfect complex, so
\begin{equation*}
\io^*\io_*\rP \in \Dp(X).
\end{equation*}
On the other hand, consider the triangle~\eqref{eq:rrr-triangle} for~$F = \rP$.  Since~$\rP$ itself is not perfect (because~$\Ext^*(\rP, \rP)$ is infinite-dimensional), it follows that the connecting morphism~$\rP \to \rP[2]$ is nonzero.  By the definition of a~$\Pinfty{q}$-object, it follows that~$q \le 2$ and, moreover, if~$q = 2$, the connecting morphism is isomorphic to~$\uptheta$, and if~$q = 1$, it is isomorphic to~$\uptheta^2$.  Finally, if~$q = 2$, it follows that~$\rM \cong \io^* \io_* \rP$ is perfect, and if~$q = 1$, it follows that~$\rM^{(2)} \cong \io^* \io_* \rP$ is perfect, so~$\rM$ is also perfect by Corollary~\ref{cor:sm2}.
\end{proof}

We would like to emphasize the following consequence. 

\begin{corollary}
\label{cor:no-pinfty-big}
If\,~$X$ is a smoothable projective variety and~$\rP \in \Db(X)$ is a $\Pinfty{q}$-object, then~$q \in \{1,2\}$.
\end{corollary}

We are ready to prove Theorem~\ref{thm:intro-deformation-absorption-to-smoothing}.  Recall that for a variety~$\cX/B$ a subcategory~$\cD \subset \Db(\cX)$ is called \emph{$B$-linear} if it is closed under tensor products with pullbacks of perfect complexes on~$B$.  Given a $B$-linear triangulated category~$\cD \subset \Db(\cX)$ and a morphism~$\varphi \colon B' \to B$, the base change category~$\cD_{B'} \subset \Db(\cX \times_B B')$ is defined in~\cite[Theorem~5.6]{K11} under appropriate assumptions (admissibility of~$\cD$ and of either of its orthogonals and finiteness of cohomological amplitude of the projection functor to~$\cD$) as
\begin{equation}
\label{eq:base-change}
\cD_{B'} \coloneqq \hocolim \left\langle \tphi^*(\cD \cap \Dp(\cX)) \right\rangle \cap \Db(\cX \times_B B'),
\end{equation}
where~$\tphi \colon \cX\times_B B' \to \cX$ is the morphism induced by~$\varphi$.  In other words, first one considers ``the perfect part''~$\cD \cap \Dp(\cX)$ of~$\cD$, then one considers the triangulated hull of the pullback~$\tphi^*(\cD \cap \Dp(\cX))$ along~$\tphi$ (this gives ``the perfect part'' of the base change category~$\cD_{B'}$), and finally one considers all homotopy colimits contained in~$\Db(\cX\times_B B')$ of chains of morphisms of objects in~$\langle\tphi^*(\cD \cap \Dp(\cX))\rangle$.

In the special case where~$B'$ is a point~$b \in B$, this defines the \emph{fiber}~$\cD_b$ of~$\cD$.

\begin{proof}[Proof of Theorem~\textup{\ref{thm:intro-deformation-absorption-to-smoothing}}]
By the definition of {deformation absorption, the subcategory~$\thick(\io_*\cP) \subset \Db(\cX)$ is admissible,} so we have a semiorthogonal decomposition
\begin{equation}
\label{eq:sod-dbcx}
\Db(\cX) = \langle \thick(\io_*\cP), \cD \rangle.
\end{equation}
Its first component is supported set-theoretically on the central fiber of~$\cX$ over~$B$, so it is~$B$-linear.  By~\cite[Lemma~2.36]{K06} the second component~$\cD \coloneqq {}^\perp(\thick(\io_*\cP))$ is~$B$-linear as well.  The components of~\eqref{eq:sod-dbcx} are both admissible because~$\cX$ is smooth over~$\kk$ and proper over~$B$, and their projection functors have finite cohomology amplitude by~\cite[Proposition~2.5]{K08} since~$\cX$ is smooth and quasiprojective.  Therefore, by~\cite[Theorem~5.6]{K11} we can talk about base change of the $B$-linear category~$\cD$; in particular, the fibers~$\cD_b$ and~$\cD_o$ are defined.

Now the equality~$\cD_b = \Db(\cX_b)$ for~$b \ne o$ follows from base change along the point embedding~$\{b\} \hookrightarrow B$ applied to~\eqref{eq:sod-dbcx} since the base change of the first component of~\eqref{eq:sod-dbcx} is zero (because the support of any object in it does not intersect the fiber~$\cX_b$).  Similarly, by base change we have a semiorthogonal decomposition
\begin{equation*}
\Db(X) = \langle \thick(\io_*\cP)_o, \cD_o \rangle.
\end{equation*}
Thus, we need to identify its first component with~$\cP$.

On the one hand, since~$\io$ is the embedding of a fiber, \cite[Corollary~5.7]{K11} shows that
\begin{equation*}
\thick(\io_*\cP)_o = \{ F \in \Db(X) \mid \io_*F \in \thick(\io_*\cP) \}.
\end{equation*}
It follows immediately that~$\cP \subset \thick(\io_*\cP)_o$.

On the other hand, triangle~\eqref{eq:rrr-triangle} applied to an object~$F \in \cP$ shows that~$\io^*\io_*F \in \cP$.  Since the functor~$\io^*$ is triangulated, we have~$\io^*G \in \cP$ for any~$G$ in~$\langle \io_*\cP \rangle$, and since~$\cP$ is idempotent complete, we have $\io^*G \in \cP$ for any~$G \in \thick(\io_*\cP)$.  This proves that~$\io^*(\thick(\io_*\cP) \cap \Dp(\cX)) = \io^*(\thick(\io_*\cP))$ is contained in~$\cP$.  Furthermore, since~$\cP$ is closed under homotopy colimits contained in~$\Db(X)$ (because by Lemma~\ref{lemma:absorption-hocolim} the category~$\cP$ is an orthogonal to a subcategory consisting of perfect objects), we have the inclusion
\begin{equation*}
\hocolim(\cP) \cap \Db(X) \subset \cP.
\end{equation*}
Combining the two inclusions proved above with~\eqref{eq:base-change}, we see that~$\thick(\io_*\cP)_o \subset \cP$, and comparing it with the opposite inclusion proved above, we obtain the equality~$\cP = \thick(\io_*\cP)_o$.

Finally, the smoothness and properness of~$\cD$ over~$B$ follow from~\cite[Theorem~2.10]{K21}.
\end{proof}

\begin{proof}[Proof of Theorem~\textup{\ref{thm:intro-pinfty2}}]
Let~$f \colon \cX \to B$ be a smoothing of~$X$, and recall that~$\io \colon X \to \cX$ denotes the embedding of the central fiber.  Consider the distinguished triangle~\eqref{eq:rrr-triangle} for~$F = \rP_i$.  Applying Lemma~\ref{lem:Pinfty-deform} we conclude that
\begin{equation*}
\io^*\io_*(\rP_i) \cong \rM_i
\end{equation*}
is the canonical self-extension of~$\rP_i$.  As a consequence, using adjunction and~\eqref{eq:ext-cm-cr}, we obtain
\begin{equation*}
\Ext^\bullet(\io_*(\rP_i), \io_*(\rP_i)) \cong
\Ext^\bullet(\io^*\io_*(\rP_i), \rP_i) \cong
\Ext^\bullet(\rM_i, \rP_i) \cong \kk,
\end{equation*}
which means that~$\io_*(\rP_i) \in \Db(\cX)$ is an exceptional object.  Similarly, if~$i > j$, we have
\begin{equation*}
\Ext^\bullet(\io_*(\rP_i), \io_*(\rP_j)) \cong
\Ext^\bullet(\io^*\io_*(\rP_i), \rP_j) \cong
\Ext^\bullet(\rM_i, \rP_j),
\end{equation*}
and using the triangle~\eqref{def:cm} for~$\rM_i$ and~$\rP_i$, we see that~$\Ext^\bullet(\rM_i, \rP_j) = 0$.  Thus, $\io_*(\rP_1), \dots, \io_*(\rP_r)$ is an exceptional collection of compactly supported objects in~$\Db(\cX)$, so
\begin{equation*}
\langle \io_*\cP \rangle = \langle \io_*(\rP_1), \dots, \io_*(\rP_r) \rangle \subset \Db(\cX)
\end{equation*}
is admissible in~$\Db(\cX)$, and therefore, $\cP$ provides a deformation absorption for~$X$.  Since this is true for any smoothing of~$X$, the subcategory~\mbox{$\cP \subset \Db(X)$} provides a universal deformation absorption.
\end{proof}

\begin{remark}
The fact that~$\io_*\rP$ is an exceptional object on the total space of a smoothing is an analogue and the limiting case of the fact that the direct image of a~$\P^n$-object with respect to the embedding into the total space of an appropriate deformation is a $(2n+1)$-spherical object; see~\cite{Huybrechts-Thomas}.
\end{remark}

Combining Theorem~\ref{thm:intro-pinfty2} with Theorem~\textup{\ref{thm:contractions}}, we obtain the following. 

\begin{corollary}
\label{cor:deformation-absorption-by-codp}
Under the conditions of Theorem~\textup{\ref{thm:contractions}}, if\,~$\cT \simeq \Db(X)$ for a projective variety~$X$ and~$p = 1$, then the category~$\cP$ constructed therein provides a universal deformation absorption of singularities of\,~$X$.
\end{corollary}

Finally, we prove Theorem~\ref{thm:intro-pinfty1}.

\begin{proof}[Proof of Theorem~\textup{\ref{thm:intro-pinfty1}}]
Let~$f \colon \cX \to B$ be a smoothing of~$X$, and recall that~$\io \colon X \to \cX$ denotes the embedding of the central fiber.  Consider the stack~$\cD^\bb_{\mathrm{pug}}(\cX/B)$ of universally gluable~$B$-perfect complexes on~$\cX$ as defined in~\cite[Section~2.1]{lieblich}.  Note that under the assumption of smoothness of the base~$B$, an object~$\cF \in\Dqc(\cX)$ is $B$-perfect if and only if~$\cF \in \Db(\cX)$. Also recall that the ``gluability condition'' for an object~$\cF \in \Db(\cX)$ is just
\begin{equation*}
f_*\cRHom(\cF,\cF) \in \Db(B)^{\ge 0},
\end{equation*}
where the right-hand side stands for the subcategory of objects with zero sheaf cohomology in negative degrees, and ``universal gluability'' is gluability after arbitrary base change.

For each~\mbox{$1 \le i \le r$} consider the canonical self-extension~$\rM_i \in \Db(X)$ of the $\Pinfty{1}$-object~$\rP_i \in \Db(X)$.  By Lemma~\ref{lem:Pinfty-deform} we have~$\rM_i \in \Dp(X)$, and by Lemma~\ref{lem:ext-cm} we have
\begin{equation*}
\Ext^\bullet(\rM_i, \rM_i) \cong \sA_0 \cong \kk[\eps]/\eps^2,
\qquad \deg(\eps) = 0,
\end{equation*}
so~$\rM_i$ gives a $\kk$-point of the stack~$\cD^\bb_{\mathrm{pug}}(\cX/B)$ over the point~$o \in B$.  Moreover, the stack~$\cD^\bb_{\mathrm{pug}}(\cX/B)$ is locally of finite presentation by~\cite[Theorem~4.2.1]{lieblich}, and it follows from~\cite[Theorem~3.1.1]{lieblich} that the natural morphism of stacks
\begin{equation*}
\cD^\bb_{\mathrm{pug}}(\cX/B) \lra B
\end{equation*}
is smooth at the point~$[\rM_i]$.  Therefore, \'etale locally this morphism admits a section passing through~$[\rM_i]$.  In other words, after an \'etale base change, there is an object~$\cM_i \in \Db(\cX)$ such that
\begin{equation*}
\io^*(\cM_i) \cong \rM_i,
\end{equation*}
and applying~\cite[Corollary~2.12]{K21}, we conclude that, shrinking~$B$, we may assume~$\cM_i \in \Dp(\cX)$.

Now consider the object
\begin{equation*}
\cR_i \coloneqq f_*\cRHom(\cM_i,\cM_i) \cong f_*\left(\cM_i \otimes_{\cO_\cX} \cM_i^\vee\right) \in \Db(B).
\end{equation*}
Since the morphism~$f$ is flat, using base change we obtain an isomorphism
\begin{equation*}
\cR_i\vert_o \cong 
(f_*\cRHom(\cM_i,\cM_i))\vert_o \cong 
\rH^\bullet(X, \io^*\cRHom(\cM_i,\cM_i)) \cong
\rH^\bullet(X, \cRHom(\rM_i,\rM_i)) \cong
\RHom(\rM_i,\rM_i),
\end{equation*}
where the left side is the derived restriction of~$\cR_i$ to the point~$o$ and the right side is isomorphic to~$\sA_0$; in particular, it is a 2-dimensional vector space sitting in degree zero.  Therefore, shrinking~$B$ further, we may assume that~$\cR_i$ is a locally free sheaf of rank~$2$.

By construction~$\cR_i$ is a sheaf of~$\cO_B$-algebras.  Consider the adjoint pair of $B$-linear functors
\begin{align*}
\Phi_{\cM_i} &\colon \bD(B,\cR_i) \lra \bD(\cX), \qquad
\cF \longmapsto f^*\cF \otimes_{f^*\cR_i} \cM_i,\\
\Phi^!_{\cM_i} &\colon \bD(\cX) \lra \bD(B,\cR_i), \qquad
\cG \longmapsto f_*\cRHom(\cM_i,\cG).
\end{align*}
Using the containment~$\cM_i \in \Dp(\cX)$, the projection formula, and the definition of~$\cR_i$, we compute
\begin{equation*}
f_*\cRHom\left(\cM_i, f^*\cF \otimes_{f^*\cR_i} \cM_i\right) \cong
f_*\left(f^*\cF \otimes_{f^*\cR_i} \cM_i \otimes_{\cO_\cX} \cM_i^\vee\right) \cong
\cF \otimes_{\cR_i} f_*\left(\cM_i \otimes_{\cO_\cX} \cM_i^\vee\right) \cong
\cF \otimes_{\cR_i} \cR_i \cong
\cF
\end{equation*}
for any~$\cF \in \bD(B, \cR_i)$.  Thus, $\Phi^!_{\cM_i} \circ \Phi_{\cM_i} \cong \id$, so~$\Phi_{\cM_i}$ is fully faithful.

Also note that the functor~$\Phi_{\cM_i}$ is $B$-linear and takes the free module~$\sA_0 \otimes \cO_o$ at point~$o \in B$ to~$\io_*\rM_i$.  Using the notation of~Section~\ref{sec:categorical-odp} and a simple induction, it is easy to check that~$\Phi_{\cM_i}(\sA_0^{(j)} \otimes \cO_o) \cong \io_*\rM_i^{(j)}$.  Now, since the functors~$\Phi_{\cM_i}$ and~$\io_*$ commute with direct sums, and hence also with homotopy colimits, we conclude that
\begin{equation*}
\Phi_{\cM_i}(\kk_\sA \otimes \cO_o) \cong
\Phi_{\cM_i}\left(\hocolim \sA_0^{(j)} \otimes \cO_o\right) \cong
\hocolim \io_*\rM_i^{(j)} \cong
\io_*\rP_i,
\end{equation*}
where we used Lemma~\ref{lemma:ap-k}\ref{item:hocolim-ap-k} and an argument of Lemma~\ref{lem:Pinf}.

Next, we interpret this computation geometrically.  Since the rank of the sheaf of~$\cO_B$-algebras~$\cR_i$ is~$2$, the sheaf is commutative, so we can consider the $B$-scheme~$\cZ_i \coloneqq \Spec_B(\cR_i)$.  Then~$\bD(\cZ_i) \simeq \bD(B, \cR_i)$ and the simple module~$\kk_\sA \otimes \cO_o$ corresponds to the structure sheaf of the unique point~$z_i \in \cZ_i$ over~$o$.  Abusing notation, we will write~$\Phi_{\cM_i}$ for the induced functor~$\bD(\cZ_i) \simeq \bD(B, \cR_i) \to \bD(\cX)$ and~$\Phi^!_{\cM_i}$ for its adjoint.  Then the above computation gives an isomorphism~$\Phi_{\cM_i}(\cO_{z_i}) \cong \io_*\rP_i$.  Using the full faithfulness of the functor~$\Phi_{\cM_i}$, Lemma~\ref{lem:Pinfty-deform}, and Corollary~\ref{cor:ext-rmi-rp}, we therefore obtain
\begin{equation*}
\Ext^\bullet_{\cZ_i}\left(\cO_{z_i}, \cO_{z_i}\right) \cong 
\Ext^\bullet_{\cX}(\io_*\rP_i, \io_*\rP_i) \cong
\Ext^\bullet_X(\io^*\io_*\rP_i, \rP_i) \cong
\Ext^\bullet_X\left(\rM_i^{(2)}, \rP_i\right) \cong
\kk[\uptheta]/\uptheta^2.
\end{equation*}
Since~$\deg(\uptheta) = 1$, we conclude that the scheme~$\cZ_i$ is smooth at~$z_i$, so, shrinking~$B$ further, we may assume that~$\cZ_i$ is smooth everywhere.

Finally, since both~$\cZ_i$ and~$\cX$ are smooth over~$\kk$ and proper over~$B$, the restricted Fourier--Mukai functor~$\Phi_{\cM_i}\vert_{\Db(\cZ_i)}$ has both adjoints, and therefore the subcategory~$\Phi_{\cM_i}(\Db(\cZ_i)) \subset \Db(\cX)$ is admissible.  Since the objects~$\rP_i$, $1 \le i \le n$, are semiorthogonal, shrinking~$B$ further, we may assume that the subcategories~$\Phi_{\cM_i}(\Db(\cZ_i)) \subset \Db(\cX)$ are semiorthogonal (see the argument of~\cite[Proposition~2.16]{FK18}).  Thus, we obtain a $B$-linear semiorthogonal decomposition
\begin{equation}
\label{eq:dbcx-dbcz-cd}
\Db(\cX) = \langle \Db(\cZ_1), \dots, \Db(\cZ_r), \cD \rangle,
\end{equation}
where the last component~$\cD$ is the orthogonal complement of the other components.

Now, consider the base change of this decomposition to any geometric point~$b \ne o$.  We obtain
\begin{equation*}
\Db\left(\cX_b\right) = \left\langle \Db\left((\cZ_1)_b\right), \dots, \Db\left((\cZ_r)_b\right), \cD_b \right\rangle.
\end{equation*}
On the one hand,~$\Db((\cZ_i)_b) \simeq \Db((\cR_i)_b)$, where~$(\cR_i)_b$ is a 2-dimensional algebra.  On the other hand, since this is a semiorthogonal component in the derived category of a smooth projective variety~$\cX_b$, the algebra should have finite homological dimension, so it should be \'etale over the residue field of the point~$b$.  Therefore, $\cZ_i \to B$ is \'etale over~$B \setminus \{o\}$.  Moreover, $\cD_b$ is also a smooth and proper category.

Furthermore, by construction the base change of~$\Db(\cZ_i)$ to the point~$o$ coincides with the subcategory~$\langle \rP_i \rangle \subset \Db(X)$, so~$\cD_o \simeq {}^\perp\cP$; in particular, it is smooth and proper.  Finally, we use~\cite[Theorem~2.10]{K21} to conclude that~$\cD$ is smooth and proper over~$B$.

It remains to prove the uniqueness of~\eqref{eq:dbcx-dbcz-cd}.  For this we note that if~$\cF \in \Db(\cX)$ is an object such that~$\io^*\cF \in \Db(\cZ_i)_o$ for some~$1 \le i \le r$, then there is a Zariski neighbourhood~$U \subset B$ of~$o$ such that the pullback of~$\cF$ to~$\cX_U = \cX \times_B U$ is contained in~$\Db(\cZ_i \times_B U) \subset \Db(\cX_U)$.  Indeed, if~$\cF_j \in \Db(\cZ_j)$ and~$\cF_\cD \in \cD$ are the components of~$\cF$ with respect to~\eqref{eq:dbcx-dbcz-cd}, then the assumption means that
\begin{equation*}
\io^*\cF_j = 0 \quad 
\text{for $j \ne i$}
\qquad\text{and}\qquad 
\io^*\cF_\cD = 0.
\end{equation*}
Then the objects~$\cF_j$ for~$j \ne i$ and~$\cF_\cD$ vanish on a Zariski neighbourhood of~$X \subset \cX$.  But since~$f \colon \cX \to B$ is proper, any Zariski neighbourhood of~$X$ contains~$\cX_U$ for an appropriate Zariski neighbourhood~$U \subset B$ of~$o$.  Therefore, the pullbacks of~$\cF_j$ for~$j \ne i$ and~$\cF_\cD$ to~$\cX_U$ vanish, and hence the pullback of~$\cF$ to~$\cX_U$ is equal to the pullback of~$\cF_i$.  Thus, the pullback of~$\cF$ to~$\cX_U$ is contained in~$\Db(\cZ_i \times_B U)$.

Now, if
\begin{equation*}
\Db(\cX) = \left\langle \Db(\cZ'_1), \dots, \Db(\cZ'_r), \cD' \right\rangle
\end{equation*}
is another $B$-linear semiorthogonal decomposition such that its base change to~$o$ gives the decomposition~$\Db(X) = \langle \rP_1, \dots, \rP_r, {}^\perp\cP \rangle$, then applying the above observation to generators~$\cG_1,\dots,\cG_r$ of the components~$\Db(\cZ'_1),\dots,\Db(\cZ'_r)$, we find Zariski neighbourhoods~$U_1,\dots,U_r \subset B$ of~$o$ such that the pullback of~$\cG_i$ to~$\cX \times_B U_i$ is contained in~$\Db(\cZ_i \times_B U_i)$.  Taking~$U = U_1 \cap \dots \cap U_r$, we deduce that
\begin{equation*}
\Db\left(\cZ'_i \times_B U\right) \subset \Db(\cZ_i \times_B U)
\end{equation*}
for all~$i$.  The same argument shows the opposite inclusion (possibly after shrinking the base further), and therefore we finally obtain an equality~$\Db(\cZ'_i \times_B U) = \Db(\cZ_i \times_B U)$.
\end{proof}


\section{Crepant categorical resolutions for nodal varieties}
\label{sec:nodal-varieties}

In this section we construct crepant categorical resolutions for varieties with ordinary double points or \emph{nodal varieties}.  In~Section~\ref{ss:bo-localization} we consider the blowup morphism~$\pi \colon \Bl_x(X) \to X$ for an ordinary double point~\mbox{$x \in X$} and check that the pushforward functor~$\pi_* \colon \Db(\Bl_x(X)) \to \Db(X)$ is a Verdier localization.  In~Section~\ref{ss:ccc-odp}, we consider a nodal variety with singular points~$x_1, \dots, x_r \in X$ and find an admissible subcategory~$\cD \subset \Db(\Bl_{x_1, \dots, x_r}(X))$ such that~$\pi_*\vert_\cD \colon \cD \to \Db(X)$ is a crepant categorical resolution and~$\Ker(\pi_*\vert_\cD)$ is generated by a completely orthogonal collection of spherical objects.

Starting from~Section~\ref{ss:ccc-odp}, we work over an algebraically closed field~$\kk$ of characteristic not equal to~$2$.

\subsection{Bondal--Orlov localization}
\label{ss:bo-localization}

It is a conjecture of Bondal and Orlov (see~\cite[Section~5]{BO02} and~\cite[Conjecture~1.9]{Ef20}) that if~$X$ is a variety with rational singularities and~$\pi \colon \tX \to X$ is a resolution, then {the pushforward functor}~$\pi_* \colon \Db(\tX) \to \Db(X)$ is a Verdier localization.  The main advance in this direction was obtained in~\cite{Ef20}; we prove a slight extension of his results in a convenient form.

Recall from Definition~\ref{def:cc} the notion of a categorical contraction.

\begin{lemma}
\label{lem:lon-contr}
A functor~$\pi_* \colon \tcT \to \cT$ is a Verdier localization if and only if it is a categorical contraction and the induced map on the Grothendieck groups~$\rK_0(\tcT) \to \rK_0(\cT)$ is surjective.
\end{lemma}

\begin{proof}
If~$\pi_*$ is a categorical contraction and~$\rK_0(\tcT) \to \rK_0(\cT)$ is surjective, then~\mbox{$\rK_0(\Ima(\pi_*)) = \rK_0(\cT)$} and by Thomason's theorem on classification of dense subcategories~\cite[Theorem~2.1]{Thomason}, we have the equivalence~$\tcT / \Ker(\pi_*) \simeq \Ima(\pi_*) = \cT$.  The other implication is obvious.
\end{proof}

If~$X$ is a variety with rational singularities over a field~$\kk$ of characteristic zero, $\pi \colon \tX \to X$ is a resolution of singularities, and~$\pi_* \colon \Db(\tX) \to \Db(X)$ is the pushforward functor, the surjectivity of the induced map of the Grothendieck groups was proved recently in~\cite[Theorem~1.2]{MauriShinder}.  Consequently, the functor~$\pi_*$ is a Verdier localization if and only if it is a categorical contraction; see~\cite[Corollary~1.3]{MauriShinder}.

We denote by
\begin{equation*}
\rG_0(X) \coloneqq \rK_0(\Db(X))
\end{equation*}
the Grothendieck group of the bounded derived category of~$X$.

\begin{theorem}[\textit{cf.}~\cite{Ef20}]
\label{thm:efimov}
Let~$\pi \colon \tX \to X$ be a proper birational morphism.  Let~$\upzeta \colon Z \hookrightarrow X$ be a closed subscheme such that~$E \coloneqq \pi^{-1}(Z)$ is a Cartier divisor and the restriction~$\pi \colon \tX \setminus E \to X \setminus Z$ is an isomorphism.  Assume the following isomorphisms for the derived pushforwards of the sheaves~$\cO_{\tX}(-mE)$:
\begin{equation}
\label{eq:pis-jm}
\pi_* \cO_{\tX}(-mE) \cong \cJ^m_{Z}
\qquad 
\text{for all~$m \ge 0$.}
\end{equation}
Let $p \colon E \to Z$ and $\upeta \colon E \hookrightarrow \tX$ be the natural morphisms so that we have a Cartesian square
\begin{equation*}
\xymatrix{
E \ar[d]_p \ar@{^{(}->}[r]^\upeta & \tX \ar[d]^\pi \\
Z \ar@{^{(}->}[r]^\upzeta & X\rlap{.}
}
\end{equation*}
If~$p_* \colon \Db(E) \to \Db(Z)$ is a Verdier localization or a categorical contraction, then~$\pi_* \colon \Db(\tX) \to \Db(X)$ is also a Verdier localization or a categorical contraction, respectively.  In both cases the category~$\Ker(\pi_*)$ is generated by~$\upeta_*(\Ker(p_*))$.
\end{theorem}

\begin{proof}
Assume that~$p_*$ is a categorical contraction.  Then the conditions of~\cite[Theorem~8.22]{Ef20} (recall the difference in terminology mentioned in Remark~\ref{rem:disclaimer}) are satisfied, so that~$\pi_*$ is also a categorical contraction and~$\Ker(\pi_*)$ is thickly generated by~$\upeta_*(\Ker(p_*))$.

To show that $\Ker(\pi_*)$ is generated by $\upeta_*(\Ker(p_*))$ as a triangulated category, by Thomason's theorem (see the proof of Lemma \ref{lem:lon-contr}) it suffices to check that the induced map on the Grothendieck groups
\begin{equation}\label{eq:K0-eta}
\rK_0(\Ker(p_*)) \xrightarrow{\ \upeta_*\ } \rK_0(\Ker(\pi_*))
\end{equation}
is surjective.  For this we consider the diagram of categories
\begin{equation*}
\xymatrix{
\Ker(p_*) \ar[d]_{\upeta_*} \ar[r] & 
\Db(E) \ar[d]_{\upeta_*} \ar[r]^{p_*} & 
{\Db(Z)} \ar[d]_{\upzeta_*} 
\\
\Ker(\pi_*) \ar[r] & 
\Db_E(\tX) \ar[r]^{\pi_*} & 
{\Db_Z(X)},
}
\end{equation*}
where~$\Db_E(\tX) \subset \Db(\tX)$ and~$\Db_Z(X) \subset \Db(X)$ are the full subcategories of objects set-theoretically supported on~$E$ and~$Z$, respectively.  Its rows are ``exact sequences'', \textit{i.e.}, the left arrows are fully faithful, and the right arrows are categorical contractions.  Indeed, for the top row this is the assumption of the theorem, and for the bottom row this is proved in~\cite[Theorem~8.22(2)]{Ef20}.  Now we apply Schlichting's K-theory machinery as explained in \cite[Section~1.2]{PS18} to this diagram; in this way we obtain a commutative diagram of K-groups
\begin{equation*}
\xymatrix{
\rG_1(E) \ar[d]_{\upeta_*} \ar[r]^{p_*} & 
\rG_1(Z) \ar[d]_{\upzeta_*} \ar[r] &
\rK_0(\Ker(p_*)) \ar[d]_{\upeta_*} \ar[r] & 
\rG_0(E) \ar[d]_{\upeta_*} \ar[r]^{p_*} & 
\rG_0(Z) \ar[d]_{\upzeta_*} 
\\
\rK_1(\Db_E(\tX)) \ar[r]^{\pi_*} & 
\rK_1(\Db_Z(X)) \ar[r] &
\rK_0(\Ker(\pi_*)) \ar[r] & 
\rK_0(\Db_E(\tX)) \ar[r]^{\pi_*} & 
\rK_0(\Db_Z(X))\rlap{.} 
}
\end{equation*}
Applying Quillen's devissage (see~\cite[Section~5, Theorem~4]{Q73} and~\cite[Lemma~2.1]{Orl11}), we conclude that the first two and the last two vertical arrows in the diagram are isomorphisms, so the middle arrow~\eqref{eq:K0-eta} is also an isomorphism.

Assume that~$p_*$ is a Verdier localization.  By Lemma~\ref{lem:lon-contr} we only need to show that the map
\begin{equation}
\label{eq:pis-g0}
\pi_* \colon \rG_0(\tX) \lra \rG_0(X)
\end{equation}
is surjective.  For any coherent sheaf~$\cF$ on~$X$ the natural morphism~$\cF \to \rR^0\pi_*(\rL_0\pi^*\cF) \to \pi_*(\rL_0\pi^*\cF)$ is an isomorphism on the complement of~$Z$, so it is enough to show that the classes of sheaves set-theoretically supported at~$Z$ are in the image of~\eqref{eq:pis-g0}.  Since any such sheaf is filtered by sheaves scheme-theoretically supported on~$Z$ and~$p_* \colon \rG_0(E) \to \rG_0(Z)$ is surjective (again by Lemma~\ref{lem:lon-contr}), we conclude that~\eqref{eq:pis-g0} is surjective as well.
\end{proof}

Let~$x \in X$ be a $\kk$-point which is a normal isolated singularity, and let~$\pi \colon \tX \coloneqq \Bl_x(X) \to X$ be the blowup of~$x \in X$ with exceptional divisor~$E \subset X$; we denote by~$\cO_E(1) \coloneqq \cO_E(-E)$ its conormal bundle.  We will use the following definition to verify the assumption~\eqref{eq:pis-jm}.

\begin{definition}
\label{def:acyclic-conical}
We say that a normal isolated singularity~$x \in X$ with the maximal ideal~$\fm_{X,x} \subset \cO_{X,x}$ is \emph{acyclic projectively normal} if the following conditions hold for all~$m \ge 0$:
\begin{aenumerate}
\item
\label{item:h0e}
The canonical map~$\fm_{X,x}^m / \fm_{X,x}^{m+1} \to \rH^0(E, \cO_E(m))$ is an isomorphism.
\item
\label{item:hie}
The vanishining $\rH^i(E, \cO_E(m)) = 0$ holds for all~$i > 0$.
\end{aenumerate}
\end{definition}

\begin{example}
\label{ex:apn}
If~$X$ is a cone over a smooth Fano complete intersection of positive dimension, or more generally~$x \in X$ is an isolated singular point such that the projective tangent cone~$\rC_{X,x}$ is a Fano complete intersection, then~$x \in X$ is acyclic projectively normal.  Indeed, complete intersections are projectively normal; combining this with~\cite[Exercise~II.5.14]{Hartshorne}, we obtain~\ref{item:h0e}.  The vanishing in~\ref{item:hie} follows easily from the standard computation of cohomology of line bundles on a complete intersection; \textit{cf.}~the proof of Lemma~\ref{lemma:odp-ac}.
\end{example}

\begin{lemma}
\label{lem:normal-cone}
Let~$x \in X$ be an acyclic projectively normal singularity.  Then~$\pi_*(\cO_\tX) \cong \cO_X$; in particular, if~$\tX$ is smooth, then~$(X,x)$ is a rational singularity.  Moreover, condition~\eqref{eq:pis-jm} is satisfied for~$Z = \{x\}$.
\end{lemma}

\begin{proof}
Let~$m \ge 1$.  The structure sheaf~$\cO_{mE} \coloneqq \cO_\tX/\cI_E^m$ of the~$\supth{m}$ infinitesimal neighbourhood of the exceptional divisor~$E$ of the blowup has a filtration with factors~$\cI_E^{k}/\cI_E^{k+1} \cong \cO_E(k)$, {$0 \le k \le m - 1$}; all these sheaves have no higher cohomology by Definition~\ref{def:acyclic-conical}\ref{item:hie}, so 
\begin{equation*}
\rH^{>0}\left(\tX, \cO_{mE}\right) = 0
\end{equation*}
for all~$m \ge 1$.  It follows from the formal functions theorem~\cite[Theorem~III.11.1]{Hartshorne} that
\begin{equation*}
\rR^{>0} \pi_*\left(\cO_\tX\right) = 0.
\end{equation*}
Since~$X$ is normal near~$x$ and~$\pi$ is an isomorphism over the complement of~$x$, we have~$\rR^0 \pi_*(\cO_\tX) \cong \cO_X$, so that~$\pi_*(\cO_\tX) \cong \cO_X$.

Let~$\cJ_{X,x} \subset \cO_X$ be the ideal sheaf of~$x$.  We now prove that~$\pi_*(\cO_{{\tX}}(-mE)) \cong \cJ_{X,x}^m$ by induction on~$m \ge 0$.  The base of induction, $m = 0$, is proved above.  Now assume that the claim is proved for some~$m \ge 0$, and consider the exact sequence
\begin{equation*}
0 \lra \cO_{{\tX}}(-(m+1)E) \lra \cO_{{\tX}}(-mE) \lra \cO_E(m) \lra 0.
\end{equation*}
Applying~$\pi_*$ and using the induction hypothesis and Definition~\ref{def:acyclic-conical}\ref{item:hie}, we obtain a distinguished triangle
\begin{equation*}
\pi_*(\cO_{{\tX}}(-(m+1)E)) \lra \cJ_{X,x}^m \lra \rH^0(E, \cO_E(m)) \otimes \cO_{x}.
\end{equation*}
The second arrow in it factors as the composition
\begin{equation*}
\cJ_{X,x}^m \longtwoheadrightarrow 
\cJ_{X,x}^m / \cJ_{X,x}^{m+1} = 
\fm_{X,x}^m / \fm_{X,x}^{m+1} \lra 
\rH^0(E, \cO_E(m)).
\end{equation*}
By Definition~\ref{def:acyclic-conical}\ref{item:h0e} the last arrow is an isomorphism, so~$\pi_*(\cO_{{\tX}}(-(m+1)E)) \cong \cJ_{X,x}^{m+1}$.
\end{proof}

Now we deduce the Bondal--Orlov localization conjecture for acyclic projectively normal singularities.  Note that when~$x$ is a $\kk$-point, Definition~\ref{def:acyclic-conical} for~$m = 0$ ensures that~$\cO_E$ is an exceptional object in~$\Db(E)$, so we have a semiorthogonal decomposition~$\Db(E) = \langle \cO_E^\perp, \cO_E \rangle$; its first component~$\cO_E^\perp$ plays the crucial role.

\begin{corollary}
\label{cor:bo-localization}
Let~$\pi \colon \tX = \Bl_{x}(X) \to X$ be the blowup of an acyclic projectively normal $\kk$-point~$x \in X$.  Then the functor~$\pi_* \colon \Db(\tX) \to \Db(X)$ is a Verdier localization, and the subcategory~$\Ker(\pi_*) \subset \Db(\tX)$ is generated by~$\upeta_*(\cO_E^\perp) \subset \Db(\tX)$, where~$\upeta \colon E \to \tX$ is the embedding of the exceptional divisor.

Moreover, the map~$\pi_* \colon \rG_0(\tX) \to \rG_0(X)$ is surjective with kernel generated by~$\upeta_*(\rK_0(\cO_E^\perp)) \subset \rG_0(\tX)$.
\end{corollary}

\begin{proof}
Since~$\cO_E \in \Db(E)$ is an exceptional object, the pushforward~$p_*\colon \Db(E) \to \Db(\Spec(\kk))$ is a Verdier localization with~$\Ker(p_*) = \cO_E^\perp$.  By Lemma~\ref{lem:normal-cone} the conditions of Theorem~\ref{thm:efimov} are satisfied for~$\pi$, so~$\pi_*$ is a Verdier localization with kernel generated by~$\upeta_*(\cO_E^\perp)$.  Finally, the statement about~$\rG_0$ follows from the localization exact sequence
\begin{equation*}
\rK_0(\Ker(\pi_*)) \lra \rK_0\left(\Db(\tX)\right) \lra \rK_0\left(\Db(X)\right) \lra 0;
\end{equation*}
see~\cite[Proposition~VIII.3.1]{SGA5}.
\end{proof}

\subsection{Crepant localization for ordinary double points}
\label{ss:ccc-odp}

From now on we work over an algebraically closed field~$\kk$ of characteristic not equal to~$2$.

Recall that an isolated singularity~$x \in X$ is an \emph{ordinary
  double point} or a \emph{node} if it is a hypersurface singularity and
the exceptional divisor~$E \subset \Bl_x(X)$ of the blowup is a smooth
quadric with conormal bundle~$\cO_E(-E)$ isomorphic to the hyperplane
line bundle~$\cO_E(1)$.  Since~$E$ is a Cartier divisor, it follows
that~$\Bl_x(X)$ is smooth along~$E$.

\begin{lemma}
\label{lemma:odp-ac}
An ordinary double point of dimension $n \ge 2$ is an acyclic projectively normal singularity.
\end{lemma}

\begin{proof}
Let~$(X,x)$ be an ordinary double point.  Since the question is \'etale local at~$x$ and~$(X,x)$ is a hypersurface singularity, we can assume that~$X$ is a hypersurface in the affine space~$\AA^{n+1}$ with equation~$f = 0$ and~$x \in X$ is the origin.  Then it is easy to see that the exceptional divisor of~$\Bl_x(X)$ is given in~$\P^n$ by the lowest-degree homogeneous component of~$f$.  Since this is a smooth quadric by assumption, we have~$f = f_2 + f_{\ge 3}$, where~$f_2$ is a nondegenerate quadratic form and~$f_{\ge 3}$ is the sum of components of higher degree.  Then, denoting by~$z_i$ coordinates on~$\AA^{n+1}$, it is easy to see that
\begin{equation*}
\fm_{X,x}^m / \fm_{X,x}^{m+1} = \kk[z_0,z_1,\dots,z_n]_m / f_2 \cdot \kk[z_0,z_1,\dots,z_n]_{m-2} \cong \rH^0(E,\cO_E(m)),
\end{equation*}
so~\ref{item:h0e} is satisfied.  Condition~\ref{item:hie} is evident since~$E$ is a quadric of dimension $n - 1 \ge 1$.
\end{proof}

Let~$X$ be a variety of dimension~$n \ge 2$ with ordinary double points, and let~$\pi \colon \tX \to X$ be the blowup of its singular locus.  A combination of Lemma~\ref{lemma:odp-ac} with Corollary~\ref{cor:bo-localization} shows that the functor~\mbox{$\pi_* \colon \Db(\tX) \to \Db(X)$} is a Verdier localization.  However, unless~$n = 2$, the functor~$\pi_*$ is not crepant (see Definition~\ref{def:ccc}).  The goal of this subsection is to construct a crepant categorical resolution of~$X$ in all dimensions~$n \ge 2$.

To state our result we need to recall some properties of spinor bundles on quadrics following \cite{Ottaviani}.  Let~$E$ be a smooth quadric over an algebraically closed field of characteristic not equal to~$2$.  In the case where~\mbox{$\dim(E) = n - 1$} is odd, we let~$\cS$ be the spinor bundle on~$E$, and if~\mbox{$\dim(E) = n - 1$} is even, we let~$\cS$ be one of the two spinor bundles on~$E$ and denote the other by~$\cS'$.  We have
\begin{equation}
\label{eq:rank-spinor}
N \coloneqq \rank(\cS) = 2^{\left\lfloor (n-2)/2 \right\rfloor}.
\end{equation}
For example~$\rank(\cS) = 1$ when~$E$ is $1$-dimensional or $2$-dimensional, in which case~$\cS = \cO(-1)$ on~$E \cong \P^1$ and~$\cS = \cO(-1,0)$ or~$\cO(0,-1)$ on~$E \cong \P^1 \times \P^1$.

All statements which follow are symmetric with respect to the swap of~$\cS$ and~$\cS'$.

The twists and duals of spinor bundles are related by~\cite[Theorem~2.8]{Ottaviani}:
\begin{equation}
\label{eq:spinor-dual}
\begin{cases}
\cS^\vee \cong \cS(1) 
& \text{if~$\dim(E) \equiv 1 \bmod{2}$},\\
\cS^\vee \cong \cS(1),\hphantom{{}'} \; \cS'^\vee \cong \cS'(1)
&  \text{if~$\dim(E) \equiv 0 \bmod{4}$},\\
\cS^\vee \cong \cS'(1), \; \cS'^\vee \cong \cS(1)
&  \text{if~$\dim(E) \equiv 2 \bmod{4}$}.\\
\end{cases}
\end{equation}
Moreover, by~\cite[Theorem 2.8]{Ottaviani} if~$\dim(E)$ is odd, there is an exact sequence
\begin{equation}
\label{eq:spinor-sequence-odd}
0 \lra \cS \lra \cO^{\oplus 2N} \lra \cS(1) \lra 0,
\end{equation}
and if~$\dim(E)$ is even, there are exact sequences
\begin{equation}
\label{eq:spinor-sequence-even}
0 \lra \cS' \lra \cO^{\oplus 2N} \lra \cS(1) \lra 0
\qquad\text{and}\qquad
0 \lra \cS \lra \cO^{\oplus 2N} \lra \cS'(1) \lra 0,
\end{equation}
where~$N$ is defined in~\eqref{eq:rank-spinor}.

By~\cite{Kapranov} spinor bundles are exceptional, completely orthogonal to each other if~$\dim(E)$ is even and can be used to construct a full exceptional collection in~$\Db(E)$.  For our purposes the following collections are convenient:
\begin{equation}
\label{eq:quadric-sod}
\Db(E) =
\begin{cases}
\langle \cO_E(1 - \dim(E)), \dots, \cO_E(-2), \cO_E(-1), \cS, \cO_E \rangle & \text{if~$\dim(E)$ is odd},\\
\langle \cO_E(1 - \dim(E)), \dots, \cO_E(-2), \cS(-1), \cO_E(-1), \cS, \cO_E \rangle & \text{if~$\dim(E)$ is even};
\end{cases}
\end{equation}
see~\cite[Lemma~2.4]{KP21}.  When~$\dim(E)$ is odd, the above collection coincides with the one from~\cite[Section~4]{Kapranov}, while if~$\dim(E)$ is even, Kapranov's collection takes the form
\begin{equation}
\label{eq:quadric-sod-even}
\Db(E) = \langle \cO_E(1 - \dim(E)), \dots, \cO_E(-2), \cO_E(-1), \cS', \cS, \cO_E \rangle,
\end{equation}
and to pass between these two, one can use the mutation given by a twist of~\eqref{eq:spinor-sequence-even}.

For future reference, we also record that if $\dim(E)$ is odd, we have
\begin{equation}
\label{eq:spinor-ext-odd}
\Ext^\bullet(\cS(1),\cS) \cong \kk[-1].
\end{equation}
and if $\dim(E)$ is even, then
\begin{equation}
\label{eq:spinor-ext-even}
\Ext^\bullet(\cS(1),\cS') \cong \Ext^\bullet(\cS'(1),\cS) \cong \kk[-1].
\end{equation}
These {isomorphisms} follow from~\eqref{eq:spinor-sequence-odd} and~\eqref{eq:spinor-sequence-even}, respectively, using exceptionality of {the spinor bundle}~$\cS$ and the vanishing~$\Ext^\bullet(\cO_E, \cS) = 0$ which is a part of semiorthogonality of~\eqref{eq:quadric-sod}.

Now we are ready to state the main result of this subsection.

\begin{theorem}
\label{thm:main-nodal}
Let~$X$ be a variety of dimension~$n \ge 2$ over an algebraically closed field~$\kk$ of characteristic not equal to~$2$ with ordinary double points~$x_1, \dots, x_r$ and no other singularities.  Let
\begin{equation*}
\pi \colon \tX = \Bl_{x_1,\dots,x_r}(X) \lra X
\end{equation*}
be the blowup of all singular points,
let~$\upeta_i \colon E_i \hookrightarrow \tX$ be the embedding of the exceptional divisor over~$x_i$,
and let~$\cS_i$ be a spinor bundle on~$E_i$.
Then the subcategory
\begin{equation}
\label{eq:def-cD}
\cD \coloneqq \left\{ F \in \Db(\tX) \mid \upeta_i^*F \in \langle \cS_i, \cO_{E_i} \rangle \ \text{for each~$1 \le i \le r$} \right\}
\end{equation}
is admissible in~$\Db(\tX)$.
Moreover,
\begin{enumerate}[label={\textup{(\roman*)}}]
\item
\label{item:pis-cd-localization}
the induced functor~$\pi_* \colon \cD \to \Db(X)$ is a crepant Verdier localization;
\item
\label{item:pis-cd-kernel}
the kernel~$\Ker(\pi_*|_\cD)$ is generated by completely orthogonal spherical objects $\rK_1, \dots, \rK_r \in \cD$,
where each~$\rK_i$ is $2$-spherical when~$\dim(X)$ is even and $3$-spherical when~$\dim(X)$ is odd.
\end{enumerate}
\end{theorem}

We note that in the case where~$X$ is a cone over a smooth quadric~$Q$, 
the category~$\cD$ appeared in~\cite{KP23} under the name of the \emph{categorical cone} over~$Q$.

\begin{remark}
Note that the category~$\cD$ is smooth over~$\kk$ and proper over~$X$ because it is an admissible subcategory of~$\Db(\tX)$.
Therefore, by~\cite{K08} it provides a crepant categorical resolution for~$\Db(X)$.
\end{remark}

The proof {of the theorem} takes  most of this subsection.
In what follows we construct the objects~$\rK_i$ explicitly.
To start with, consider the case~$r = 1$.
In this case we denote by~$E$ the only exceptional divisor of~$\pi$,
by~$\upeta \colon E \hookrightarrow \tX$ its embedding,
and by~$\cS$ a spinor bundle chosen on~$E$.
Then by~\cite[Proposition~4.1]{K08} applied to~\eqref{eq:quadric-sod}, 
the category~$\cD$ defined by~\eqref{eq:def-cD} is a part of the semiorthogonal decomposition
\begin{equation}
\label{eq:nodal-sod-tx}
\Db(\tX) =
\begin{cases}
\langle \upeta_*\cO_E(1 - \dim(E)), \dots, \upeta_*\cO_E(-2), \upeta_*\cO_E(-1), \cD \rangle & \text{if~$\dim(E)$ is odd},\\
\langle \upeta_*\cO_E(1 - \dim(E)), \dots, \upeta_*\cO_E(-2), \upeta_*\cO_E(-1), \upeta_*\cS', \cD \rangle & \text{if~$\dim(E)$ is even.}
\end{cases}
\end{equation}
To define the object~$\rK$, consider the natural distinguished triangle
\begin{equation}
\label{eq:eps-eps-cs}
\upeta^*\upeta_*\cS \lra \cS \lra \cS(1)[2]
\end{equation}
(which uses the isomorphism~$\cO_E(-E) \cong \cO_E(1)$).
If~$\dim(E)$ is odd, using adjunction and~\eqref{eq:spinor-ext-odd} we obtain
\begin{equation}\label{eq:ext_eps-cs-odd}
\Ext^\bullet(\upeta_*\cS,\upeta_*\cS) \cong \kk \oplus \kk[-2].
\end{equation}
If $\dim(E)$ is even, we have
\begin{equation}
\label{eq:ext_eps-cs-even}
\Ext^\bullet(\upeta_*\cS,\upeta_*\cS) \cong \kk,
\qquad
\Ext^\bullet(\upeta_*\cS,\upeta_*\cS') \cong 
\Ext^\bullet(\upeta_*\cS',\upeta_*\cS) \cong 
\kk[-2],
\end{equation}
where the first isomorphism follows from~\eqref{eq:nodal-sod-tx} 
as the object~$\upeta_*\cS'$ (and by symmetry also~$\upeta_*\cS$) is exceptional,
and the second and third isomorphisms follow as in the odd-dimensional case 
using~\eqref{eq:spinor-ext-even} instead of~\eqref{eq:spinor-ext-odd}.

Now we define the object~$\rK \in \Db(\tX)$ as follows:
\begin{align}
\label{def:ck-odd}
\rK \coloneqq \upeta_*\cS\qquad
& \text{if~$\dim(E)$ is odd}, 
\\
\label{def:ck-even}
\rK \lra \upeta_*\cS \lra \upeta_*\cS'[2]\qquad
& \text{if~$\dim(E)$ is even},
\end{align}
where in~\eqref{def:ck-even} the second arrow is nontrivial (see~\eqref{eq:ext_eps-cs-even}).

The following lemma provides the key computation for the proof of the theorem.

\begin{lemma}
\label{lemma:ck-spherical}
The object~$\rK$ has the following properties:
\begin{enumerate}[label={\textup{(\roman*)}}]
\item
\label{item:ck-projection}
$\rK$ is the projection of\, $\upeta_*\cS \in \Db(\tX)$ to~$\cD$ with respect to~\eqref{eq:nodal-sod-tx}; in particular, $\rK \in \cD$.
\item
\label{item:ker-pis-cd}
$\Ker(\pi_*|_\cD) = \langle \rK \rangle$.
\item
\label{item:ck-spherical}
The object~$\rK$ is $(p+2)$-spherical, where~$p \in \{0,1\}$ is the parity of\,~$\dim(X)$.
\end{enumerate}
\end{lemma}

\begin{proof}
\ref{item:ck-projection}~
If~$\dim(E)$ is odd, then by the definition~\eqref{eq:def-cD} of~$\cD$, we must check that~$\upeta^*\upeta_*\cS \in \langle \cS, \cO \rangle$.  Using~\eqref{eq:eps-eps-cs} we see that it is enough to check that~$\cS(1) \in \langle \cS, \cO \rangle$, which follows immediately from~\eqref{eq:spinor-sequence-odd}.

If~$\dim(E)$ is even, we use~\eqref{eq:nodal-sod-tx} instead of~\eqref{eq:def-cD}.  Note that~$\rK$ is orthogonal to the exceptional collection
\begin{equation*}
\upeta_*\cO_E(1 - \dim(E)), \dots, \upeta_*\cO_E(-2), \upeta_*\cO_E(-1)
\end{equation*}
(because both $\upeta_*\cS$, $\upeta_*\cS'$ are, the latter by~\eqref{eq:nodal-sod-tx}, and the former by symmetry).  Furthermore, by~\eqref{eq:ext_eps-cs-even} the triangle~\eqref{def:ck-even} realizes $\rK$ as the right mutation of $\upeta_*\cS$ with respect to $\upeta_*\cS'$; therefore,~$\rK$ is orthogonal to~$\upeta_*\cS'$ as well, and hence~\eqref{eq:nodal-sod-tx} implies that~$\rK$ is the projection of~$\upeta_*\cS$ to~$\tcD$.

\ref{item:ker-pis-cd}~
First, note that~$\Ker(\pi_*\vert_\cD) = \Ker(\pi_*) \cap \cD$ and that~\eqref{eq:nodal-sod-tx} implies a semiorthogonal decomposition
\begin{equation*}
\Ker(\pi_*) =
\begin{cases}
\langle \upeta_*\cO_E(1 - \dim(E)), \dots, \upeta_*\cO_E(-2), \upeta_*\cO_E(-1),  \Ker(\pi_*) \cap \cD \rangle & \text{if~$\dim(E)$ is odd},\\
\langle \upeta_*\cO_E(1 - \dim(E)), \dots, \upeta_*\cO_E(-2), \upeta_*\cO_E(-1), \upeta_*\cS',  \Ker(\pi_*) \cap \cD \rangle & \text{if~$\dim(E)$ is even},
\end{cases}
\end{equation*}
just because all components of~\eqref{eq:nodal-sod-tx} except for the last one are contained in~$\Ker(\pi_*)$.  Therefore,
\begin{equation*}
\Ker(\pi_*\vert_\cD) = \pr_\cD(\Ker(\pi_*)),
\end{equation*}
where~$\pr_\cD$ is the projection to~$\cD$ with respect to~\eqref{eq:nodal-sod-tx}.

On the other hand, recall from Corollary~\ref{cor:bo-localization} that~$\Ker(\pi_*)$ is generated by pushforwards to~$\tX$ of objects from~$\cO_E^\perp$.  The latter category has a full exceptional collection induced by the first row in~\eqref{eq:quadric-sod} if~$\dim(E)$ is odd or by~\eqref{eq:quadric-sod-even} if~$\dim(E)$ is even.  Therefore, $\Ker(\pi_*\vert_\cD)$ is generated by the projections to~$\cD$ of the objects~$\upeta_*\cO_E(1-\dim(E))$, \dots, $\upeta_*\cO_E(-1)$, $\upeta_*\cS$ (and additionally of~$\upeta_*\cS'$ if~$\dim(E)$ is even).  But all these objects project to zero except for~$\upeta_*\cS$, which projects to~$\rK$ by part~\ref{item:ck-projection}.  Thus, $\Ker(\pi_*\vert_\cD) = \langle \rK \rangle$.

\ref{item:ck-spherical}~
The description of~$\Ext^\bullet(\rK, \rK)$ is given by~\eqref{eq:ext_eps-cs-odd} in the odd case, and in the even case it can be computed as follows.  It follows from~\ref{item:ck-projection} that~$\Ext^\bullet(\rK,\upeta_*\cS') = 0$, so that applying the functor~$\Ext^\bullet(\rK,-)$ to~\eqref{def:ck-even}, we get~$\Ext^\bullet(\rK, \rK) \cong \Ext^\bullet(\rK, \upeta_*\cS)$, and the latter group is computed easily using~\eqref{def:ck-even} and~\eqref{eq:ext_eps-cs-even}.

To check that~$\rK$ is spherical, it remains to show that~$\bS_{\cD}(\rK) \cong \rK[p + 2]$, where recall~$p \in \{0,1\}$ is the parity of~$\dim(X)$ and~$\bS_{\cD}$ is the Serre functor of~$\cD$.  Let us write~$d = \dim(E) {{}= \dim(X) - 1}$.

We first consider the case where~$d$ is odd, so~$p = 0$.  Let~$\omega_\tX$ be the canonical line bundle of~$\tX$.  Since by adjunction formula $\omega_{\tX}|_E \cong \cO(1-d)$, the Serre functor~$\bS_\tX$ of~$\tX$ acts on~$\rK = \upeta_*\cS$ as
\begin{equation*}
\bS_\tX(\upeta_*\cS) \cong \upeta_*(\cS(1-d))[d+1].
\end{equation*}
On the other hand, by~\eqref{eq:nodal-sod-tx} the Serre functor~$\bS_\cD$ is equal to the composition
\begin{equation*}
\bS_\cD \cong \bR_{\langle \upeta_*\cO_E(1 - d), \dots, \upeta_*\cO_E(-2), \upeta_*\cO_E(-1) \rangle} \circ \bS_\tX\vert_\cD,
\end{equation*}
where~$\bR$ stands for the right mutation functor.  Twisting the sequence~\eqref{eq:spinor-sequence-odd} by~$\cO_E(-i)$ with $1 \le i \le d - 1$ and pushing it forward to~$\tX$, we obtain a chain of morphisms
\begin{equation*}
\upeta_*\cS[2] \lra \upeta_*\cS(-1)[3] \lra \cdots \lra \upeta_*(\cS(1 - d))[d+1]
\end{equation*}
with cones isomorphic to shifts of~$\upeta_*\cO_E(-1)^{\oplus 2N}$, \dots, $\upeta_*\cO_E(1-d)^{\oplus 2N}$, so the cone of the composition is contained in the subcategory~$\langle \upeta_*\cO_E(1 - d), \dots, \upeta_*\cO_E(-2), \upeta_*\cO_E(-1) \rangle = \cD^\perp \subset \Db(\tX)$.  Since, on the other hand,~$\upeta_*\cS \in \cD$ as we proved in~\ref{item:ck-projection}, we conclude that~$\upeta_*\cS[2]$ is the right mutation of~$\upeta_*(\cS(1 - d))[d+1]$ through~$\cD^\perp$. Thus, we have 
\begin{equation*}
\bS_\cD(\upeta_*\cS) \cong \upeta_*\cS[2],
\end{equation*}
and~$\rK \cong \upeta_*\cS$ is 2-spherical.

In the case where~$d$ is even, {so that~$p = 1$,} we can refer to~\cite[Proposition~3.15(iii)]{K21}.  Alternatively, we can argue in a way similar  to that in the odd case.  Let~$\pr_\cD$ be the projection functor to~$\cD$ with respect to the second decomposition in~\eqref{eq:nodal-sod-tx}, so that applying~$\bS_\cD \cong \pr_\cD \circ\, \bS_\tX\vert_\cD$ to~\eqref{def:ck-even}, we obtain the triangle
\begin{equation}\label{eq:Serre-cK-even}
\bS_{\cD}(\rK) \lra \pr_\cD (\upeta_* \cS(1-d))[d+1] \lra \pr_\cD (\upeta_* \cS'(1-d))[d+3].
\end{equation}
As before, we have $\pr_\cD = \bR_{\upeta_* \cS'} \circ \bR_{\langle \upeta_*\cO_E(1 - d), \dots, \upeta_*\cO_E(-2), \upeta_*\cO_E(-1) \rangle}$.  Using twists of~\eqref{eq:spinor-sequence-even} repeatedly as in the odd case to project to $\langle \upeta_*\cO_E(1 - d), \dots, \upeta_*\cO_E(-2), \upeta_*\cO_E(-1) \rangle^\perp$, we obtain
\begin{align*}
\pr_\cD (\upeta_* \cS\hphantom{{}'}(1-d))[d+1] &\cong \bR_{\upeta_* \cS'} (\upeta_* \cS'[2]) = 0,\\
\pr_\cD (\upeta_* \cS'(1-d))[d+3] &\cong \bR_{\upeta_* \cS'} (\upeta_* \cS\hphantom{{}'}[4]) \cong \rK[4],
\end{align*}
where in the last isomorphism we used~\ref{item:ck-projection}.  Thus, \eqref{eq:Serre-cK-even} implies that~$\bS_{\cD}(\rK) \cong \rK[3]$.
\end{proof}

Now we are ready to prove the theorem.

\begin{proof}[Proof of Theorem~\textup{\ref{thm:main-nodal}}]
Recall that~$\pi \colon \tX = \Bl_{x_1,\dots,x_r}(X) \to X$ is the blowup of a variety with~$r$ ordinary double points~$x_i$ (in particular, $\tX$ is smooth), $\upeta_i \colon E_i \to \tX$, $1 \le i \le r$, is the embedding of the exceptional divisor over the point~$x_i$, $\cS_i$ is a spinor bundle on~$E_i$, and the category~$\cD$ is defined by~\eqref{eq:def-cD}.  By~\cite[Proposition 4.1]{K08}, the category~$\cD$ is admissible in~$\Db(\tX)$ and comes with a semiorthogonal decomposition analogous to~\eqref{eq:nodal-sod-tx}, where the left orthogonal to~$\cD$ is generated by the exceptional objects~$\upeta_{i*}\cO_{E_i}(-s)$ with~\mbox{$1 \le s \le \dim(E_i) - 1$} (and {also} by~$\upeta_{i*}\cS'_i$ if~$\dim(E_i)$ is even) and~$1 \le i \le r$.  Since all these objects are contained in~$\Ker(\pi_*)$, it follows from Corollary~\ref{cor:bo-localization} and Proposition~\ref{prop:sod-crepancy} that~$\pi_*\vert_\cD \colon \cD \to \Db(X)$ is a Verdier localization.  Furthermore, the arguments of Lemma~\ref{lemma:ck-spherical} show that
\begin{equation*}
\Ker(\pi_*\vert_\cD) = \langle \rK_1, \dots, \rK_r \rangle,
\end{equation*}
where~$\rK_i = \pr_\cD(\upeta_{i*}\cS_i)$ is a~${(p+2)}$-spherical object, and as each of these objects is defined by the isomorphism~\eqref{def:ck-odd} or the triangle~\eqref{def:ck-even} for the corresponding divisor~$E_i$, their supports are disjoint, so they are completely orthogonal.  Finally, we have
\begin{equation*}
\Ker(\pi_*\vert_\cD)^\perp = 
\bigcap_{i=1}^r \rK_i^\perp = 
\bigcap_{i=1}^r {}^\perp(\bS_\cD\rK_i) = 
\bigcap_{i=1}^r {}^\perp \rK_i = 
{}^\perp \Ker(\pi_*\vert_\cD),
\end{equation*}
so the localization~$\pi_*\vert_\cD$ is crepant by Definition~\ref{def:ccc}.
\end{proof}

If~$\dim(X) = 2$, decomposition~\eqref{eq:nodal-sod-tx} shows that the category~$\cD$ from Theorem~\ref{thm:main-nodal} coincides with~$\Db(\tX)$, and if~$\dim(X) = 3$, we prove below that~$\cD$ is equivalent to the derived category of a small resolution of~$X$.  In higher dimensions~$\cD$ should be considered as a categorical version of a small resolution.

Recall that if~$x \in X$ is a 3-dimensional node and~$\pi \colon \tX \to X$ is the blowup of~$x$, the exceptional divisor of~$\pi$ is~$E \cong \P^1 \times \P^1$, and its normal bundle is~$\cN_{E/\tX} \cong \cO_E(-1,-1)$.  The two contractions~$E \to \P^1$ induce two factorizations of the blowup morphism that fit into a commutative diagram of algebraic spaces
\begin{equation*}
\vcenter{\xymatrix@=10pt{
& 
\tX \ar[dl]_{\sigma_-} \ar[dr]^{\sigma_+} \ar[dd]|-{\ \mathstrut\pi\ }
\\
\hX_- \ar[dr]_{\varpi_-} &&
\hX_+ \ar[dl]^{\varpi_+}
\\
&
X\rlap{.}
}}
\end{equation*}
The algebraic spaces~$\hX_\pm$ are smooth (they are known as the small resolutions of~$X$), the exceptional loci of the maps~$\varpi_\pm \colon \hX_\pm \to X$ are smooth rational curves~$C_\pm \subset \hX_\pm$, the maps~$\sigma_\pm \colon \tX \to \hX_\pm$ are the blowups of these curves, and the birational isomorphism~$\sigma_+ \circ \sigma_-^{-1} = \varpi_+^{-1} \circ \varpi_- \colon \hX_- \dashrightarrow \hX_+$ is the Atiyah flop.

Similarly, if~$X$ is a nodal threefold with nodes~$x_1,\dots,x_r$, one can choose one of the two small resolutions for each of the nodes independently and obtain~$2^r$ small resolutions in the category of algebraic spaces.  Some of these small resolutions may also exist in the category of projective varieties.

\begin{corollary}
\label{cor:resolution-choices}
If\,~$X$ is a nodal threefold, $\pi \colon \tX \to X$ is the blowup of the nodes, and~$\varpi \colon \hX \to X$ is a small resolution by an algebraic space~$\hX$, then for an appropriate choice of spinor bundles on the exceptional divisors of~$\pi$, we have a canonical equivalence~$\cD \simeq \Db(\hX)$ such that~$\pi_*\vert_\cD \cong \varpi_*$.  Under this equivalence the~$3$-spherical object~$\rK_i$ corresponds to~$\cO_{C_i}(-1)$, where~$C_i$ is the exceptional curve in~$\hX$ over~$x_i \in X$.
\end{corollary}

\begin{proof}
To simplify the notation, assume that~$x \in X$ is the only ordinary double point of~$X$.  Then the blowup~$\pi \colon \tX \to X$ of~$x \in X$ admits a factorization through the small resolution
\begin{equation*}
\tX \xrightarrow{\ \sigma\ } \hX \xrightarrow{\ \varpi\ } X,
\end{equation*}
where~$\sigma$ is the blowup of the exceptional curve~$C \subset \hX$.  Let~$E \cong \bP^1 \times \bP^1$ be the exceptional divisor of~$\sigma$ (it coincides with the exceptional divisor of~$\pi$).  Note that~$\sigma\vert_E \colon \bP^1 \times \bP^1 \to \bP^1$ is the projection to the first factor.  Therefore, the blowup formula applied to~$\sigma$ gives
\begin{equation*}
\Db(\tX) = \left\langle \cO_E(-1,-1), \cO_E(0,-1), \sigma^*\left(\Db(\hX)\right)\right \rangle.
\end{equation*}
Since the spinor bundles on~$E$ are precisely the line bundles~$\cS = \cO_E(-1,0)$ and~$\cS' = \cO_E(0,-1)$, comparing the decomposition above with the description of~$\cD$ {given in}~\eqref{eq:nodal-sod-tx}, we obtain the equivalence~\mbox{$\sigma_* \colon \cD \simeq \Db(\hX)$}.

Finally, let us compute~$\sigma_*(\rK)$.  Applying~$\sigma_*$ to~\eqref{def:ck-even} and using the fact that~${\sigma \circ \upeta} \colon \bP^1 \times \bP^1 \to \bP^1$ is the projection~$p_1$ onto the first factor, we obtain the distinguished triangle
\begin{equation*}
\sigma_*(\rK) \lra {p_1}_* \cO_E(-1,0) \lra {p_1}_* \cO_E(0,-1)[2].
\end{equation*}
But we have~$p_{1*} \cO_E(-1,0) \cong{} \cO_C(-1)$ and~$p_{1*} \cO_E(0,-1) = 0$, hence~$\sigma_*(\rK) \cong \cO_C(-1)$.
\end{proof}

If~$\dim(X)$ is odd and~$X$ has~$r$ ordinary double points, the definition of~$\cD$ in~\eqref{eq:def-cD} depends on a choice of one of the two spinor bundles on each of the quadrics~$E_i$, so we have constructed~$2^r$ categorical resolutions.  However, all these resolutions are equivalent and related by \emph{categorical flops}, see~\cite[Proposition~3.15]{K21}, which are analogous to Atiyah flops relating~$2^r$ small resolutions in the $3$-dimensional case.  On the other hand, if~$\dim(X)$ is even, the constructed categorical resolution~$\cD$ is canonical; \textit{i.e.}, it does not depend on any choice.


\section{Absorption of singularities for nodal varieties}
\label{sec:nodal-absorption}

In this section we combine the results obtained in the previous sections and apply them for nodal varieties.  In~Section~\ref{ss:db-nodal} we construct (under appropriate assumptions) an absorption of singularities for a nodal variety~$X$, in~Section~\ref{ss:obstructions} we discuss an obstruction to the existence of absorption of singularities by categorical ordinary double points, and in~Section~\ref{ss:examples} we demonstrate how our approach works for nodal curves and threefolds.

We keep working over an algebraically closed field~$\kk$ of characteristic not equal to~$2$.

\subsection{Absorption for nodal varieties}
\label{ss:db-nodal}

First, we apply the construction of Theorem~\ref{thm:contractions} to the crepant Verdier localization of Theorem~\ref{thm:main-nodal}.  Recall that for an exceptional divisor~$E$ of the blowup of an ordinary double point, we write~$\cO_E(1)$ for the conormal bundle~$\cO_E(-E)$ and denote by~$\cS$ a spinor bundle on~$E$; we also write~$\cS'$ for the other spinor bundle if~$\dim(E)$ is even, or for the same spinor bundle if~$\dim(E)$ is odd.

\begin{theorem}
\label{thm:main-nodal-adherence}
Let~$X$ be a proper variety of dimension~$n \ge 2$ with ordinary double points~$x_1, \dots, x_r$ and no other singularities.  Let~$\pi \colon \tX = \Bl_{x_1,\dots,x_r}(X) \to X$ be the blowup, and let~$\upeta_i \colon E_i \hookrightarrow \tX$ be the embedding of the exceptional divisor over~$x_i$.  Let~$p \in \{0,1\}$ be the parity of~$n$.

Assume that there exists a \textup(nonfull\,\textup) exceptional collection~$\cE_1, \dots, \cE_r \in \Db(\tX)$ such that
\begin{equation}
\label{eq:cei-restricted-to-ej}
\upeta_j^*(\cE_i) \cong
\begin{cases}
\cS_j\text{ or }\ \cS'_j(1) & \text{if\,~$i = j$},\\
V_{i,j} \otimes \cO_{E_j} & \text{if\, $i \ne j$},
\end{cases}
\end{equation}
for~$V_{i,j} \in \Db(\kk)$ and some choice of spinor bundles~$\cS_j$ on~$E_j$.  Then the following statements hold: 
\begin{enumerate}[label={\textup{(\roman*)}}]
\item
\label{item:rpi-cci-in-dbx}
Each~$\rP_i \coloneqq \pi_*(\cE_i) \in \Db(X)$ is a $\Pinfty{p+1}$-object, and each~$\cP_i \coloneqq \langle \rP_i \rangle \subset \Db(X)$ is an admissible subcategory, equivalent to {the} categorical ordinary double point of degree~$p$.
\item
\label{item:nodal-absorption}
The collection~$\cP_1,\dots,\cP_r$ is semiorthogonal, the subcategory~$\cP \coloneqq \langle \cP_1, \dots, \cP_r\rangle \subset \Db(X)$ absorbs singularities of~$X$, and the categories~${}^\perp\cP$ and~$\cP^\perp$ are equivalent to admissible subcategories in~$\Db(\tX)$.

\item
\label{item:nodal-deformation-absorption}
If~$n$ is odd and~$X$ projective, $\cP$ provides a universal deformation absorption of singularities for~$X$.
\end{enumerate}
\end{theorem}

Note that~\eqref{eq:cei-restricted-to-ej} implies that~$\cE_i$ must be locally free of rank~$2^{\left\lfloor (n-2)/2 \right\rfloor}$ in a neighbourhood of the divisor~$E_i$.

\begin{proof}
Consider the crepant categorical resolutions~$\pi_* \colon \cD \to \Db(X)$ constructed in Theorem~\ref{thm:main-nodal} and associated with the choice~$\cS_j$ of spinor bundles on~$E_j$ determined by~\eqref{eq:cei-restricted-to-ej}.  Let~$\rK_1, \dots, \rK_r \in \cD$ be the corresponding (completely orthogonal) sequence of~$(p + 2)$-spherical objects.

To start with, note that the assumptions~\eqref{eq:cei-restricted-to-ej} together with~\eqref{eq:spinor-sequence-odd} and~\eqref{eq:spinor-sequence-even} imply that
\begin{equation*}
\upeta_j^*(\cE_i) \in \langle \cS_j, \cO_{E_j} \rangle
\end{equation*}
for all~$i,j$.  Comparing with the definition~\eqref{eq:def-cD} of~$\cD$, we see that~$\cE_i \in \cD$ for each~$1 \le i \le r$.  Furthermore, if~$\pr_\cD$ denotes the projection functor to~$\cD$ {with respect to~\eqref{eq:nodal-sod-tx} (or its obvious analogue if~$r > 1$)}, it follows from Lemma~\ref{lemma:ck-spherical} that
\begin{equation*}
\Ext^\bullet_{\cD}\left(\cE_i, \rK_j\right) \cong
\Ext^\bullet_{\cD}\left(\cE_i, \pr_\cD\left(\upeta_{j*}\cS_j\right)\right) \cong
\Ext^\bullet_{\tX}\left(\cE_i, \upeta_{j*}\cS_j\right) \cong
\Ext^\bullet_{E_j}\left(\upeta_j^*(\cE_i), \cS_j\right).
\end{equation*}
Using~\eqref{eq:cei-restricted-to-ej} together with~\eqref{eq:spinor-ext-odd}, \eqref{eq:spinor-ext-even}, and the exceptionality of~$\cS_j$, we see that the graded space~$\Ext^\bullet(\cE_i, \rK_i)$ is 1-dimensional, and using the semiorthogonality of the pair~$(\cS_j,\cO_{E_j})$ (see~\eqref{eq:quadric-sod}), we see that~$\Ext^\bullet(\cE_i, {\rK_j}) = 0$ {for~$j \ne i$}.  Thus, the adherence assumption~\eqref{eq:dim-ext-ce-ck} is satisfied.  Also note that~$\Db(X) = \cD / \langle \rK_1, \dots, \rK_r \rangle$ by Theorem~\ref{thm:main-nodal}.  Therefore, Theorem~\ref{thm:contractions}\ref{item:cc} applies to this situation and implies parts~\ref{item:rpi-cci-in-dbx} and~\ref{item:nodal-absorption}.  Similarly, part~\ref{item:nodal-deformation-absorption} follows from Corollary~\ref{cor:deformation-absorption-by-codp}.
\end{proof}

We observe the following nice homological property of the $\P^\infty$-objects~$\rP_i$ constructed in Theorem~\ref{thm:main-nodal-adherence}.  Recall from~\cite[Definition~4.2.1]{Buchweitz} that a coherent sheaf~$\cF$ on a Gorenstein variety~$X$ is called \emph{maximal Cohen--Macaulay} if~$\cExt^i(\cF, \cO_X) = 0$ for~$i > 0$.

\begin{proposition}\label{prop:MCM}
Under the assumptions of Theorem~\textup{\ref{thm:main-nodal-adherence}}, assume that~$\cE_i$ is a locally free sheaf on~$\tX$.  Then each~$\P^\infty$-object~$\rP_i = \pi_*(\cE_i)$ on~$X$ is a maximal Cohen--Macaulay sheaf locally free on~$X \setminus \{x_i\}$.
\end{proposition}

To prove Proposition~\ref{prop:MCM}, we use a simple criterion for the pushforward of a sheaf~$\cF$ on~$\tX$ to be a sheaf.

\begin{lemma}\label{lem:pi-acyclic}
Let~$X$ be an $n$-dimensional variety with an ordinary double point~$x \in X$, and let~$\upeta \colon E \hookrightarrow \tX$ be the exceptional divisor of the blowup~$\pi \colon \tX = \Bl_x(X) \to X$.  If a sheaf\,~$\cF$ on~$\tX$ has the property
\begin{equation*}
\rH^{>0}(E, \upeta^*\cF(m)) = 0
\qquad
\text{for $m \ge 0$},
\end{equation*}
then $\rR^{>0}\pi_*(\cF) = 0$.  In particular, this holds when $\upeta^*\cF \cong \cS(a)$ with~$a \ge 2-n$.
\end{lemma}

\begin{proof}
Consider the exact sequence
\begin{equation*}
0 \lra \cF(-(m+1)E) \lra \cF(-mE) \lra \upeta_*(\upeta^*\cF(m)) \lra 0.
\end{equation*}
Since~$\rH^{>0}(E, \upeta^*\cF(m)) = 0$ for~$m \ge 0$ by assumption, it follows that the sheaf~$\rR^i\pi_*\cF(-(m+1)E)$ surjects onto~$\rR^i\pi_*\cF(-mE)$ for all~$i \ge 1$ and~$m \ge 0$.  On the other hand, $\rR^{>0}\pi_*(\cF(-mE)) = 0$ for~$m \gg 0$ because~$-E$ is relatively ample for~$\pi$.  Therefore, $\rR^{>0}\pi_*(\cF(-mE)) = 0$ for all~$m \ge 0$.

For the second part we just note that~$\rH^{>0}(E,\cS(m)) = 0$ for~$m \ge 2 - n$ by~\cite[Theorem~2.3]{Ottaviani} and Serre duality, so the first part applies.
\end{proof}

\begin{proof}[Proof of Proposition~\textup{\ref{prop:MCM}}]
The maximal Cohen--Macaulay property of~$\pi_*\cE$ is local with respect to~$X$, so we may assume that~$X$ has a single ordinary double point.  We will also assume that~$\upeta^*\cE \cong \cS$ (the case where~$\upeta^*\cE \cong \cS'(1)$ is similar).

First, we apply Lemma~\ref{lem:pi-acyclic} to~$\cF = \cE$ and~$a = 0$; since~$n \ge 2$ it implies that~$\rP \coloneqq \pi_*\cE$ is a sheaf.  Next, taking into account that~$\omega_{\tX/X} \cong \cO_\tX((n-2)E)$ and using the Grothendieck duality, we obtain
\begin{equation*}
\cRHom\left(\pi_*\cE, \cO_X\right) \cong
\pi_*\cRHom\left(\cE, \pi^!\cO_X\right) \cong
\pi_*\left(\cE^\vee \otimes \omega_{\tX/X}\right) \cong
\pi_*\left(\cE^\vee((n-2)E)\right).
\end{equation*}
It remains to note that~$\upeta^*(\cE^\vee((n-2)E)) \cong \cS^\vee(2-n)$, and by~\eqref{eq:spinor-dual} this is isomorphic either to~\mbox{$\cS(3-n)$} or to~\mbox{$\cS'(3-n)$}; in both cases using Lemma~\ref{lem:pi-acyclic} with $a = 3-n$, we conclude that~$\pi_*(\cE^\vee((n-2)E))$ is a sheaf.  Therefore,~$\pi_*\cE$ is maximal Cohen--Macaulay.

The fact that~$\rP_i$ is locally free on~$X \setminus \{x_i\}$ follows from the triviality of~$\cE_i$ on the exceptional divisors~$E_j$ for~$i \ne j$; see, \textit{e.g.}, the argument of~\cite[Lemma~2.5]{KKS20}.  Indeed, this claim is local around the points~$x_j$ for~$j \ne i$, so we may again assume that~$X$ has a single ordinary double point~$x$ and~$\cE$ is trivial on the exceptional divisor~$E$ of its blowup~$\tX = \Bl_x(X)$.  Then, on the level of the unbounded from below categories, we have a semiorthogonal decomposition
\begin{equation*}
\bD^-(\tX) = \langle \Ker(\pi_*), \pi^*(\bD^-(X)) \rangle,
\end{equation*}
and since~$\Ker(\pi_*) \subset \bD^-(\tX)$ is generated by~$\upeta_*(\cO_E^\perp)$ (where the orthogonal is taken in~$\bD^-(E)$), the triviality of~$\upeta^*\cE$ implies that~$\cE \cong \pi^*\cF$ for~$\cF \in \bD^-(X)$ (alternatively, the same result follows from~\cite[Proposition~5.5(ii)]{KS:hfd}).  Finally, using the isomorphisms~$\Ext^\bullet(\cE, \cO_{\tilde{x}}) \cong \Ext^\bullet(\cF, \cO_x)$ for~$\tilde{x} \in E$, it is easy to deduce that~$\cF$ is locally free at~$x$ if~$\cE$ is locally free along~$E$, and it remains to note that~$\cF \cong \pi_*\cE$.
\end{proof}

The following proposition shows an example of an application of Theorem~\ref{thm:main-nodal-adherence}.

\begin{proposition}
\label{prop:quadric}
Let~$X \subset \P^{n+1}$ be a nodal $n$-dimensional projective quadric with node~$x \in X$, \textit{i.e.}, a cone over a smooth quadric~$Q^{n-1}$ of dimension~$n-1$.  Assume~$n \ge 2$, and let~$\cS$ be a spinor bundle on~$Q^{n-1}$.  If~$\uprho_0 \colon X \setminus \{x\} \to Q^{n-1}$ is the natural projection, the sheaf~$\uprho_0^*(\cS)$ has a unique maximal Cohen--Macaulay extension~$\rP$ to~$X$.  The sheaf\,~$\rP$ is a~$\Pinfty{p+1}$-object on~$X$, where as usual~$p \in \{0,1\}$ is the parity of~$n$, it absorbs singularities of\,~$X$, and if~$n$ is odd, this is a universal deformation absorption.
\end{proposition}

\begin{proof}
Let~$\pi \colon \tX \to X$ be the blowup of~$x$.
Then 
\begin{equation*}
\tX \cong \P_{Q^{n-1}}(\cO \oplus \cO(-1)).
\end{equation*}
Note that the exceptional divisor~$E$ of the blowup is the section of the projection~$\uprho \colon \tX \to Q^{n-1}$ corresponding to the summand~$\cO$ in the bundle~$\cO \oplus \cO(-1)$ above.  Since~$\uprho$ is a $\P^1$-bundle, the functor~$\uprho^*$ is fully faithful, so the bundle
\begin{equation*}
\cE \coloneqq \uprho^*(\cS)
\end{equation*}
is exceptional and satisfies~\eqref{eq:cei-restricted-to-ej}.  Therefore, Theorem~\ref{thm:main-nodal-adherence} implies that~$\rP \coloneqq \pi_*(\cE)$ is a~$\Pinfty{p+1}$-object on~$X$ which absorbs singularities of~$X$, and when~$n$ is odd, this is a universal deformation absorption.  Further, Proposition~\ref{prop:MCM} shows that~$\rP$ is maximal Cohen--Macaulay, and since~$\pi$ induces an isomorphism~\mbox{$\tX \setminus E \cong X \setminus \{x\}$} compatible with the projections~$\uprho$ and~$\uprho_0$, it follows that~$\rP\vert_{X \setminus \{x\}} \cong \uprho_0^*(\cS)$.  On the other hand, since maximal Cohen--Macaulay sheaves are reflexive by~\cite[Lemma~4.2.2(iii)]{Buchweitz}, the sheaf~$\rP$ is isomorphic to the pushforward of the sheaf~$\uprho_0^*(\cS)$ along the inclusion~$X \setminus \{x\} \hookrightarrow X$.  Thus, $\rP$ is a maximal Cohen--Macaulay extension of~$\uprho_0^*(\cS)$, and its unicity is obvious.
\end{proof}

The maximal Cohen--Macaulay extension~$\rP$ of~$\uprho_0^*(\cS)$ is known as a \emph{spinor sheaf} on~$X$.

\begin{remark}
\label{rem:quadric-hpd}
Another approach to constructing the same absorption of singularities for~$X$ relies on homological projective duality.  Indeed, consider~$X$ as a singular hyperplane section~$X = Y \cap H$ of a smooth quadric~$Y \subset \P^{n + 2}$.  Then by~\cite[Theorem~1.1]{KP21} the homological projective dual~$Y^\natural$ of~$Y$ is
\begin{itemize}
\item 
either the projectively dual quadric of~$Y$, if~$n$ is odd, 
\item 
or the double covering of~$\check{\P}^{n+2}$ ramified over the projectively dual quadric of~$Y$, if~$n$ is even.
\end{itemize}
Therefore, applying the main theorem of homological projective duality~\cite[Theorem~6.3]{K07}, we obtain a semiorthogonal decomposition of~$\Db(X)$, where one of the components is equivalent to the (derived) fiber~$(Y^\natural)_H$ of~$Y^\natural$ over the point~$[H] \in \check{\P}^{n+2}$ of the dual projective space.  The above description of~$Y^\natural$ implies that~$\Db((Y^\natural)_H) \simeq \Db(\sA_p)$, where recall that the right-hand side is a categorical ordinary point of degree~$p$, where~$p \in \{0,1\}$ is the parity of~$n$.
\end{remark}

\subsection{Obstructions}
\label{ss:obstructions}

If the category~$\Db(X)$ is indecomposable (\textit{e.g.}, if~$X$ is a Calabi--Yau variety or, more generally, if~$X$ is a projective Cohen--Macaulay variety and the dualizing sheaf~$\omega_X$ has small base locus, see~\cite{Spence}, \cite{LopesMartin-deSalas}, or~\cite[Corollary~6.7]{KS:hfd} for details), then~$X$ admits no nontrivial absorption.  This shows that the existence of nontrivial absorption is a global condition on~$X$, not a local condition around singularities.  In this subsection we discuss obstructions to absorption of singularities by categorical ordinary double points.  To state our obstruction in the most general form, we recall the following definition.

Let~$\cT$ be a triangulated category with finite-dimensional $\Hom$-spaces.  Then the triangulated \emph{singularity category}~$\cT_{\sing}$ is defined in~\cite[Definition~1.7]{Orl06} (see also~\cite[Remark~4.9]{KS:hfd}) as the Verdier localization
\begin{equation*}
\cT_\sing \coloneqq \cT / \cT_\hf,
\end{equation*}
where~$\cT_\hf \subset \cT$ is the subcategory of left homologically finite-dimensional objects.  Note that for this notion to behave well, it is better to assume that~$\cT$ is hfd-closed in the sense of~\cite[Definition~4.1]{KS:hfd}.  On the other hand, any admissible subcategory of~$\Db(X)$ is hfd-closed if~$X$ is projective over a perfect field; see~\cite[Propositions~6.1(ii) and~4.6(iii)]{KS:hfd}.

\begin{lemma}
\label{lem:tsg}
Let~$\cT$ be a triangulated category with finite-dimensional $\Hom$-spaces.
\begin{enumerate}[label={\textup{(\roman*)}}]
\item 
\label{it:tsg-additive}
If\,~$\cT = \langle \cT_1,\dots,\cT_m \rangle$ is a semiorthogonal decomposition with admissible components, then
\begin{equation*}
\cT_\sing = \left\langle (\cT_1)_\sing,\dots, (\cT_m)_\sing \right\rangle.
\end{equation*}
\item 
\label{it:tsg-proper}
If\,~$\cT$ is proper, then~$\cT_\sing = 0$.
\end{enumerate}
\end{lemma}

\begin{proof}
Part~\ref{it:tsg-additive} is~\cite[Proposition~1.10]{Orl06}, and part~\ref{it:tsg-proper} is obvious because~$\cT_\hf = \cT$ if~$\cT$ is proper.
\end{proof}

Using this and Proposition~\ref{prop:quadric}, we can state the general obstruction.  In the case of varieties of dimension at most~$3$, it will be later reformulated as an explicit numerical condition (see Proposition~\ref{prop:obstruction123}).

\begin{proposition}
\label{prop:cdpsg-idcomplete}
If~$p \in \{0,1\}$, the singularity category~$\Db(\sA_p)_\sing$ of the categorical ordinary double point~$\Db(\sA_p)$ is idempotent complete.
\end{proposition}

\begin{proof}
By Proposition~\ref{prop:quadric} the categorical ordinary double point~$\Db(\sA_p)$ absorbs singularities of a nodal quadric~$X$ of dimension~$n = 2k + p$ for any~$k \ge 1$; \textit{i.e.}, there is a semiorthogonal decomposition
\begin{equation*}
\Db(X) = \left\langle \cT, \Db\left(\sA_p\right) \right\rangle,
\end{equation*}
where~$\cT$ is smooth and proper.  Applying Lemma~\ref{lem:tsg} we conclude that
\begin{equation*}
\Db\left(\sA_p\right)_\sing \simeq \Db(X)_\sing,
\end{equation*}
so it is enough to check that~$\Db(X)_\sing$ is idempotent complete.  Indeed, by Kn\"orrer periodicity (\textit{cf.} \cite[Theorem~2.1]{Orl04}) we may assume~$k = 0$, and then in the case~$p = 0$, the category~$\Db(X)_\sing$ is described explicitly in~\cite[Section~3.3]{Orl04} (as additive category it is equivalent to the category of vector spaces), and its idempotent completeness is obvious.  In the case~$p = 1$, idempotent completeness is proved in~\cite[Lemma~2.20 and Proposition~3.1]{Kalck-Pavic-Shinder}.
\end{proof}

\begin{remark}
\label{rem:Ap-sg}
In fact, the same result holds for any~$p \ge 0$; indeed, one can identify the category~$\Db(\sA_p)_\sing$ with the additive category of~$\ZZ/(p+1)$-graded finite-dimensional vector spaces (where the shift functor acts as the shift of grading), see~\cite[Section~7.1]{Keller-orbits}, so it is obviously idempotent complete.
\end{remark}

\begin{corollary}
\label{cor:absorption-criterion}
Let~$X$ be a projective variety.  If there exists a semiorthogonal collection of admissible subcategories~$\cP_1, \dots, \cP_r \subset \Db(X)$, each equivalent to a categorical ordinary double point and such that the category~$\cP \coloneqq \langle \cP_1, \dots, \cP_r \rangle$ absorbs singularities of~$X$, then the category~$\Db(X)_\sing$ is idempotent complete.
\end{corollary}

\begin{proof}
Since the category~$\cP^\perp$ is smooth and proper, Lemma~\ref{lem:tsg} gives a semiorthogonal decomposition
\begin{equation*}
\Db(X)_\sing = \left\langle (\cP_1)_\sing, \dots, (\cP_r)_\sing \right\rangle.
\end{equation*}
By Proposition~\ref{prop:cdpsg-idcomplete} and Remark~\ref{rem:Ap-sg}, every component in the right-hand side is idempotent complete, so~$\Db(X)_\sing$ is idempotent complete as well; see~\cite[Lemma~2.2]{Kalck-Pavic-Shinder}.
\end{proof}

In the rest of this subsection we will make the criterion of Corollary~\ref{cor:absorption-criterion} explicit for nodal varieties of small dimension.  In dimension~$3$ we use the notion of maximal nonfactoriality, \textit{cf.}~\cite{Kalck-Pavic-Shinder}, which we state in terms of the blowup~$\pi\colon \tX \to X$ of the singular locus with exceptional divisors~$E_1, \dots, E_r$.  Note that for each~$i$ we have~$E_i \cong \P^1 \times \P^1$ and
\begin{equation}
\label{eq:normal-3}
\cO_{E_i}(E_i) \cong \cO_{\P^1 \times \P^1}(-1,-1),
\end{equation} 
hence~$\Pic(E_i)/[\cO_{E_i}(E_i)] \cong \ZZ$, and the two choices of isomorphism correspond to two contractions~\mbox{$E_i \to \P^1$}.

\begin{definition}[\textit{cf.}~\cite{Kalck-Pavic-Shinder}]
\label{def:mnf}
A nodal
threefold~$X$ is \emph{maximally nonfactorial}
(respectively, \emph{{$\QQ$}-maximally nonfactorial})
if the morphism
\begin{equation}
\label{eq:class-groups}
\Pic(\tX) \lra
\bigoplus_{i = 1}^r \Pic(E_i) \lra{}
\bigoplus_{i = 1}^r (\Pic(E_i)/[\cO_{E_i}(E_i)]) \cong \Z^r
\end{equation}
is surjective (respectively, has finite cokernel). 
\end{definition}

\begin{remark}
The map~\eqref{eq:class-groups} factors through the quotient~$\Cl(X) \cong \Pic(\tX) / (\oplus_{i=1}^r \Z \cdot [E_i])$, and the induced homomorphism from~$\Cl(X)$ to the right-hand side of~\eqref{eq:class-groups} can be identified with the restriction map to the sum of the class groups of the completions of~$X$ at the singular points.  Thus, Definition~\ref{def:mnf} is equivalent to that of~\cite{Kalck-Pavic-Shinder}.
\end{remark}

We denote by~$\Br(X)$ the Brauer group of~$X$.

\begin{proposition}
\label{prop:obstruction123}
Let~$X$ be a nodal projective variety.  If there exists a semiorthogonal collection of admissible subcategories~$\cP_1, \dots, \cP_r \subset \Db(X)$, each equivalent to a categorical ordinary double point and such that the category~$\cP \coloneqq \langle \cP_1, \dots, \cP_r \rangle$ absorbs singularities of~$X$, then the following hold: 
\begin{renumerate}
\item
If\,~$\dim(X) = 1$, the dual graph of\,~$X$ is a tree.
\item
If\,~$\dim(X) = 2$, then $\Br(X) = 0$.
If in addition $X$ is a rational surface, then $\Cl(X)$ is torsion-free.
\item
If\,~$\dim(X) = 3$, $X$ is maximally nonfactorial.
\end{renumerate}
\end{proposition}

\begin{proof}
We use~\cite[Corollary~3.3]{Kalck-Pavic-Shinder} in the case~$\dim(X) = 1$, \cite[Proposition~3.7]{Kalck-Pavic-Shinder} and~\cite[Proposition~4.4]{KKS20} in the case~$\dim(X) = 2$, and~\cite[Corollary~3.8]{Kalck-Pavic-Shinder} in the case~$\dim(X) = 3$.
\end{proof}

We conclude this subsection with a discussion of the geometric meaning of maximal nonfactoriality.  Recall from Section~\ref{ss:ccc-odp} a description of small resolutions of a nodal threefold.  The following Proposition~\ref{prop:mnf} shows in particular that $\QQ$-maximal nonfactoriality implies that all~$2^r$ small resolutions of~$X$ are projective.

\begin{proposition}
\label{prop:mnf}
Let~$X$ be a nodal projective threefold with~$r$ nodes.  Consider the following properties:
\begin{aenumerate}
\item 
\label{item:mnf-km1}
$\Db(X)_\sing$ is idempotent complete.
\item 
\label{item:mnf-mnf}
$X$ is maximally nonfactorial.
\item 
\label{item:mnf-qmnf}
$X$ is $\QQ$-maximally nonfactorial.
\item 
\label{item:mnf-all-sr}
$X$ admits $2^r$ projective small resolutions.
\item 
\label{item:mnf-sr}
$X$ admits a projective small resolution.
\item 
\label{item:mnf-nf}
$X$ is nonfactorial.
\end{aenumerate}
We have the following implications:
\ref{item:mnf-km1} $\iff$ 
\ref{item:mnf-mnf} $\implies$ 
\ref{item:mnf-qmnf} $\iff$ 
\ref{item:mnf-all-sr} $\implies$ 
\ref{item:mnf-sr} $\implies$ 
\ref{item:mnf-nf}.
If, moreover, $X$ has a single node, \textit{i.e.}, $r = 1$, then
\ref{item:mnf-qmnf} $\iff$ \ref{item:mnf-all-sr} $\iff$ \ref{item:mnf-sr} $\iff$ \ref{item:mnf-nf}.
\end{proposition}

We do not know if the equivalence~\ref{item:mnf-mnf} $\iff$~\ref{item:mnf-qmnf} always holds for nodal threefolds.  However, in~\cite[Proposition~A.14]{KS-II} we prove the equivalence~\ref{item:mnf-mnf} $\iff$~\ref{item:mnf-qmnf} for complex nodal Fano threefolds.

\begin{proof}
The equivalence \ref{item:mnf-km1} $\iff$ \ref{item:mnf-mnf} is~\cite[Corollary~3.8]{Kalck-Pavic-Shinder}.  The implication \ref{item:mnf-mnf} $\implies$ \ref{item:mnf-qmnf} is trivial.

For the equivalence~\ref{item:mnf-qmnf} $\iff$ \ref{item:mnf-all-sr} we consider the blowup~$\tX$ of~$X$ at all the nodes and note that each exceptional divisor~$E_i \cong \P^1 \times \P^1$ has two $\P^1$-bundle structures~$p_{i,\pm} \colon E_i \to \P^1$, and by~\eqref{eq:normal-3} the conormal bundle~$\cO_{E_i}(-E_i)$ is relatively ample for both of them.  We will use the following criterion: if~$I \subset \{1,\dots,r\}$ is a subset, then there is a contraction~$\sigma_I \colon \tX \to \hX$ over~$X$ to a smooth projective variety~$\hX$ such that
\begin{equation*}
\sigma_I\vert_{E_i} = 
\begin{cases}
p_{i,+} & \text{for~$i \in I$},\\
p_{i,-} & \text{for~$i \not\in I$}
\end{cases}
\end{equation*}
if and only if there is a globally generated line bundle~$\cL$ on~$\tX$ such that
\begin{equation}
\label{eq:cl-ei}
\cL\vert_{E_i} \cong 
\begin{cases}
p_{i,+}^*\cO_{\P^1}(d_i),\quad d_i > 0, & \text{for~$i \in I$},\\
p_{i,-}^*\cO_{\P^1}(d_i),\quad d_i > 0, & \text{for~$i \not\in I$}
\end{cases}
\end{equation}
(more precisely, the ``if'' part is proved in~\cite[Theorem~1]{Ishii}, while the ``only if'' part is obvious).

Now, assume that~$X$ is $\QQ$-maximally nonfactorial.  Since the cokernel of~\eqref{eq:class-groups} is finite, an appropriate positive multiple of any vector is in the image of~\eqref{eq:class-groups}, so for each subset~$I$ we can find a line bundle~$\cL$ satisfying~\eqref{eq:cl-ei}, and twisting it by the pullback of a sufficiently ample line bundle on~$X$, we can assume that it is globally generated.  Thus, $X$ admits~$2^r$ projective small resolutions.  This proves~\ref{item:mnf-qmnf} $\implies$ \ref{item:mnf-all-sr}.

Conversely, assume that $X$ is not $\QQ$-maximally nonfactorial, \textit{i.e.}, the map~\eqref{eq:class-groups} has infinite cokernel.  Then there is a nonzero linear function~$\upsilon \colon \bigoplus_{i = 1}^r (\Pic(E_i)/[\cO_{E_i}(E_i)]) \to \ZZ$ such that the image of~\eqref{eq:class-groups} is contained in~$\Ker\upsilon$.  Let~$I_\pm$ be the sets of~$i \in \{1,\dots,r\}$ such that~$\upsilon([p_{i,\pm}^*\cO_{\P^1}(1)]) > 0$.  By~\eqref{eq:normal-3} the elements~$[p_{i,\pm}^*\cO_{\P^1}(1)]$ are opposite in~$\Pic(E_i)/[\cO_{E_i}(E_i)]$, so the sets~$I_+$ and~$I_-$ are disjoint; moreover, since~$\upsilon \ne 0$, at least one of these sets is nonempty.  Now consider the small resolution~$\hX \to X$ such that the restriction of the factorization map~$\sigma \colon \tX \to \hX$ to~$E_i$ coincides with~$p_{i,\pm}$ if~$i \in I_\pm$; in other words,~$\sigma = \sigma_I$ as above with~$I_+ \subset I$ and~$I \cap I_- = \varnothing$.  If~$\hX$ is projective, the above criterion shows the existence of a line bundle~$\cL$ on~$\tX$ satisfying~\eqref{eq:cl-ei}, and then the value of~$\upsilon$ on the image of~$\cL$ is positive.  Indeed, if~$I_+ \ne \varnothing$, then the positivity follows from~$I_+ \subset I$, and if~$I_- \ne \varnothing$, the positivity follows from~$I \cap I_- = \varnothing$.  This contradiction proves~\ref{item:mnf-all-sr} $\implies$ \ref{item:mnf-qmnf}.

The implication \ref{item:mnf-all-sr} $\implies$ \ref{item:mnf-sr} is trivial.  Finally, the implication \ref{item:mnf-sr} $\implies$ \ref{item:mnf-nf} follows from the fact that factorial nodal threefolds admit no projective small resolutions (indeed, if~$\hX \to X$ is a small projective resolution, then~$\Pic(X) \subsetneq \Pic(\hX) = \Cl(\hX) = \Cl(X)$, so~$X$ is not factorial).

Now assume that~$X$ has a single node.  It suffices to show \ref{item:mnf-nf}~$\implies$~\ref{item:mnf-qmnf}.  This implication is equivalent to the statement that the map in~\eqref{eq:class-groups} is nonzero if and only it has finite cokernel, which is obvious because its target is~$\ZZ$.
\end{proof}

Finally, we state a criterion for a nodal threefold~$X$ to be maximally nonfactorial in terms of a small resolution by an algebraic space.

\begin{lemma}
\label{lem:mnf-resol}
Let~$\tX \xrightarrow{\ \sigma\ } \hX \xrightarrow{\ \varpi\ } X$ be a factorization of the blowup~$\pi$ through a small algebraic space resolution.  Let~$C_i = \sigma(E_i) \cong \P^1$ be the exceptional curves of~$\varpi$.  Then the map~\eqref{eq:class-groups} factors as the composition
\begin{equation*}
\Pic\left(\tX\right) \xrightarrow{\ \sigma_*\ }
\Pic\left(\hX\right) \xrightarrow{\quad}
\bigoplus_{i=1}^r \Pic(C_i) 
\xrightiso{}
\bigoplus_{i = 1}^r (\Pic(E_i)/[\cO_{E_i}(E_i)]), 
\end{equation*}
where the first map is surjective, the second map is given by the restriction, and the last map is induced by the pullbacks~$(\sigma\vert_{E_i})^*$.  In particular, $X$ is maximally nonfactorial if and only if there exist divisor classes~\mbox{$D_1, \dots, D_r \in \Pic(\hX)$} such that
\begin{equation*}
D_i \cdot C_j = \delta_{ij}.
\end{equation*}
\end{lemma}

\begin{proof}
The map~$\sigma \colon \tX \to \hX$ is the blowup of the curves~$C_1,\dots,C_r$ with exceptional divisors~$E_1,\dots,E_r$, so~$\Pic(\tX) = \sigma^*(\Pic(\hX)) \oplus ( \oplus \ZZ[E_i])$.  Thus, it is enough to check that the maps agree on~$\sigma^*(\Pic(\hX))$ and on the~$E_i$.  The first follows from the equalities~\mbox{$\sigma_*(\sigma^*(D)) = D$} (which also proves the surjectivity of~$\sigma_*$) and~$\sigma^*(D)\vert_{E_i} = (\sigma\vert_{E_i})^*(D\vert_{C_i})$, and the second is obvious as~$E_i$ is taken to zero by both maps.
\end{proof}

\subsection{Curves and threefolds}
\label{ss:examples}

In this subsection we collect examples of absorption for nodal curves and threefolds.  We do not discuss the case of nodal surfaces because it essentially reduces to the results obtained in~\cite{KKS20}; see also Example~\ref{ex:surf-node-resol}.

\subsubsection{Curves}
\label{sss:curves}

We start with the case of nodal curves.

\begin{proposition}\label{prop:curve-R}
Assume that $C = C' \cup C''$ is a reducible Gorenstein curve, where~$C' \cong \P^1$ and the scheme intersection~$C' \cap C''$ is a single point~$x$ which is smooth on~$C''$.  Let~$r' \colon C' \to C$ be the embedding.  Then for any line bundle~$\cL'$ on~$C'$ the object
\begin{equation*}
\rP \coloneqq r'_*\cL' \in \Db(C)
\end{equation*}
is a~$\Pinfty{2}$-object.  Moreover, the subcategory~$\cP \coloneqq \langle \rP \rangle \subset \Db(C)$ is admissible, and
\begin{equation}
\label{eq:curve-orthogonals}
{}^\perp\cP \simeq \Db(C'').
\end{equation}
In particular, if~$C''$ is smooth, then~$\cP$ provides a universal deformation absorption for~$C$.
\end{proposition}

\begin{proof}
For any line bundle~$\cL$ on~$C$ we have an exact sequence
\begin{equation*}
0 \lra \cL \lra \cL\vert_{C'} \oplus \cL\vert_{C''} \lra \cL\vert_x \lra 0
\end{equation*}
(obtained by tensor product of~$\cL$ with~$0 \to \cO_C \to \cO_{C'} \oplus \cO_{C''} \to \cO_x \to 0$).  Conversely, for any pair of line bundles~$\cL'$ on~$C'$ and~$\cL''$ on~$C''$ we can define
\begin{equation*}
\cL \coloneqq \Ker(\cL' \oplus \cL'' \lra \cO_x),
\end{equation*}
where the map is given by trivializations of~$\cL'$ and~$\cL''$ at~$x$, and it is easy to see that~$\cL$ is a line bundle.  This shows that restriction of line bundles gives an isomorphism
\begin{equation}
\label{eq:pic-sum}
\Pic(C) = \Pic(C') \oplus \Pic(C'').
\end{equation}
In particular, any line bundle on~$C'$ is obtained by restriction from~$C$, so, twisting by a line bundle on~$C$ if necessary, we may assume~$\cL' = \cO_{C'}(-1)$.

Now consider the line bundle~$\cL_0$ on~$C$ which restricts to~$C'$ as~$\cO_{C'}(-1)$ and to~$C''$ as~$\cO_{C''}$, and the line bundle~$\cL_1$ which restricts to~$C''$ as~$\cO_{C''}(-x)$ and to~$C'$ as~$\cO_{C'}$.  Then we have exact sequences
\begin{equation*}
0 \lra \cO_{C''}(-x) \lra \cL_0 \lra \cO_{C'}(-1) \lra 0
\qquad\text{and}\qquad 
0 \lra \cO_{C'}(-1) \lra \cL_1 \lra \cO_{C''}(-x) \lra 0.
\end{equation*}
Merging these we obtain a long exact sequence
\begin{equation*}
0 \lra \cO_{C'}(-1) \lra
\cL_1 \lra
\cL_0 \lra
\cO_{C'}(-1) \lra 0.
\end{equation*}
This exact sequence can be considered as triangle~\eqref{def:cm} with~$\rP = \cO_{C'}(-1)$, \mbox{$\rM = \Cone(\cL_1 \to \cL_0)$}, and~\mbox{$q = 2$}.  Applying~$\Ext^\bullet(-, \cO_{C'}(-1))$ to the triangle~$\cL_1 \to \cL_0 \to \rM$, we obtain~$\Ext^\bullet(\rM,\rP) = \kk$.  Besides, the condition~$\hocolim \rP[2i] = 0$ follows from Lemma~\ref{lem:hocolim-shift} (see Remark~\ref{rem:pinfty}).  Therefore, $\rP$ is a $\Pinfty{2}$-object, and the subcategory~$\cP \coloneqq \langle \rP \rangle$ is admissible in~$\Db(C)$ by Lemma~\ref{lem:Pinf} as~$C$ is Gorenstein and~\mbox{$\rM \in \Dp(C)$}.

Now consider the contraction~$\sigma\colon C \to C''$ of the component~$C'$ to the point~$x \in C''$.  Since~$x$ is a smooth point on~$C''$, we have a well-defined functor~$\sigma^* \colon \Db(C'') \to \Db(C)$ which is left adjoint to~$\sigma_*$, and since the curve~$C'$ is rational, we have~$\sigma_*(\cO_C) \cong \cO_{C''}$, hence~$\sigma_* \circ \sigma^* \cong \id$.  Thus, $\sigma^*$ is fully faithful, and we have a semiorthogonal decomposition
\begin{equation}\label{eq:DbC}
\Db(C) = \left\langle \Ker(\sigma_*), \sigma^*\left(\Db(C'')\right)\right \rangle.
\end{equation}
A standard argument (see~\cite[Lemma~2.1]{KKS20}) shows that~$\Ker(\sigma_*)$ is generated by~\mbox{$\rP = \cO_{C'}(-1)$}.  This proves~\eqref{eq:curve-orthogonals}.

The last statement follows from Theorem~\ref{thm:intro-pinfty2}.
\end{proof}

\begin{example}[\textit{cf.}~\cite{Kalck-Pavic-Shinder}]
\label{ex:tree}
Let $C$ be a nodal tree of smooth rational curves with~$r + 1$ components.  Then choosing an appropriate ordering for the components of~$C$ (removing rational tails one by one) and using Proposition~\ref{prop:curve-R} inductively, we obtain an admissible semiorthogonal decomposition of~$\Db(C)$ with~$r$ components generated by~$\Pinfty{2}$-objects and one component equivalent to~$\Db(\P^1)$ (hence generated by two exceptional objects).
\end{example}

Note that conversely, by Proposition \ref{prop:obstruction123}, if~$X$ is a nodal curve which admits absorption of its singularities by a semiorthogonal collection of admissible subcategories generated by $\P^\infty$-objects, then the dual graph of~$X$ is a tree.

\subsubsection{Threefolds}
\label{sss:threefolds}

The most interesting case for our applications in the sequel~\cite{KS-II} to this paper is the case of nodal threefolds.  Recall that~$\delta_{ij}$ is the Kronecker delta and~$\bT_{\cO_{C_i}(-1)}$ are the spherical twists.

\begin{theorem}
\label{thm:threefolds}
Let~$X$ be a threefold with ordinary double points~$x_1, \dots, x_r$ and no other singularities, and let~\mbox{$\varpi \colon \hX \to X$} be a small resolution by an algebraic space with exceptional curves $C_1, \dots, C_r$.  Assume that there is a \textup(nonfull\,\textup) exceptional collection~$\cE_1, \dots, \cE_r$ in~$\Db(\hX)$ such that
\begin{equation}
\label{eq:3-cei-restricted-to-ej}
\cE_i|_{C_j} \cong \cO_{C_j}\left(\pm\delta_{ij}\right).
\end{equation}
Then~$\hX$ is a projective variety, and
\begin{enumerate}[label={\textup{(\roman*)}}]
\item
\label{item:3-rpi-cci-in-dbx}
each~$\rP_i \coloneqq \varpi_*(\cE_i) \in \Db(X)$ is a $\Pinfty{2}$-object, and each~$\cP_i \coloneqq \langle \rP_i \rangle \subset \Db(X)$ is an admissible subcategory, equivalent to a categorical ordinary double point;
\item
\label{item:3-nodal-absorption}
the collection~$\cP_1,\dots,\cP_r$ is semiorthogonal, the subcategory~$\cP \coloneqq \langle \cP_1, \dots, \cP_r\rangle \subset \Db(X)$ provides a universal deformation absorption of singularities for~$X$, and the functor~$\varpi_*$ induces equivalences
\begin{align*}
{}^\perp\cP\hphantom{{}^\perp}
&\simeq 
{}^\perp\left\langle \cE_1, \bT_{\cO_{C_1}(-1)}(\cE_1), \dots,
{\cE_r}, \bT_{\cO_{C_r}(-1)}(\cE_r) \right\rangle,
\\
\cP^\perp
&\simeq 
\hphantom{{}^\perp}
\left\langle \cE_1, \bT_{\cO_{C_1}(-1)}(\cE_1), \dots,
{\cE_r}, \bT_{\cO_{C_r}(-1)}(\cE_r) \right\rangle^\perp,
\end{align*}
\end{enumerate}
where the orthogonals on the left-hand side are taken in~$\Db(X)$ and the orthogonals on the right-hand side are taken in~$\Db(\hX)$.
\end{theorem}

\begin{proof}
The projectivity of~$\hX$ follows from Lemma~\ref{lem:mnf-resol} and Proposition~\ref{prop:mnf}.  The remaining part of the theorem is a special case of Theorem~\ref{thm:main-nodal-adherence}; indeed, by Corollary~\ref{cor:resolution-choices} there is an equivalence~\mbox{$\cD \cong \Db(\hX)$} such that the objects~$\rK_i$ correspond to~$\cO_{C_i}(-1)$ and the functor~$\pi_*\vert_\cD$ corresponds to~$\varpi_*$.  Moreover, we have $\cS = \cO_E(-1,0)$, $\cS' = \cO_E(0,-1)$, $\cS'(1) = \cO_E(1,0)$, and it is clear that condition~\eqref{eq:3-cei-restricted-to-ej} on~$\hX$ translates into condition~\eqref{eq:cei-restricted-to-ej} in~$\cD$.  Therefore, Theorems~\ref{thm:main-nodal-adherence} and~\ref{thm:contractions} give the required results.
\end{proof}

Condition~\eqref{eq:3-cei-restricted-to-ej} is subtle in general, but in the case of a single node, it simplifies.

\begin{corollary}\label{cor:mnf-suff}
Let~$X$ be a projective threefold with~$\rH^i(X, \cO_X) = 0$ for~$i > 0$ and a single node.  Assume that~$X$ is maximally nonfactorial.  Then~$\Db(X)$ contains a categorical ordinary double point subcategory~\mbox{$\cP \subset \Db(X)$} which provides a universal deformation absorption of singularities for~$X$.
\end{corollary}

\begin{proof}
Since~$X$ is a maximally nonfactorial threefold, Proposition~\ref{prop:mnf} ensures the existence of a projective small resolution~$\hX \to X$.  Let~$C \subset \hX$ be its exceptional curve.  Then Lemma~\ref{lem:mnf-resol} shows that there is a divisor class~\mbox{$D \in \Pic(\hX)$} such that~$D \cdot C = 1$.  Therefore, Theorem~\ref{thm:threefolds} applied to the exceptional line bundle~$\cE = \cO_{\hX}(D)$ implies the corollary.
\end{proof}

In the rest of this section we discuss del Pezzo threefolds over an algebraically closed field $\kk$ of characteristic zero.  In this case a combination of Proposition~\ref{prop:obstruction123} with~\cite[Theorem~1.1]{Pavic-Shinder-delPezzo} proves that nodal del Pezzo threefolds of degree~$d \le 4$ do not admit absorption by categorical ordinary double points.  Therefore, we concentrate on the case of quintic del Pezzo threefolds.  Similar results in terms of Kawamata decompositions have been obtained by~\cite{Fei-delPezzo} and~\cite{Pavic-Shinder-delPezzo}.

\subsubsection{Quintic del Pezzo threefolds}
\label{sss:v5}

Let~$X$ be a \emph{quintic del Pezzo threefold}, \textit{i.e.}, a complete intersection
\begin{equation*}
X \coloneqq \Gr(2,5) \cap \P^6 \subset \P^9.
\end{equation*}
If~$X$ has isolated singularities, then all of them are nodes, see~\cite[Theorem~2.9]{Fuj86}, the number~$r$ of nodes is~$1 \le r \le 3$, and for each such~$r$ there is a unique (up to isomorphism) quintic del Pezzo threefold~$X$ with~$r$ nodes; see~\cite[Theorem~7.1]{KPr22}.

\begin{proposition}
Let~$X$ be a nodal quintic del Pezzo threefold with~$r$ nodes over an algebraically closed field $\kk$ of characteristic zero.  Then there is a semiorthogonal decomposition
\begin{equation*}
\Db(X) = \langle \rP_1, \dots, \rP_r, \cC \rangle,
\end{equation*}
where the~$\rP_i$ are completely orthogonal~$\Pinfty{2}$-objects which provide a universal deformation absorption of singularities and~$\cC$ is a smooth and proper category.  Moreover, each~$\rP_i$ is a maximal Cohen--Macaulay sheaf on~$X$ locally free on the complement of the corresponding node.
\end{proposition}

\begin{proof}
It is well known (see, \textit{e.g.}, \cite[Theorem~7.1]{KPr22}) that for each~$1 \le r \le 3$ there is a diagram
\begin{equation*}
\xymatrix{
& \hX \ar[dl]_\upsigma \ar[dr]^\varpi
\\
Y &&
X,
}
\end{equation*}
where the following hold: 
\begin{itemize}
\item 
If~$r = 1$, then~$Y = \P^2$, the map~$\upsigma$ is the projectivization of the vector bundle~$\cV$ on~$\P^2$ defined by the exact sequence
\begin{equation*}
0 \lra \cV \xrightarrow{\quad} \Omega(1) \oplus \cO(-1) \oplus \cO(-1) \xrightarrow{\ \varphi\ } \cO^{\oplus 2} \lra 0,
\end{equation*}
and the map~$\varpi$ contracts a smooth rational curve~$C \subset \P(\cV)$ 
which projects isomorphically to a line in~$\P^2$.
In this case we take~$\cE_1 \coloneqq \cO(-H)$, where~$H$ is the pullback of the hyperplane class of~$Y = \P^2$.
\item 
If~$r = 2$, then~$Y \cong \Fl(1,2;3)$, the map~$\upsigma$ is the blowup of a point~$y \in Y$, and the map~$\varpi$ contracts the strict transforms~$C_1, C_2 \subset \hX$ of the fibers of the two projections~$\Fl(1,2;3) \to \P^2$ passing through the point~$y$.  In this case we take~$\cE_i = \cO(-H_i)$, $i = 1,2$, where~$H_i$ is the pullback of the divisor class on~$\Fl(1,2;3)$ inducing the projection~$\Fl(1,2;3) \subset \P^2 \times \P^2 \to \P^2$ to the $\supth{i}$ factor.
\item 
If~$r = 3$, then~$Y \cong \P^3$, the map~$\upsigma$ is the blowup of three points~$y_1,y_2,y_3 \in Y$, and the map~$\varpi$ contracts the strict transforms~$C_1, C_2, C_3 \subset \hX$ of the lines connecting the points (we label the lines in such a way that~$y_i$ does not lie on~$\upsigma(C_i)$).  In this case we take~$\cE_i \coloneqq \cO(H - D_i)$, where~$H$ is the pullback of the hyperplane class of~$Y = \P^3$ and~$D_i$ is the exceptional divisor of~$\upsigma$ over~$y_i$.
\end{itemize}
In all these cases it is easy to see that condition~\eqref{eq:3-cei-restricted-to-ej} of Theorem~\ref{thm:threefolds} is satisfied for any ordering of the singular points, so we obtain the completely orthogonal collection of~$\Pinfty{2}$-objects~$\rP_i \coloneqq \varpi_*(\cE_i)$.  The objects~$\rP_i$ are maximal Cohen--Macaulay sheaves by Proposition \ref{prop:MCM}.
\end{proof}

\begin{remark}
Similarly to Remark~\ref{rem:quadric-hpd}, one can use homological projective duality to obtain another description of the~$\Pinfty{2}$-objects~$\rP_i$.  Indeed, any~$X$ is a linear section of~$\Gr(2,5) \subset \P^9$ by an appropriate linear subspace~$\P^6 \subset \P^9$, and the homological projective dual variety of the Grassmannian is the dual Grassmannian~$\Gr(3,5) \subset \check{\P}^9$; see~\cite[Section~6.1]{K06}.  Thus, the objects~$\rP_i$ correspond to natural generators of the (derived) intersection~$\Gr(3,5) \cap \P^2$ (which consists of~$1$, $2$, or~$3$ reduced points, respectively) by the linear subspace~$\P^2 \subset \check{\P}^9$ orthogonal to~$\P^6 \subset \P^9$.
\end{remark}


\newcommand{\etalchar}[1]{$^{#1}$}

\end{document}